\documentclass[11pt]{article}
\usepackage{graphicx,psfrag}
\usepackage{amssymb,amsfonts,amsthm}
\usepackage{qtree}
\usepackage{amsmath,amsthm,amssymb}
\usepackage{amsfonts}
\usepackage{relsize}
\usepackage[T1]{fontenc}
\usepackage[utf8]{inputenc}
\usepackage{graphicx}
\usepackage{color}
\usepackage{subcaption,mathtools}
\newcommand{\Stab}{\mathrm{Stab}}
\newcommand{\lab}{\mathrm{lab}}

\newcommand{\pl}{\mathrm{PL}}
\newcommand{\Asum}{\mathcal A^{\mathrm{sum}}}
\newcommand{\Asuf}{\mathcal A^{\mathrm{suf}}}
\newcommand{\sump}{\mathrm{sum}_p}
\newcommand{\Psum}{\mathcal P^{\mathrm{sum}}}
\newcommand{\Psuf}{\mathcal P^{\mathrm{suf}}}
\newcommand{\sufp}{\mathrm{suf}_p}
\newcommand{\Fp}{\arr{F}_p}
\newcommand{\CFp}{\mathcal C(\arr{F}_p)}
\newcommand{\C}{\mathcal C}
\newcommand{\Ssum}{S^{\mathrm{sum}}}
\newcommand{\Ssuf}{S^{\mathrm{suf}}}

\newcommand{\la}{\langle}
\newcommand{\ra}{\rangle}

\renewcommand{\P}{\mathcal{P}}

\newcommand{\ddd}{{\mathcal {DG}}}

\newcommand{\be}{\begin{equation}}
\newcommand{\ee}{\end{equation}}

\newcommand {\N}{\mathbb{N}} 

\newcommand {\iv}{^{-1}}
\newcommand {\ab}{\mathrm{ab}}

\newcommand{\Cl}{\mathrm{Cl}}

\numberwithin{equation}{section}

\newcounter{AbcT}

\newcommand{\nc}{\newcommand}
\nc{\meet}{\wedge}
\nc{\op}{\operatorname}\nc{\FP}{\op{FP}}\nc{\FS}{\op{FS}}\nc{\FPhat}{\widehat{\op{FP}}}

\newtheorem {Theorem}    {Theorem}[section]

\newtheorem {Problem}    [Theorem]{Problem}

\newtheorem {Lemma}      [Theorem]    {Lemma}
\newtheorem {Corollary}   [Theorem] {Corollary}
\newtheorem {Proposition}[Theorem]    {Proposition}

\newtheorem {Example}      [Theorem]    {Example}

\theoremstyle{remark}
\newtheorem {Remark}		 [Theorem]    {\bf{Remark}}

\newtheorem {Definition} [Theorem]    {\bf{Definition}}

\newcommand{\1}{\mathbf{1}}
\newcommand{\arr}{\overrightarrow}

\newcommand{\A}{\mathcal A}

\begin{document}

\title{On Maximal Subgroups of Thompson's Group $F$}

\author{Gili Golan Polak\thanks{The research was supported by ISF grant 2322/19.}}
\date{}
\maketitle

\abstract{
	We study subgroups of Thompson's group $F$ by means of an automaton associated with them. We prove that every maximal subgroup of $F$ of infinite index is \emph{closed}, that is, it coincides with the subgroup of $F$ accepted by the automaton associated with it. It follows that every finitely generated maximal subgroup of $F$ is undistorted in $F$. We also prove that every finitely generated  subgroup of $F$ is contained in a finitely generated maximal subgroup of $F$ and construct an infinite family of non-isomorphic maximal subgroups of infinite index in $F$. 
 }

\section{Introduction}

Recall  that $R.$ Thompson's group $F$ is the group of all piecewise-linear homeomorphisms of the interval $[0,1]$, where all breakpoints are dyadic fractions (i.e., elements of the set $\mathbb{Z}[\frac{1}{2}]\cap (0,1)$) and all slopes are integer powers of $2$. The group $F$ is finitely presented, does not contain free non-abelian subgroups and satisfies many other remarkable properties.

In \cite{Sav,Sav1} Savchuk initiated the study of maximal subgroups of Thompson's group $F$ by proving that for every number $\alpha$ in $(0,1)$ the stabilizer of $\alpha$ in $F$ (i.e., the subgroup of $F$ of all functions which fix $\alpha$) is a maximal subgroup of $F$. He asked whether these are all the maximal subgroups of infinite index in $F$ (maximal subgroups of $F$ of finite index are in one-to-one correspondence with maximal subgroups of its abelianization $\mathbb{Z}^2$ and are well understood).

In answer to this problem, in \cite{GS1}, Sapir and the author constructed an explicit maximal subgroup of infinite index in $F$ which does not fix any number in $(0,1)$. Recall that Jones showed that elements of $F$ encode in a natural way all links and knots and that elements of a subgroup of $F$, denoted $\arr{F}$, encode in a natural way all oriented links and knots \cite{Jones} (see also \cite{Jones2,A1}).  The explicit maximal subgroup of $F$  constructed in \cite{GS1} was the image of $\arr{F}$ under an injective endomorphism of $F$.

In addition to the explicit maximal subgroup of $F$ constructed in \cite{GS1}, we also gave a method for proving the existence of many other maximal subgroups of $F$ (but the method did not yield explicit examples). 
Improving on that method, in \cite{G}, the author demonstrated a method for constructing explicit examples of maximal subgroups of $F$. In that paper, $3$ new explicit examples of maximal subgroups of $F$ were constructed. One of the examples served as a strong counterexample to Savchuk's problem. Indeed, one of the maximal subgroups of $F$ constructed in \cite{G} acts transitively on the set of dyadic fractions in the interval $(0,1)$. In \cite{AN}, Aiello and Nagnibeda  constructed $3$ more explicit examples of maximal subgroups of $F$ (relying on the method from \cite{GS1} to a certain extent). 

Note that the constructions of all known maximal subgroups of $F$ of infinite index, other than Savchuk's subgroups and the first explicit example isomorphic to Jones' subgroup, relied  on the construction of the Stallings $2$-core of subgroups of $F$. 

The Stallings $2$-core (or the core, for short) of a subgroup $H$ of $F$ was defined in \cite{GS1} in an analogous way to the Stallings core of a subgroup of a free group. Recall that elements of Thompson's group $F$ can be viewed as diagrams over directed $2$-complexes (see \cite{GuSa97,GuSa99}) or as tree-diagrams (i.e., as pairs of finite binary trees, see Section \ref{sec:tree} below). In \cite{GS1,G} we considered elements of $F$ as diagrams over directed $2$-complexes and defined the core of a subgroup of $F$ in those terms. 
In this paper, we prefer the more standard approach to elements of $F$ in terms of tree-diagrams. In these terms, the core of a subgroup $H$ of $F$, denoted $\C(H)$, can be defined as a rooted tree-automaton (that is, a directed edge-labeled graph with a distinguished ``root'' vertex which satisfies certain properties, see Definition \ref{def:tree-aut}) associated with the subgroup. The core of a subgroup $H$ of $F$ accepts some of the tree-diagrams in $F$ (see Definition \ref{def:accepts} below). 

By construction, the core $\mathcal C(H)$ accepts all tree-diagrams in $H$, but unlike in the case of free groups, the core $\mathcal C(H)$ can accept tree-diagrams not in $H$. We defined the \emph{closure} of $H$ to be the subgroup of $F$ of all tree-diagrams accepted by the core $\mathcal C(H)$. The closure operation satisfies the usual properties of closure. Namely, $H\leq\Cl(H)$, $\Cl(\Cl(H))=\Cl(H)$ and if $H_1\leq H_2$ then $\Cl(H_1)\leq \Cl(H_2)$. We say that a subgroup $H$ of $F$ is \emph{closed} if $H=\Cl(H)$. If $H$ is finitely generated, then its core $\mathcal C(H)$ is a finite automaton and it is decidable whether  
a given tree-diagram in $F$ is accepted by $\mathcal C(H)$. Hence, if $H$ is finitely generated then the membership problem in the closure of $H$ is decidable. Note that if $H$ is finitely generated then its closure $\Cl(H)$ is also finitely generated \cite{GS3}. 
Note also that the closure of subgroups of $F$ can also be described when $F$ is viewed as a group of homeomorphisms of the interval $[0,1]$. 
 Indeed, by \cite[Theorem 5.6]{G}, the closure of a subgroup $H$ of $F$ is the subgroup of $F$ of all piecewise-$H$ functions. In particular, a subgroup $H$ of $F$ is closed if and only if every piecewise-$H$ function in  $F$ belongs to $H$. 

In this paper, we prove the following. 

\begin{Theorem}\label{thm:int1}
	All maximal subgroups of $F$ which have infinite index in $F$ are closed. 
\end{Theorem}

Theorem \ref{thm:int1} answers  \cite[Problem 5.11]{GS1}. Note that Theorem \ref{thm:int1} implies that the membership problem is decidable in every finitely generated maximal subgroup of $F$. In \cite{GS3}, Sapir and the author proved that every finitely generated closed subgroup of $F$ is undistorted in $F$. Hence all finitely generated maximal subgroups of $F$ are undistorted in $F$.

Recall that in \cite{G}, we used the core of subgroups of $F$ to give a solution to the generation problem in $F$ (i.e., to give an algorithm which given a finite set of elements in $F$ determines whether it generates $F$). Indeed, in \cite{G}, we proved the following. 

\begin{Theorem}[{\cite[Corollary 1.4]{G}}]\label{thm:int2}
	Let $H$ be a subgroup of $F$. Then $H=F$ if and only if the following conditions hold. 
	\begin{enumerate}
		\item[$(1)$] $H[F,F]=F$
		\item[$(2)$] $[F,F]\leq \Cl(H)$
		\item[$(3)$] There is a function $h\in H$ which fixes a dyadic fraction $\alpha\in(0,1)$ such that the slope $h'(\alpha^-)=2$ and the slope $h'(\alpha^+)=1$. 
	\end{enumerate}
\end{Theorem}

Given a finite subset $X$ of $F$, we let $H$ be the subgroup generated by $X$. Then it is (easily) decidable if condition (1) holds for $H$ (indeed, one only has to check the image of $H$ in the abelianization $F/[F,F]\cong\mathbb{Z}^2$). Checking if condition (2) holds is also simple and amounts to constructing the core of $H$ (see Lemma \ref{core of derived subgroup} below). In \cite[Section 8]{G}, we gave an algorithm for deciding if $H$ satisfies condition (3), given that $H$ satisfies condition (2). Hence, we got a solution for the generation problem in $F$.

In this paper, we improve the solution.
 Namely, we prove that Condition (3) in Theorem \ref{thm:int2} is superfluous (giving a positive solution to \cite[Problem 5.12]{GS1} and \cite[Problem 12.2]{G}).

\begin{Theorem}\label{thm:int3}
	Let $H$ be a subgroup of $F$. Then $H=F$ if and only if the following conditions hold. 
	\begin{enumerate}
		\item[$(1)$] $H[F,F]=F$
		\item[$(2)$] $[F,F]\leq \Cl(H)$
	\end{enumerate}
\end{Theorem}

In addition to giving a better (linear-time) solution for the generation problem in $F$, Theorem \ref{thm:int3} implies Theorem \ref{thm:int1} (see the proof of Corollary \ref{main cor} below). 
Note that in \cite{GGJ}, Gelander, Juschenko and the author proved that Thomspon's group $F$ is invariably generated by $3$-elements (i.e., there are $3$ elements $f_1,f_2,f_3\in F$ such that regardless of how each one of them is conjugated, together they generate $F$.) The proof relied on the solution of the generation problem in $F$ from \cite{G}. The improved solution implies that in fact Thompson's group $F$ is invariably generated by a set of two elements (see also \cite[Lemma 15]{GS4}).

Theorem \ref{thm:int1} and the study of morphisms of  rooted tree-automata imply further results regarding maximal subgroups of $F$. In \cite{GS1} it was observed that since $F$ is finitely generated, by Zorn's lemma, every proper subgroup of $F$ is contained in some maximal subgroup of $F$ (this observation was used in proving the existence of maximal subgroups of $F$ of infinite index which do not fix any number in $(0,1)$). It was asked (see \cite[Problem 4.6]{GS1}) whether every proper finitely generated subgroup of $F$ is contained inside some finitely generated maximal subgroup of $F$. In Section \ref{sec:max}, we answer this problem affirmatively.

\begin{Theorem}\label{thm:int4}
	Let $H$ be a finitely generated proper subgroup of $F$. Then the
	following assertions hold. 
	\begin{enumerate}
		\item[$(1)$] There exists a finitely generated maximal subgroup $M\leq F$ which contains $H$. 
		\item[$(2)$] If the action of  $H$  on the set of dyadic fractions $\mathcal D$ has finitely many orbits then every maximal subgroup of $F$ which contains $H$ is finitely generated. Moreover, there are only finitely many maximal subgroups of infinite index in $F$ which contain $H$. 
	\end{enumerate}
\end{Theorem}

Recall that for each number $\alpha\in (0,1)$ the stabilizer of $\alpha$ in $F$, denoted $\Stab(\alpha)$, is a maximal subgroup of $F$. It follows that there are uncountably many distinct maximal subgroups of Thompson's group $F$. However, in \cite{GS2}, Sapir and the author proved that the subgroups $\Stab(\alpha)$ for $\alpha\in (0,1)$ fall into three isomorphism classes, depending on the type of $\alpha$ (i.e., on whether $\alpha$ is dyadic, rational non-dyadic or irrational). Hence, until now there were only finitely many known isomorphism classes of maximal subgroups of infinite index in $F$.
 In this paper we prove that there is an infinite  family of pairwise non-isomorphic maximal subgroups of infinite index in $F$. Indeed, in \cite{GS}, we studied a family of subgroups which we called \textit{Jones' subgroups} $\overrightarrow F_n$ (for $n\geq 2$). These subgroups can be defined in an analogous way to Jones' subgroup $\overrightarrow F$, where $\overrightarrow F_2=\overrightarrow F$ (for further details, see \cite[Section 5]{GS})\footnote{There is a subgroup introduced by Jones in \cite{Jones2}
 and studied by Aiello and Nagnibeda in \cite{AN1} that is also denoted $\arr{F}_3$. This subgroup is different from the subgroup $\arr{F}_3$ defined in \cite{GS} as part of the family of subgroups $\arr{F}_n$.}. In this paper we prove that for every prime number $p$, Thompson's group $F$ has a maximal subgroup isomorphic to Jones' subgroup $\arr{F}_p$. 
 

\vskip.2cm

\textbf{Organization:} The paper is organized as follows. In Section \ref{sec:pre} preliminaries about Thompson's group $F$, closed subgroups and the core of subgroups of $F$ are given. In Section \ref{main section}, we improve the solution of the generation problem from \cite{G} and deduce that all maximal subgroups of infinite index in $F$ are closed. In Section \ref{sec:morph} we study  rooted tree-automata and morphisms between them. In Section \ref{sec:max} we derive results about maximal subgroups of $F$, proving Theorem \ref{thm:int4}. In Section \ref{sec:core} we study and recall properties of rooted tree-automata that are isomorphic to the core of a subgroup of $F$. In Section \ref{sec:Jones} we prove that there is an infinite family of non-isomorphic maximal subgroups of infinite index in $F$ and in Section \ref{sec:open} we give some final remarks and discuss some open problems. 
\vskip.2cm

\textbf{Acknowledgments}: The author would like to thank Mark Sapir for helpful conversations. The author would also like to thank the anonymous referee for helpful comments and suggestions.

\section{Preliminaries on $F$}\label{sec:pre}

\subsection{$F$ as a group of homeomorphisms}\label{ss:nf}

Recall that $F$ consists of all piecewise-linear increasing self-homeomorphisms of the unit interval with slopes of all linear pieces powers of $2$ and all break points of the derivative in $\mathbb{Z}[\frac{1}{2}]\cap(0,1)$. The group $F$ is generated by two functions $x_0$ and $x_1$ defined as follows \cite{CFP}.

\[
x_0(t) =
\begin{cases}
2t &  \hbox{ if }  0\le t\le \frac{1}{4} \\
t+\frac14       & \hbox{ if } \frac14\le t\le \frac12 \\
\frac{t}{2}+\frac12       & \hbox{ if } \frac12\le t\le 1
\end{cases} 	\qquad	
x_1(t) =
\begin{cases}
t &  \hbox{ if } 0\le t\le \frac12 \\
2t-\frac12       & \hbox{ if } \frac12\le t\le \frac{5}{8} \\
t+\frac18       & \hbox{ if } \frac{5}{8}\le t\le \frac34 \\
\frac{t}{2}+\frac12       & \hbox{ if } \frac34\le t\le 1 	
\end{cases}
\]

The composition in $F$ is from left to right.

Every element of $F$ is completely determined by how it acts on the set $\mathbb{Z}[\frac{1}{2}]\cap (0,1)$. Every number in $(0,1)$ can be described as $.s$ where $s$ is an infinite word in $\{0,1\}$. For each element $g\in F$ there exists a finite collection of pairs of (finite) words $(u_i,v_i)$ in the alphabet $\{0,1\}$ such that every infinite word in $\{0,1\}$ starts with exactly one of the $u_i$'s. The action of $F$ on a number $.s$ is the following: if $s$ starts with $u_i$, we replace $u_i$ by $v_i$ (the procedure of  associating the pairs of words $(u_iv_i)$ to an element of F is described on page 6). For example, $x_0$ and $x_1$  are the following functions:

\[
x_0(t) =
\begin{cases}
.0\alpha &  \hbox{ if }  t=.00\alpha \\
.10\alpha       & \hbox{ if } t=.01\alpha\\
.11\alpha       & \hbox{ if } t=.1\alpha\
\end{cases} 	\qquad	
x_1(t) =
\begin{cases}
.0\alpha &  \hbox{ if } t=.0\alpha\\
.10\alpha  &   \hbox{ if } t=.100\alpha\\
.110\alpha            &  \hbox{ if } t=.101\alpha\\
.111\alpha                      & \hbox{ if } t=.11\alpha\
\end{cases}
\]
where $\alpha$ is any infinite binary word.
For the generators $x_0,x_1$ defined above, the group $F$ has the following finite presentation \cite{CFP}.
$$F=\la x_0,x_1\mid [x_0x_1^{-1},x_1^{x_0}]=1,[x_0x_1^{-1},x_1^{x_0^2}]=1\ra,$$ where $a^b$ denotes $b\iv ab$.

Sometimes, it is more convenient to consider an infinite presentation of $F$. For $i\ge 1$, let $x_{i+1}=x_0^{-i}x_1x_0^i$. In these generators, the group $F$ has the following presentation \cite{CFP}
$$\la x_i, i\ge 0\mid x_i^{x_j}=x_{i+1} \hbox{ for every}\ j<i\ra.$$

\subsection{Elements of F as pairs of finite binary trees} \label{sec:tree}

Often, it is more convenient to describe elements of $F$ using pairs of finite binary trees drawn on a plane. Trees are considered up to isotopies of the plane. Elements of $F$ are  pairs of full finite binary trees $(T_+,T_-)$ which have the same number of leaves. Such a pair will sometimes be called a \emph{tree-diagram}.

If $T$ is a (finite or infinite) binary tree, a \emph{branch} in $T$ is a maximal simple path starting from the root.
Every vertex of $T$ is either a \emph{leaf} (i.e., a vertex with no outgoing edges) or has exactly two outgoing edges: a left edge and a right edge.
If every left edge of $T$ is labeled by $0$ and every right edge is labeled by $1$, then every branch of $T$ is labeled by a (finite or infinite) binary word $u$.  We will usually ignore the distinction between a branch and its label.

Let $(T_+,T_-)$ be a tree-diagram where $T_+$ and $T_-$ have $n$ leaves. Let $u_1,\dots,u_n$ (resp. $v_1,\dots,v_n$) be the branches of $T_+$ (resp. $T_-$), ordered from left to right.
For each $i=1,\dots,n$ we say that the tree-diagram $(T_+,T_-)$ has the \emph{pair of branches} $u_i\rightarrow v_i$. The function $g$ from $F$ corresponding to this tree-diagram takes binary fraction $.u_i\alpha$ to $.v_i\alpha$ for every $i$ and every infinite binary word $\alpha$. We will also say that the element $g$ takes the branch $u_i$ to the branch $v_i$.
The tree-diagrams of the generators of $F$, $x_0$ and $x_1$, appear in Figure \ref{fig:x0x1}.

\begin{figure}[ht]
	\centering
	\begin{subfigure}{.5\textwidth}
		\centering
		\includegraphics[width=.5\linewidth]{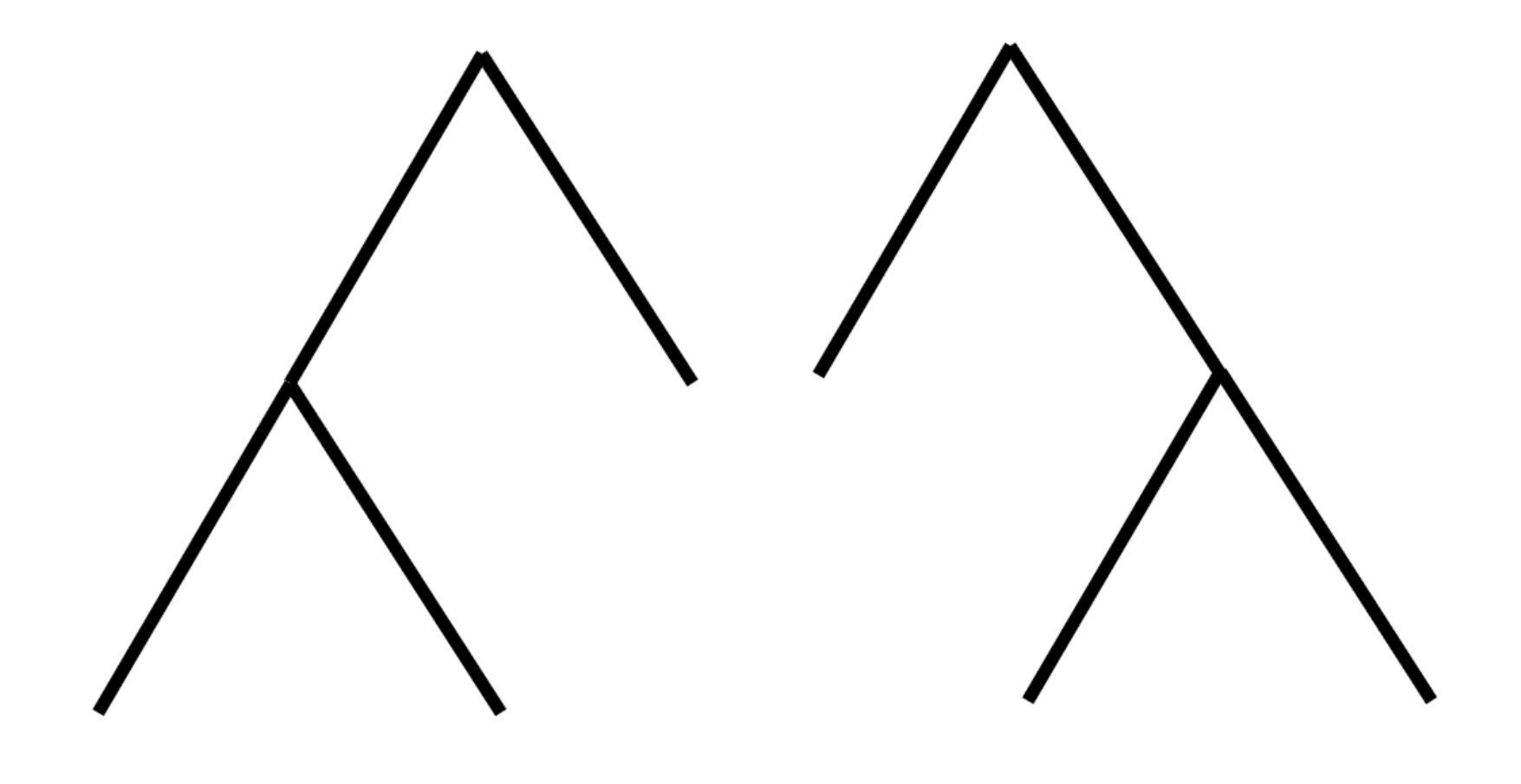}
		\caption{}
		\label{fig:x0}
	\end{subfigure}%
	\begin{subfigure}{.5\textwidth}
		\centering
		\includegraphics[width=.5\linewidth]{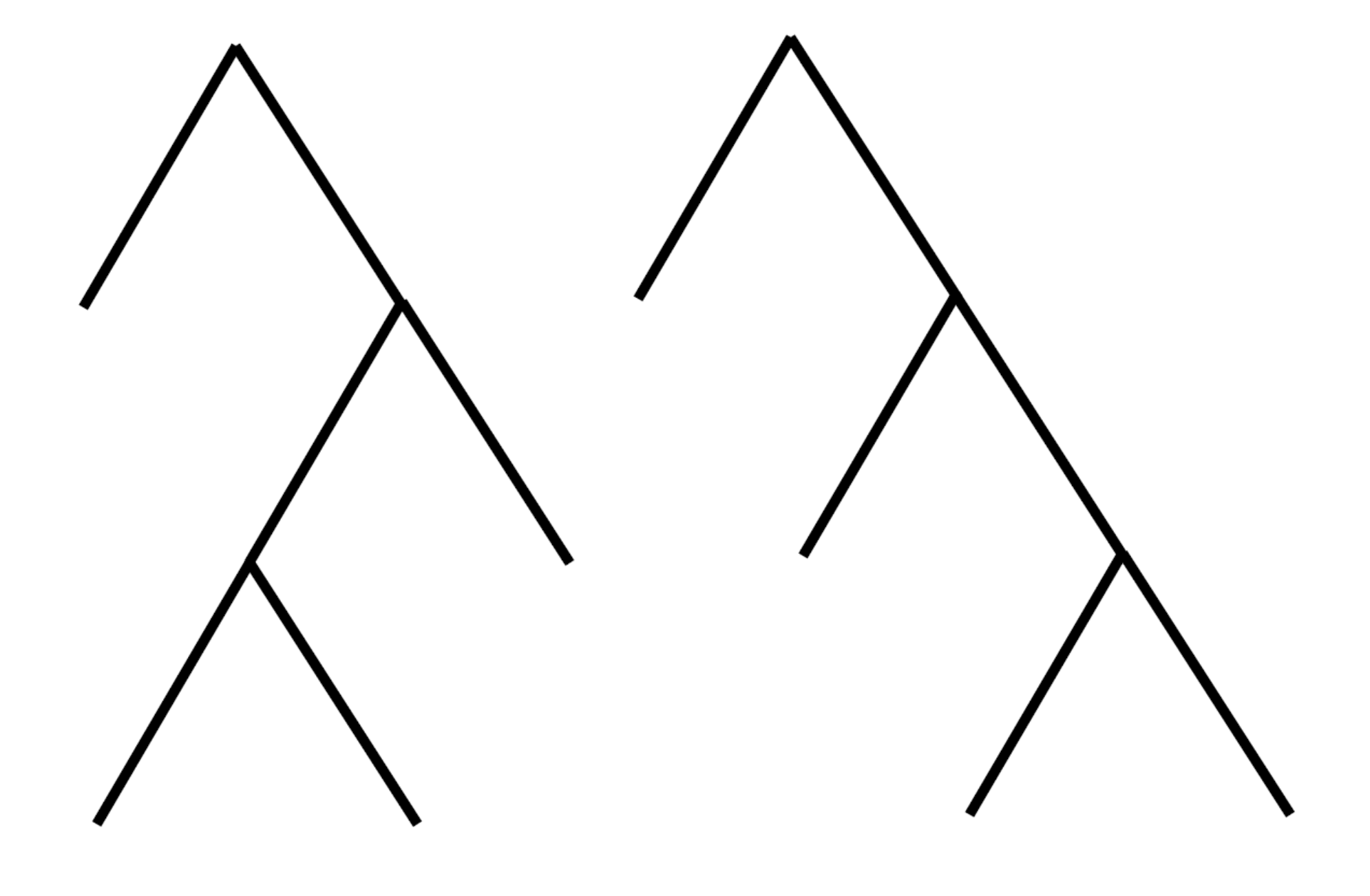}
		\caption{}
		\label{fig:x1}
	\end{subfigure}
	\caption{(a) The tree-diagram of $x_0$. (b) The tree-diagram of $x_1$. In both figures, $T_+$ is on the left and $T_-$ is on the right.}
	\label{fig:x0x1}
\end{figure}

A \emph{caret} is a binary tree which consists of a single vertex with two children. If $(T_+,T_-)$ is a tree-diagram, then attaching a caret to the $i$-th leaf of both $T_+$ and $T_-$ does not affect the function in $F$ represented by the tree-diagram $(T_+,T_-)$. The inverse action of \emph{reducing} common carets does not affect the function either (the pair $(T_+,T_-)$ has a \emph{common caret} if leaves number $i$ and $i+1$ have a common father in $T_+$ as well as in $T_-$).  Two pairs of trees $(T_+,T_-)$ and $(R_+,R_-)$ are said to be \emph{equivalent} if one results from the other by a finite sequence of inserting and reducing common carets. If $(T_+,T_-)$ does not have a common caret then $(T_+,T_-)$ is said to be \emph{reduced}. Every tree-diagram is equivalent to a unique reduced tree-diagram. Thus elements of $F$ can be represented uniquely by reduced tree-diagrams \cite{CFP}.

An alternative way of describing the function in $F$ corresponding to a given tree-diagram is the following. For each finite binary word $u$, we let the \emph{dyadic interval associated with $u$}, denoted by $[u]$, be the interval $[.u,.u1^\N]$. If $(T_+,T_-)$ is a tree-diagram for $f\in F$, we let $u_1,\dots,u_n$ be the branches of $T_+$ and $v_1,\dots,v_n$ be the branches of $T_-$. Then the intervals $[u_1],\dots,[u_n]$ (resp. $[v_1],\dots,[v_n]$) form a subdivision of the interval $[0,1]$. The function $f$ maps each interval $[u_i]$ linearly onto the interval $[v_i]$.

Below, when we say that a function $f$ has a pair of branches $u_i\rightarrow v_i$, the meaning is that some tree-diagram representing $f$ has this pair of branches. In other words, this is equivalent to saying that $f$ maps $[u_i]$ linearly onto $[v_i]$. The following  remark will be useful. 

\begin{Remark}\label{rem:branches}
	Let $f$ be a function in $F$ and assume that $u\to v$ is a pair of branches of $f$. Then, there exists a common (possibly empty) suffix $w$ of both $u$ and $v$ and finite binary words $p$ and $q$ such that $u\equiv pw$\footnote{Throughout this paper, for words $u$ and $v$, $u\equiv v$ denotes letter-by-letter equality.}, $v\equiv qw$ and such that $p\to q$ is a pair of branches of the reduced tree-diagram of $f$. 
\end{Remark}

\begin{Remark}[See \cite{CFP}]\label{r:000}
	The tree-diagram where both trees are just singletons plays the role of identity in $F$. Given a tree-diagram $(T_+^1,T_-^1)$, the inverse tree-diagram is $(T_-^1,T_+^1)$. If $(T_+^2,T_-^2)$ is another  tree-diagram then the product of $(T_+^1,T_-^1)$ and $(T_+^2,T_-^2)$ is defined as follows. There is a minimal finite binary tree $S$ such that $T_-^1$ and $T_+^2$ are rooted subtrees of $S$ (in terms of subdivisions of $[0,1]$, the subdivision corresponding to $S$ is the intersection of the subdivisions corresponding to $T_-^1$ and $T_+^2$). Clearly, $(T_+^1,T_-^1)$ is equivalent to a tree-diagram $(T_+,S)$ for some finite binary tree $T_+$. Similarly, $(T_+^2,T_-^2)$ is equivalent to a tree-diagram $(S,T_-)$. The \emph{product}  $(T_+^1,T_-^1)\cdot(T_+^2,T_-^2)$ is (the reduced tree-diagram equivalent to) $(T_+,T_-)$.
\end{Remark}

Obviously, the mapping of tree-diagrams to functions in $F$ respects the operations defined in Remark \ref{r:000}.

Now, let $\mathcal D$ be the set of dyadic fractions, i.e., the set $\mathbb{Z}[\frac{1}{2}]\cap (0,1)$. We will often be interested in  dyadic fractions $\alpha\in \mathcal D$ fixed by a function $f\in F$. 
More generally, if $S\subset(0,1)$, we say that an element $f\in F$ \emph{fixes} $S$, if it fixes $S$ pointwise. We  say that an element $f\in F$ \emph{stabilizes} $S$ if $f(S)=S$.

The following lemma will be useful. 

\begin{Lemma}[{\cite[Lemma 2.6]{G}}]\label{4parts}
	Let $f\in F$ be an element which fixes some dyadic fraction $\alpha\in \mathcal D$. Let $u\equiv u'1$ be the finite binary word such that $\alpha=.u$. Then the following assertions hold.
	\begin{enumerate}
		\item[$(1)$] $f$ has a pair of branches $u0^{m_1}\rightarrow u0^{m_2}$ for some $m_1,m_2\ge 0$. 
		\item[$(2)$] $f$ has a pair of branches $u'01^{n_1}\rightarrow u'01^{n_2}$ for some $n_1,n_2\ge 0$.
		\item[$(3)$] If $f'(\alpha^+)=2^k$ for $k\neq 0$, then every tree-diagram representing $f$ has a pair of branches $u0^m\rightarrow u0^{m-k}$ for some $m\ge \max\{0,k\}$. 
		\item[(4)] If $f'(\alpha^-)=2^\ell$ for $\ell\neq 0$, then every tree-diagram representing $f$ has a pair of branches $u'01^n\rightarrow u'01^{n-\ell}$ for some $n\ge \max\{0,\ell\}$. 
	\end{enumerate}
\end{Lemma}

\subsection{Natural copies of $F$}\label{sec:copies}

Let $f$ be a function in Thompson group $F$. The \emph{support of $f$}, denoted $\mathrm{Supp}(f)$, is the closure in $[0,1]$ of the subset $\{x\in(0,1):f(x)\neq x\}$. We say that $f$ \emph{is supported in an interval $J$} if the support of $f$ is contained in $J$. Note that in this case the endpoints of $J$ are necessarily fixed by $f$. Hence the set of all functions from $F$  supported in $J$ is a subgroup of $F$. We denote this subgroup by $F_J$.

Thompson group $F$ contains many copies of itself (see \cite{Brin}). 
Let $a$ and $b$ be numbers from $\mathbb{Z}[\frac{1}{2}]$ and consider the subgroup $F_{[a,b]}$.
This subgroup is isomorphic to $F$ (we will refer to such subgroups of $F$ as \emph{natural copies} of $F$).  Indeed, $F$ can be viewed as a subgroup of $\pl_2(\mathbb{R})$ of all piecewise linear homeomorphisms of  $\mathbb R$ with finitely many dyadic break points and absolute values of all slopes powers of 2.  Let $f\in \pl_2(\mathbb{R})$ be a function which maps $0$ to $a$ and $1$ to $b$, (such a function clearly exists). Then $F^f$ is the subgroup of $\pl_2(\mathbb{R})$  of all orientation preserving homeomorphisms with support in $[a,b]$, that is, $F^f=F_{[a,b]}$.

Let $u$ be a finite binary word and let $[u]$ be the dyadic interval associated with it. The isomorphism between $F$ and $F_{[u]}$ can also be defined   using tree-diagrams. Let $g$ be an element of $F$ represented by a tree-diagram $(T_+,T_-)$. We map $g$ to an element in $F_{[u]}$, denoted by $g_{[u]}$ and referred to as the \emph{$[u]$-copy of $g$}.
To construct the element $g_{[u]}$ we start with a minimal finite binary tree $T$ which contains the branch $u$. We take two copies of the tree $T$. To the first copy, we attach the tree $T_+$ at the end of the branch $u$. To the second copy we attach the tree $T_-$ at the end of the branch $u$. The resulting trees are denoted by $R_+$ and $R_-$, respectively. The element $g_{[u]}$ is the one represented by the tree-diagram $(R_+,R_-)$.  Note that if $g$ consists of pairs of branches $v_i\to w_i, i=1,...,k,$ and $B$ is the set of branches of $T$ which are not equal to $u$, then $g_{[u]}$ consists of pairs of branches $uv_i\to uw_i, i=1,...,k$, and $p\to p, p\in B$. Note also that if $(T_+,T_-)$ is the reduced tree-diagram of $g$, then $(R_+,R_-)$ is the reduced tree-diagram of $g_{[u]}$. 

For example, the copies of the generators $x_0,x_1$ of $F$ in $F_{[0]}$ are depicted in Figure \ref{fig:0x0x1}.
It is obvious that these copies generate the subgroup $F_{[0]}$.

\begin{figure}[ht]
	\centering
	\begin{subfigure}{.5\textwidth}
		\centering
		\includegraphics[width=.5\linewidth]{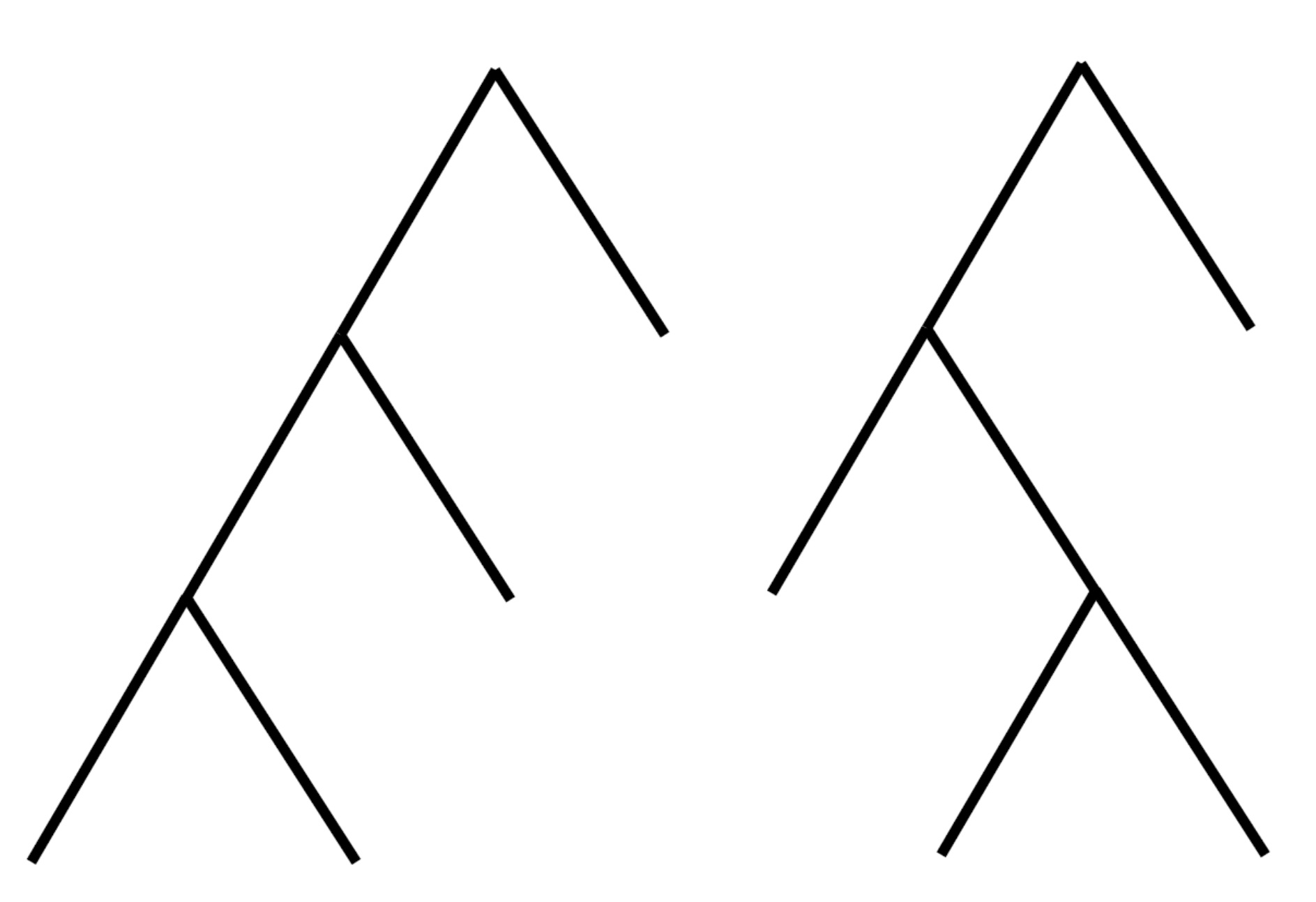}
		\caption{The tree-diagram of $(x_0)_{[0]}$}
		\label{fig:0x0}
	\end{subfigure}%
	\begin{subfigure}{.5\textwidth}
		\centering
		\includegraphics[width=.5\linewidth]{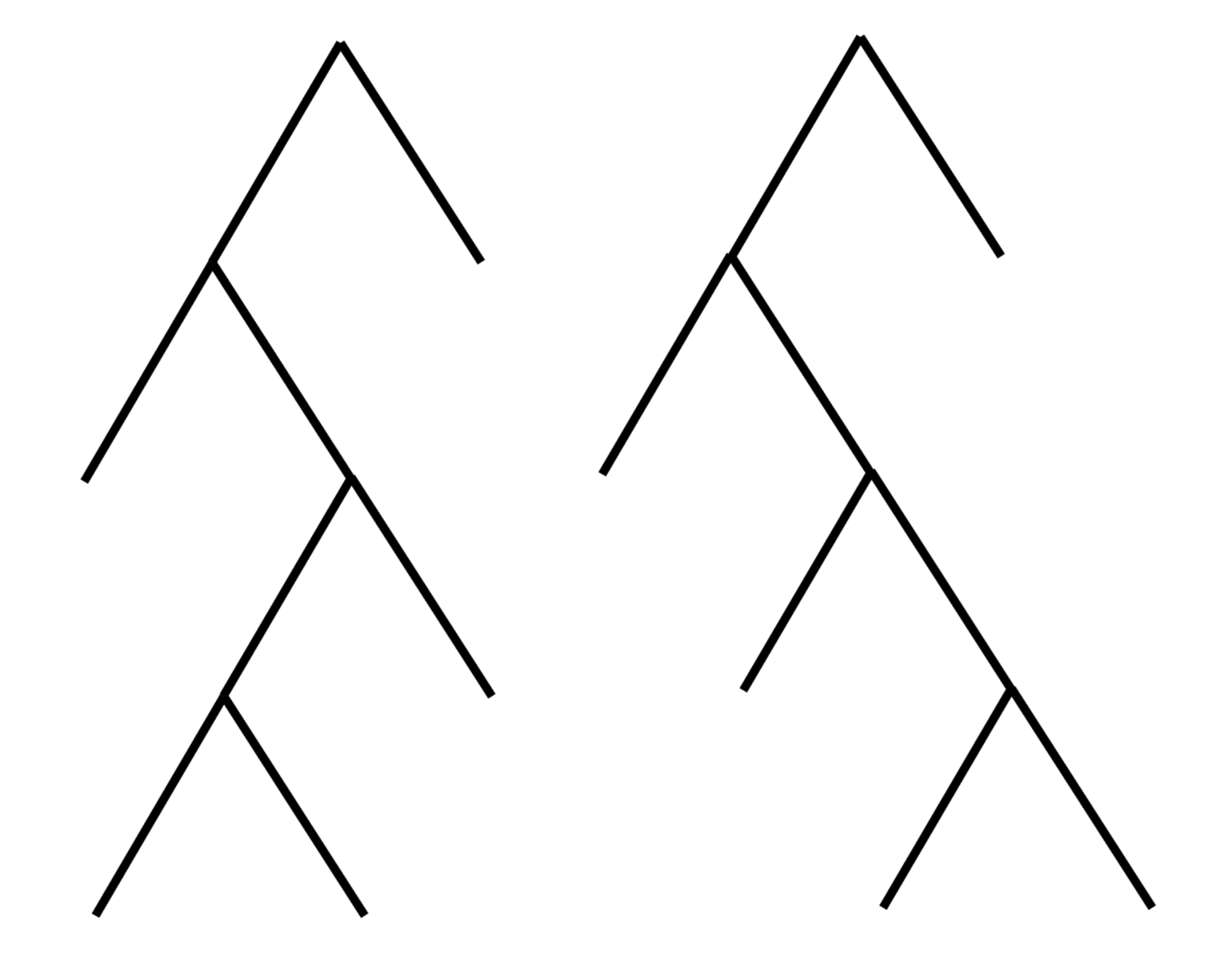}
		\caption{The tree-diagram of $(x_1)_{[0]}$}
		\label{fig:0x1}
	\end{subfigure}
	\caption{}
	\label{fig:0x0x1}
\end{figure}

The isomorphism above guarantees that if $f,g\in F$ then $f_{[u]}g_{[u]}=(fg)_{[u]}$. Given a subset $S$ of $F$ and a finite binary word $u$, we will denote by $S_{[u]}$ the image of $S$ in $F_{[u]}$ under the above isomorphism. Similarly, if $G$ is a subgroup of $F$, we will denote by $G_{[u]}$ the copy of $G$ in $F_{[u]}$ (i.e., the image of $G$ in $F_{[u]}$ under the above isomorphism).

Using this isomorphism, we define an addition operation in Thompson group $F$ as follows. We denote by $\1$ the trivial element in $F$. We define the sum of an element $g\in F$ with the trivial element $\1$, denoted by $g\oplus\1$,
to be the copy of $g$ in $F_{[0]}$. Similarly, the sum of $\1$ and $g$, denoted by $\1\oplus g$, is the copy of $g$ in $F_{[1]}$. If $g,h\in F$ we define the \emph{sum} of $g$ and $h$, denoted by $g\oplus h$, to be the product $(g\oplus \1)(\1\oplus h)$; i.e. $g\oplus h$ is an element from $Stab(\{\frac12\})$ that acts as a copy of $g$ on $[0]$ and as a copy of $h$ on $[1]$. It is easy to see that for $g=\1$ or $h=\1$ this definition coincides with the previous one.
Note that if $f,g\in F$, the slope of $f\oplus g$ at $0^+$ coincides with $f'(0^+)$ and the slope of $g$ at $1^-$ coincides with $g'(1^-)$.

\subsection{Closed subgroups of $F$}

The original definition of closed subgroups of $F$  was given in \cite{GS1} (see also \cite{G}) in the language of diagram groups over directed $2$-complexes. In this section, we adapt the definition (or rather, one of the equivalent definitions from \cite{G}) to the language of tree-diagrams and automata.

In this section, to define closed subgroups of $F$, we define diagram groups over rooted tree-automata (see below). Diagram groups over rooted tree-automata are a special case of the diagram groups studied by Guba and Sapir \cite{GuSa97,GuSa99}.

We choose to give our somewhat narrow definitions in the language of tree-diagrams (as opposed to the language of diagrams used in \cite{GS1,G,GS3}) in the hope that the notions of closed subgroups of $F$ and the core of subgroups of $F$ will be more easily accessible to the wider community of researchers of Thompson group $F$. The terminology of trees would also be convenient in Section \ref{main section}.

Recall that an \emph{automaton} $\mathcal A$ is a directed edge-labeled graph. Every automaton considered in this paper will have a distinguished vertex $r$ called the \emph{initial vertex} or the \emph{root}. We will usually denote such an automaton by $\A_r$ and call it a \emph{rooted automaton}. In this paper, a \emph{path} in a rooted automaton $\mathcal A_r$ is a finite directed path which starts from the root. More formally, if $e$ is a directed edge in a rooted automaton $\A_r$, we denote by $e_-$ the initial vertex of $e$ and by $e_+$ the terminal vertex of $e$. A path in $\A_r$ is a sequence of edges $e_1,\dots, e_n$ such that ${e_1}_-=r$ and for each $i=1,\dots,n-1$, we have ${e_i}_+={e_{i+1}}_-$.

\begin{Definition}\label{def:tree-aut}
	Let $\mathcal A_r$ be a rooted automaton with root $r$. 
	The automaton $\A_r$ is called a \emph{rooted tree-automaton}, or a \textit{tree-automaton} for short, if the following conditions hold. 
	\begin{enumerate}
		\item[$(1)$] Every vertex in $\mathcal A_r$ has either zero or two outgoing edges. 
		\item[$(2)$] If a vertex $x$ in $\A_r$ has two outgoing edges (in which case, we say $x$ is a \emph{father}), then one of the outgoing edges (which we call a \emph{left} edge) is labeled ``0'' and the other one (which we call a \emph{right} edge) is labeled ``1''. The end vertices of these edges are called the left and right \emph{children} of $x$ respectively.
		\item[$(3)$] If  $x_1$ and $x_2$  are distinct fathers in $\mathcal A_r$, then the left children of $x_1$ and $x_2$ are distinct or the right children of $x_1$ and $x_2$ are distinct.
		\item[$(4)$] For every vertex $x$ in $\A_r$, there is a directed path in $\A_r$
		 ending in $x$. 
	\end{enumerate} 
\end{Definition}

A vertex of a tree-automaton $\A_r$ which has no outgoing edges is called a \emph{leaf}. A vertex of $\A_r$ which has two outgoing edges is called an \emph{inner} vertex (or a father vertex).  
Note that if $\A_r$ is a tree-automaton then every  path in $\A_r$ is labeled by a finite binary word $u$. We will rarely distinguish between a path and its label. Note that every finite binary word labels at most one path in $\A_r$. If $u$ is (the label of) a path in $\A_r$, we will denote the end vertex of the path by $u^+$. We say that a finite binary word $u$ is \emph{readable} on $\A_r$ if $u$ labels a path in $\A_r$. 
 We say that a finite binary tree $T$ is \emph{readable} on $\A_r$ if every branch $u$ of $T$ labels a path in $\A_r$.

\begin{Definition}\label{def:accepts}
	Let $\mathcal A_r$ be a tree-automaton. Let $(T_+,T_-)$ be a tree-diagram of an element in $F$.
	\begin{enumerate}
		\item[$(1)$] We say that $(T_+,T_-)$ is \emph{readable} on    $\mathcal A_r$ if both $T_+$ and $T_-$ are readable on $\A_r$. 
		\item[$(2)$] 	We say that $(T_+,T_-)$ is  \emph{accepted} by  $\A_r$ if it is readable on $\A_r$ and for every pair of branches $u\to v$ of $(T_+,T_-)$, we have that $u^+=v^+$ in $\A_r$ (i.e., the end vertices $u^+$ and $v^+$ of the paths $u$ on $v$ in $\A_r$ coincide).	
	\end{enumerate}
\end{Definition}

Note that given a finite  tree-automaton $\A_r$ (i.e., a tree-automaton which has finitely many vertices) and a tree-diagram $(T_+,T_-)$, it is decidable if $(T_+,T_-)$ is accepted by $\A_r$. Indeed, one can check for each pair of branches $u\to v$ of $(T_+,T_-)$ whether $u$ and $v$ label paths in $\A_r$ and if so, whether they terminate on the same vertex of $\A_r$. 

\begin{Example}
	Consider the rooted tree-automaton $\A_r$ given in Figure \ref{fig:Ex}. 
	\begin{enumerate}
		\item[(1)] The reduced tree-diagram $(T_+,T_-)$ of $(x_1)_{[0]}$ (see Figure \ref{fig:0x0x1}(b)) is not readable on $\A_r$. Indeed, the tree $T_-$ is not readable on $\A_r$ since its branch $0111$ does not label a path in $\A_r$.
		\item[(2)] The reduced tree-diagram $(R_+,R_-)$ of $x_1$ (see Figure \ref{fig:x0x1}(b)) is readable on $\A_r$ but is not accepted by $\A_r$. Indeed, since each branch of the trees $R_+$ and $R_-$ labels a path in $\A_r$, the tree-diagram $(R_+,R_-)$ is readable on $\A_r$. Since $101\to 110$ is a pair of branches of $(R_+,R_-)$ and in $\A_r$ we have $(101)^+=k$ whereas $(110)^+=h$, the tree-diagram $(R_+,R_-)$ is not accepted by $\A_r$. 
		\item[(3)] The reduced tree-diagram $(S_+,S_-)$ of $x_0$ (see Figure \ref{fig:x0x1}(a)) is accepted by $\A_r$. Indeed, for each pair of branches $u\to v$ of $(S_+,S_-)$, both $u$ and $v$ label paths in $\A_r$ and $u^+=v^+$ in $\A_r$. 
	\end{enumerate}
\end{Example}

\begin{figure}[ht]
	\centering
	\includegraphics[width=.38\linewidth]{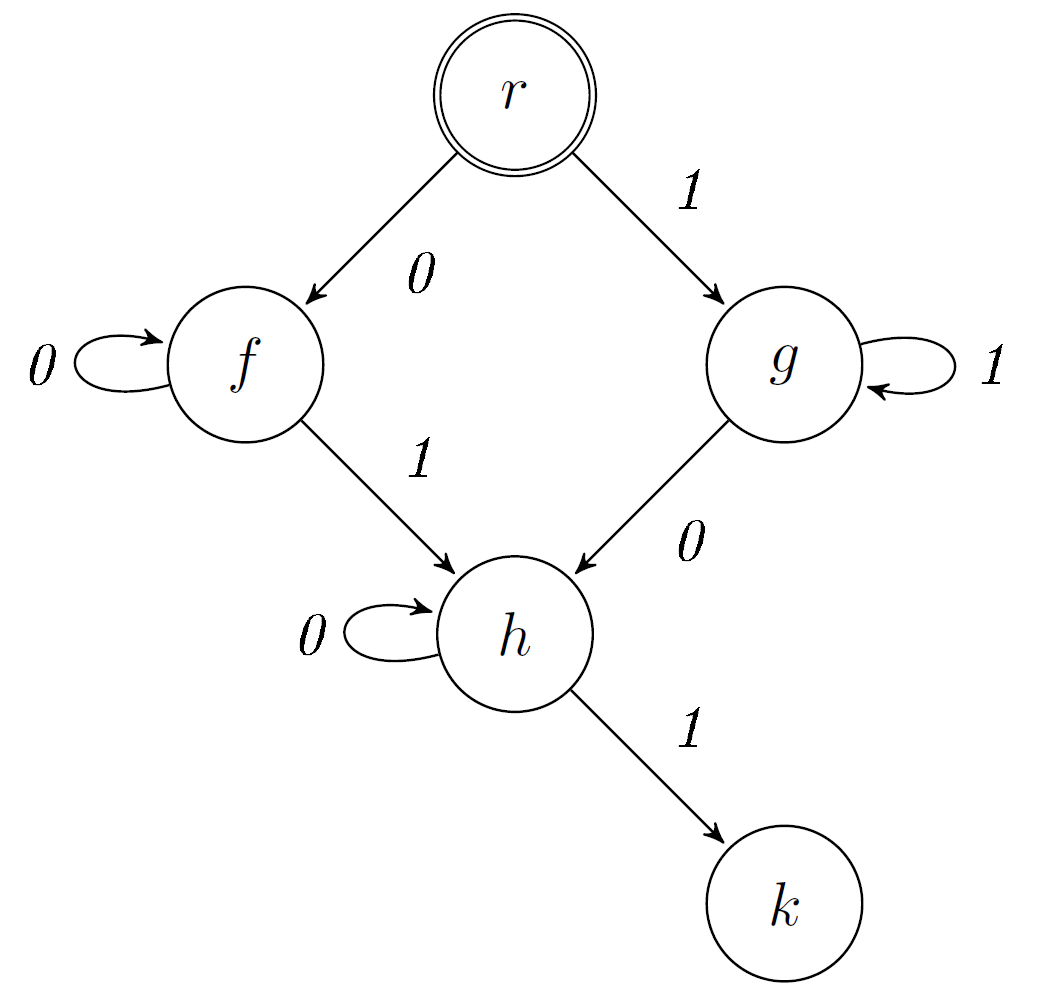}
	\caption{A rooted tree-automaton (we gave labels to the vertices in the figure so it would be easier to refer to them).}\label{fig:Ex}
\end{figure}

We make the following  observation. 

\begin{Lemma}\label{reduced accepted}
	Let $\mathcal A_r$ be a tree-automaton. Let $(T_+,T_-)$ be a tree-diagram accepted by $\A_r$. Then the reduced tree-diagram equivalent to $(T_+,T_-)$ is also accepted by $\A_r$. 
\end{Lemma}

\begin{proof}
	Assume that $(T_+,T_-)$ is not reduced and let $(R_+,R_-)$ be a tree-diagram obtained from $(T_+,T_-)$ by the reduction of a single common caret. It suffices to prove that $(R_+,R_-)$ is accepted by $\A_r$. To do so, consider the relation between the pairs of branches of $(R_+,R_-)$ and the pairs of branches of $(T_+,T_-)$. 
	There exists one pair of branches $u\to v$ of $(R_+,R_-)$ such that $(T_+,T_-)$ has the pairs of branches $u0\to v0$ and $u1\to v1$. All other pairs of branches of $(R_+,R_-)$ are also pairs of branches of $(T_+,T_-)$. Hence, to prove that $(R_+,R_-)$ is accepted by $\A_r$ it suffices to prove that $u$ and $v$ label paths in $\A_r$ such that $u^+=v^+$. 
	Since $(T_+,T_-)$ is accepted by $\A_r$, $u0$, $v0$, $u1$ and $v1$ label paths in $\A_r$ such that $(u0)^+=(v0)^+$ 
	 and  $(u1)^+=(v1)^+$. It follows that $u$ and $v$ label paths in $\A_r$ and that both $u^+$ and $v^+$ are fathers such that their left child is $(u0)^+=(v0)^+$ and their right child is $(u1)^+=(v1)^+$. 
	  Hence, by Condition $(3)$ in the definition of a tree-automaton, the vertices $u^+$ and $v^+$ in $\A_r$ must coincide. Hence, $(R_+,R_-)$ is accepted by $\A_r$. 
\end{proof}

More generally, we have the following. 

\begin{Lemma}\label{inv_reduced}
	Let $\A_r$ be a tree-automaton and let $(T_+,T_-)$ be a  tree-diagram accepted by $\A_r$. Let $(R_+,R_-)$ be a tree-diagram equivalent to $(T_+,T_-)$. Then $(R_+,R_-)$ is accepted by $\A_r$ if and only if $R_+$ (equiv., $R_-$) is readable on $\A_r$. 
\end{Lemma}

The proof of Lemma \ref{inv_reduced} is similar to the proof of Lemma \ref{reduced accepted}, 
and follows easily from Conditions $(2)$ and $(3)$ in the definition of a tree-automaton.

\begin{Lemma}\label{product}
	Let $\A_r$ be a tree-automaton. Let $(T_+,T_-)$ and $(R_+,R_-)$ be reduced tree diagrams accepted by $\A_r$. Then the product $(T_+,T_-)\cdot (R_+,R_-)$ is accepted by $\A_r$.
\end{Lemma}

\begin{proof}
	By assumption, the trees $T_-$ and $R_+$ are readable on $\A_r$. Let $S$ be the minimal finite binary tree such that $T_-$ and $R_+$ are rooted subtrees of $S$. Since every branch of $S$ is either a branch of $T_-$ or a branch of $R_+$, the tree $S$ is readable on $\A_r$. 
	One can insert common carets to the tree-diagram $(T_+,T_-)$ until one gets an equivalent tree-diagram of the form $(T',S)$.
	Similarly, one can insert common carets to the tree-diagram $(R_+,R_-)$ to get the equivalent tree-diagram $(S,R')$. Since $S$ is readable on $\A_r$, by Lemma \ref{inv_reduced}, both $(T',S)$ and $(S,R')$ are accepted by $\A_r$. 
	It follows easily that $(T',R')$ is accepted by $\A_r$. But the product of $(T_+,T_-)$ and  $(R_+,R_-)$ is the reduced tree-diagram equivalent to $(T',R')$. Hence, by Lemma \ref{reduced accepted}, it is accepted by $\A_r$. 
\end{proof}

Lemma \ref{product} implies that if $\A_r$ is a tree-automaton then the set of all reduced tree-diagrams in $F$ accepted by $\A_r$ is a subgroup of $F$. 

\begin{Definition}\label{lab closed}
	Let $\A_r$ be a  tree-automaton. We define the \emph{diagram group over $\A_r$}, denoted $\ddd (\A_r)$, to be the subgroup of $F$ of all (reduced) tree-diagrams $(T_+,T_-)$ accepted by $\A_r$. 
\end{Definition}

Note that the diagram groups defined in Definition \ref{lab closed} are a special case of the diagram groups defined in \cite{GuSa97} by Guba and Sapir.  

\begin{Definition}\label{def closed}
	A subgroup $H$ of $F$ is  \emph{closed}  if it is a diagram group over some tree-automaton, i.e., if there exists a tree-automaton $\mathcal A_r$ such that $H=\ddd(\A_r)$. 
\end{Definition}

\begin{Example}
	Thompson's group $F$ is closed. Indeed, let $\A_r$ be the tree-automaton with a unique vertex: the root $r$; and two directed loops from $r$ to itself  (one labeled ``0'' and the other labeled ``1''). Then the diagram group $\ddd(\A_r)=F$. 
\end{Example}

Note that if $\A_r$ is a finite tree-automaton then the membership problem in the subgroup $\ddd(A_r)$ of Thompson's group $F$ is decidable. Indeed, as noted above, given a reduced tree-diagram $(T_+,T_-)$ in $F$, it is decidable whether $(T_+,T_-)$ is accepted by $\A_r$.

Let $H$ be a subgroup of $F$. A function $f\in F$ is said to be a \emph{piecewise-$H$} function if there is a finite subdivision of the interval $[0,1]$ such that on each interval in the subdivision $f$ coincides with some function in $H$. We note that since all breakpoints of elements in $F$ are dyadic fractions, a function $f\in F$ is a piecewise-$H$ function if and only if there is a  dyadic subdivision of the interval $[0,1]$ into finitely many pieces such that on each dyadic interval in the subdivision $f$ coincides with some function in $H$. Equivalently, a function $f\in F$ is a  piecewise-$H$ function if and only if it has a (not necessarily reduced) tree-diagram $(T_+,T_-)$ such that each pair of branches $u\to v$ of $(T_+,T_-)$ is a pair of branches of some element in $H$. 
 The following lemma was proved in \cite{G}.

\begin{Lemma}\label{dyadic piecewise}
	Let $H$ be a subgroup of $F$. Then $H$ is closed (i.e., there is a tree-automaton $\A_r$ such that $H=\ddd(\A_r)$) if and only if every function $f\in F$ which is a piecewise-$H$ function belongs to $H$.  	
\end{Lemma}

\begin{Remark}\label{intersection}
	It follows from Lemma \ref{dyadic piecewise} that the intersection of closed subgroups of $F$ is a closed subgroup of $F$. One can also show it directly from Definition \ref{lab closed} using an appropriately defined ``pullback'' of rooted tree-automata. 
\end{Remark}

\subsection{The core of subgroups of Thompson's group $F$}

Let $H$ be a subgroup of $F$. We are interested in the smallest closed subgroup of $F$ which contains $H$ (note that by Remark \ref{intersection}, such a subgroup exists). For that, we define the core of the subgroup $H$ of $F$. The following definition is an adaptation of the definition from \cite{GS1} to the language of tree-diagrams. 

\begin{Definition}\label{core}
	Let $H$ be a subgroup of $F$ generated by the set $\mathcal S=\{(T^i_+,T^i_-): i\in\mathcal I\}$ of reduced tree-diagrams. The \emph{core} of $H$ is the rooted tree-automaton denoted $\mathcal C(H)$ defined as follows. 
	
	For each $i\in \mathcal I$ we can consider the trees $T^i_+$ and $T^i_-$ as directed edge-labeled graphs, where edges are directed away from the root and left edges are labeled by ``0'' and right edges are labeled ``1''. For each $i\in \mathcal I$, we ``glue''  each leaf of  $T^i_+$ to the corresponding leaf of $T^i_-$ 
	 (i.e., we identify each pair of corresponding leaves to a single vertex). We also glue the root of $T^i_+$ to the root of $T^i_-$ and denote the oriented graph obtained by $S_i$
	(one can think of $S_i$ as drawn on a sphere.). The ``root'' of $S_i$ is the vertex formed by the identification of the roots of $T^i_+$ and $T^i_-$. Note that this is the only vertex in $S_i$ with no incoming edges. 
	
	Next, we identify the roots of all the directed graphs $S_i$ to a single root vertex $r$. To the directed edge-labeled graph obtained we apply foldings of two different types:
	\begin{enumerate}
		\item[$(1)$] If a vertex $x$ has several outgoing edges labeled by the same label, we identify all of these edges to a single edge and all of their end-vertices to a single vertex. 
		
		We repeat step $(1)$ as long as it is applicable. 
		As a result (if $S$ is infinite, then in the limit state, after possibly infinitely many foldings) we get a directed edge labeled-graph where every vertex $x$ has either zero or two outgoing edges: one \textit{left} edge labeled ``0'' and one \textit{right} edge  labeled ``1'' (in which case we will refer to their end vertices as the \textit{children} of $x$).
		\item[$(2)$] If $x$ and $y$ are distinct vertices in the directed graph obtained such that both $x$ and $y$ have $2$ outgoing edges and such that each child of $x$ coincides with the respective child of $y$, we identify the vertices $x$ and $y$, we identify their left outgoing edges and identify  their right outgoing edges. 
		
		We repeat step $(2)$ as long as it is applicable (if $S$ is infinite we may have to apply infinitely many foldings). 
	\end{enumerate}
	Note that at the end of this process, 
	 every vertex has either zero or two ougoing edges, one labeled ``0'' and the other labeled ``1'', and the unique vertex with no incoming edges is the root $r$. The foldings guarantee that the resulting directed edge-labeled graph satisfies Conditions $(2)$ and $(3)$ from Definition \ref{def:tree-aut}. Condition (4) from the definition is also satisfied, since even before the application of foldings, for each vertex in the graph there was a directed path from the root to the vertex. Hence, the directed graph obtained is a rooted tree-automaton. It is called the \textit{core} of $H$ and denoted  $\mathcal C(H)$.   
\end{Definition}


\begin{Example}
	Let $H=\la x_0x_1^{-1},x_0^2x_1x_3^{-1}\ra$. We demonstrate the construction of the core $\mathcal C(H)$. First, we start with the reduced tree-diagrams $(T_+^1,T_-^1)$ and $(T_+^2,T_-^2)$ of the generatorts  $x_0x_1^{-1}$ and $x_0^2x_1x_3^{-1}$, where each tree is considered as a directed edge-labeld graph (see Figure \ref{fig:genH}). Next, for $i=1,2$, we identify each leaf of $T_+^i$ with the corresponding leaf of $T_-^i$ and we identify the roots of the trees $T_+^1,T_-^1,T_+^2,T_-^2$ to a single root vertex as depicted in Figure \ref{fig:step1}. In the figure, we labeled all the vertices, giving identified vertices the same label and labeling the root by $r$. Next, we apply foldings, starting with foldings of type (1): The vertex $r$ has $4$ distinct outgoing edges  labeled ``0''. We identify all of them to a single edge and identify their end vertices $11,1,15,5$ to a single vertex. Similarly, the vertex $r$ has $4$ distinct outgoing edges labeled ``1''. We identify all of them to a single edge and identify their end vertices $12,13,16,19$ to a single vertex. The result is depicted in Figure \ref{fig:firstfolding} (for convenience, when edges are identified we color them by the same (non-black) color, when several vertices are identified, we label all of them (in every place they appear in the figure) by the smallest of their labels).  Now, in Figure \ref{fig:firstfolding}, the vertex $1$ has two distinct outgoing edges labeled ``0'' and two distinct outgoing edges labeled ``1''. We identify the outgoing edges labeled ``0'' to a single edge and identify their end vertices $1,17$ to a single vertex. Similarly, we identify the outgoing edges labeled ``1'' to a single edge and their end vertices $2,8$ to a single vertex. The result is depicted in Figure \ref{fig:secondfolding}. 
	 Now, In Figure \ref{fig:secondfolding}, the vertex $1$ has two distinct outgoing edges labeled ``0'' (a green edge and a black edge). We identify them to a single edge (note that their end vertices are already identified). Similarly, the vertex $1$ has two distinct outgoing edges labeled ``1'' (an orange edge and a black edge). We identify them to a single edge and we identify their end vertices $2,18$ to a single vertex. The result is depicted in Figure \ref{fig:thirdfolding}.  Now, in Figure \ref{fig:thirdfolding}, the vertex $12$ has four distinct outgoing edges labeled ``0''. We identify them to a single edge and their end vertices $3,14,9,6$ to a single vertex . Similarly, the vertex $12$ has four distinct outgoing edges labeled ``1''. We identify them to a single edge and we identify their end vertices $4,10,20$ to a single vertex. The result is depicted in Figure \ref{fig:fourthfolding}. Notice that in Figure \ref{fig:fourthfolding} there is no vertex with distinct outgoing edges labeled by the same label. Hence, we are done applying foldings of type (1) and we move on to applying foldings of type (2): In Figure \ref{fig:fourthfolding} the vertices $3$ and $22$ are distinct vertices, but their left children coincide and their right children coincide. Hence, we identify these vertices, as well as their left outgoing edges and their right outgoing edges. The result is depicted in Figure \ref{fig:first2folding}. Now, in Figure \ref{fig:first2folding}, the vertices $12$ and $21$ are distinct vertices, but their left children coincide and their right children coincide. Hence, we identify these vertices, as well as their left outgoing edges and their right outgoing edges. The result is depicted in Figure \ref{fig:second2folding}. Notice that in Figure \ref{fig:second2folding} there are no more applicable foldings. Hence, the process is finished and the rooted tree-automaton in Figure \ref{fig:second2folding} is the core of $H$ (where all vertices with the same label are identified and all non-black edges with the same color are identified). The obtained core of $H$ is also depicted in Figure \ref{fig:coreH}.

	

\end{Example}

\begin{figure}
	\begin{center}
		\includegraphics{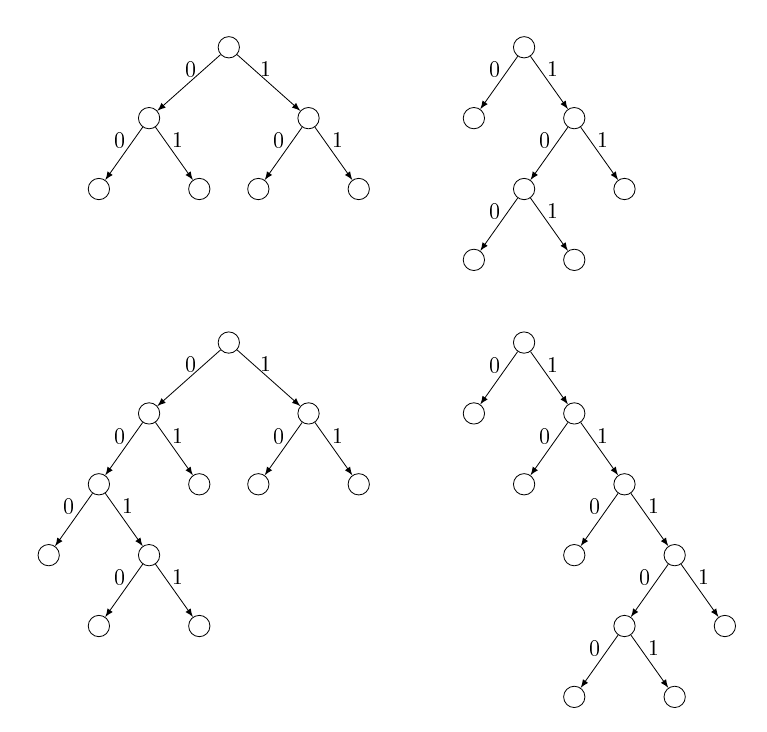}
	\end{center}
	\caption{The reduced tree-diagram $(T_+^1,T_-^1)$ of $x_0x_1^{-1}$ appears in the first row. The reduced tree-diagram $(T_+^2,T_-^2)$ of $x_0^2x_1x_3^{-1}$ appears in the second row.}
	\label{fig:genH}
\end{figure}

\begin{figure}
	\begin{center}
		\includegraphics{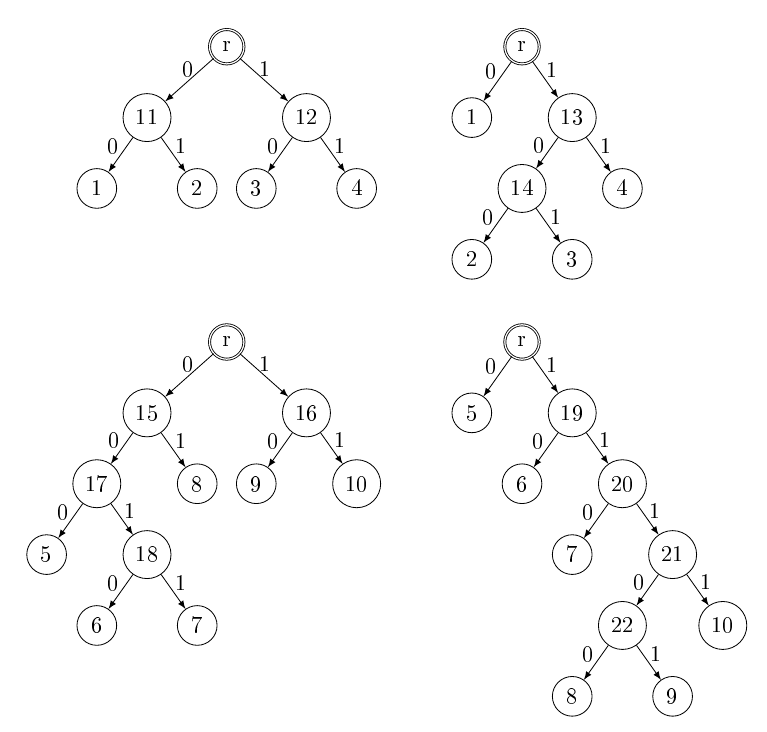}
	\end{center}
	\caption{For $i=1,2$, we identified each leaf of $T_+^i$ with the corresponding leaf of $T_-^i$ and we identified the roots of the trees $T_+^1,T_-^1,T_+^2,T_-^2$ to a single root vertex $r$.}
	\label{fig:step1}
\end{figure}

\begin{figure}
	\begin{center}
				\includegraphics{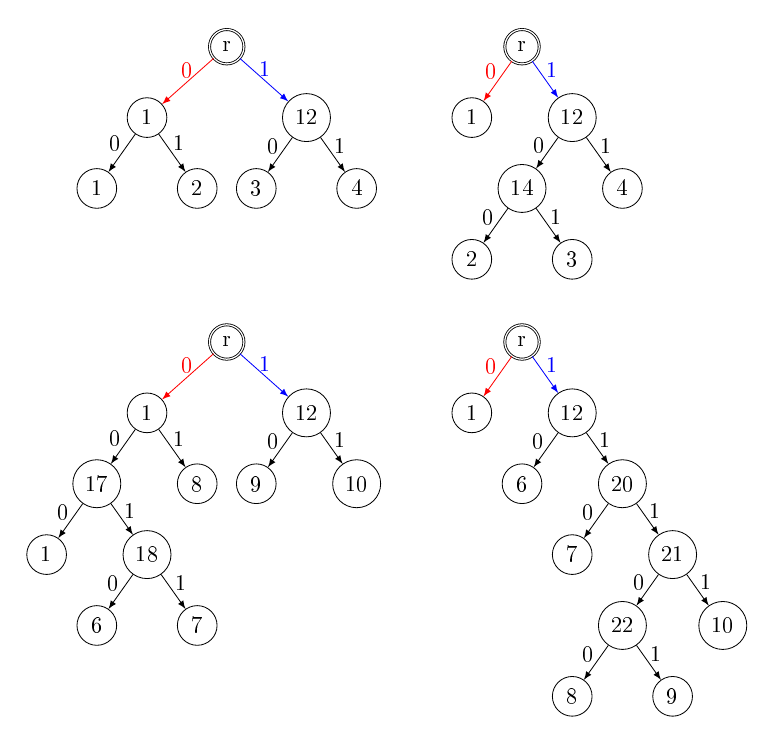}
	\end{center}
	\caption{The red edges are now identified and their end vertices are also identified. Similarly, for the blue edges.}
	\label{fig:firstfolding}
\end{figure}

\begin{figure}
	\begin{center}
		\includegraphics{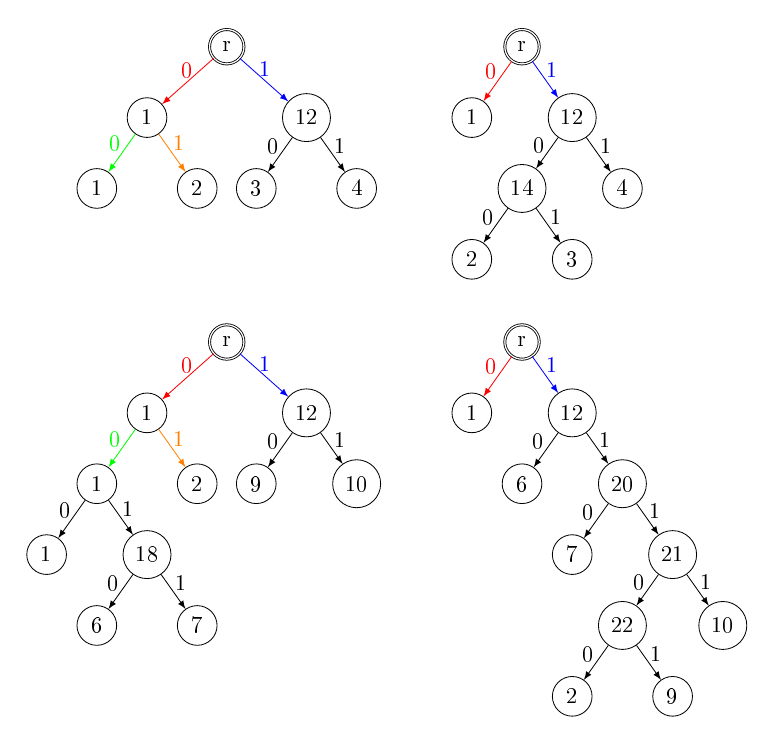}
	\end{center}
	\caption{The green edges are now identified and their end vertices are identified as well. Similarly for the orange edges.}
	\label{fig:secondfolding}
\end{figure}

\begin{figure}
	\begin{center}
		\includegraphics{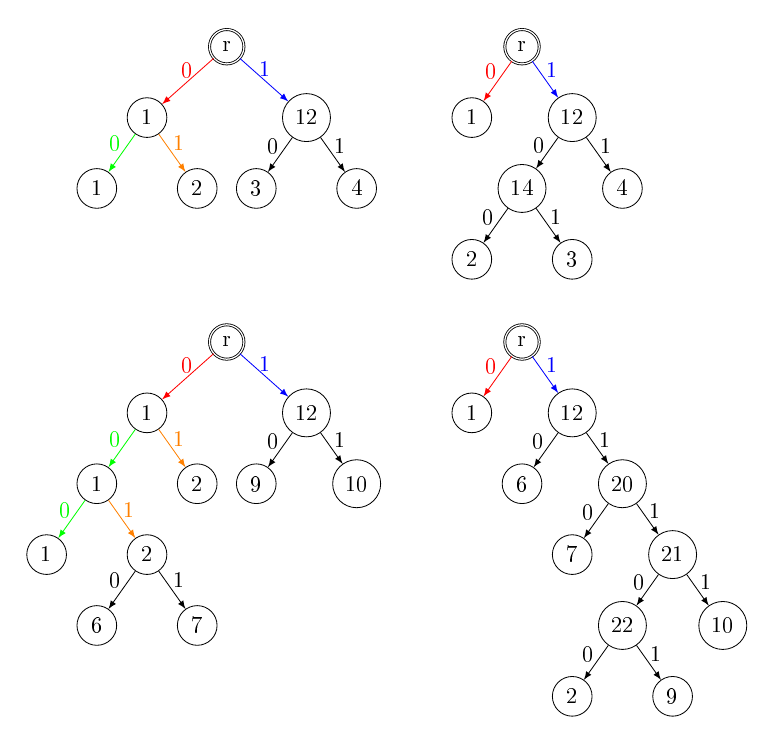}
	\end{center}
	\caption{All the green edges are now identified, all the orange edges are now identified and their end vertices are identified as well.}
	\label{fig:thirdfolding}
\end{figure}

\begin{figure}
	\begin{center}
		\includegraphics{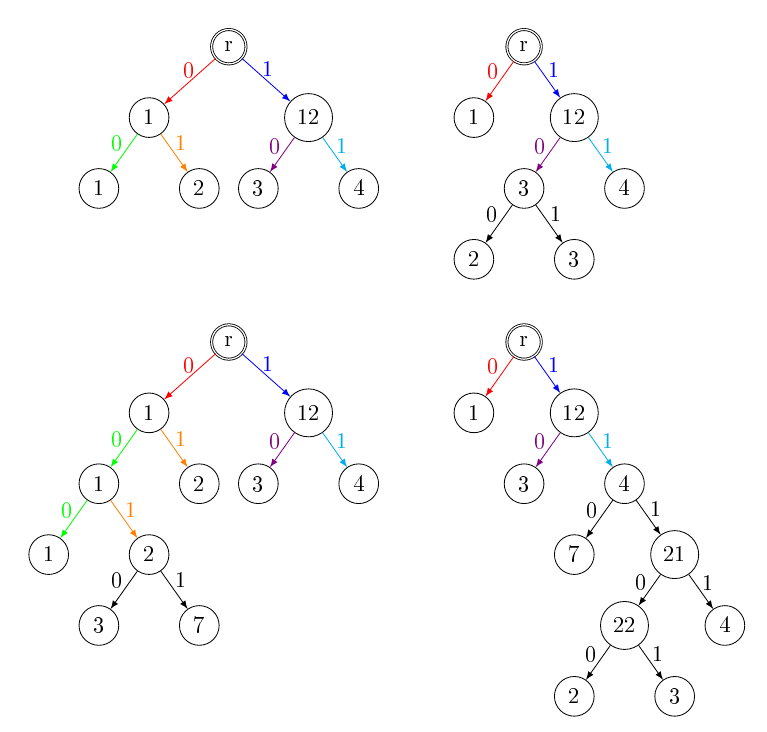}
	\end{center}
	\caption{All the violet edges are now identified and their end vertices are identified as well. Similarly for the cyan edges.}
	\label{fig:fourthfolding}
\end{figure}

\begin{figure}
	\begin{center}
		\includegraphics{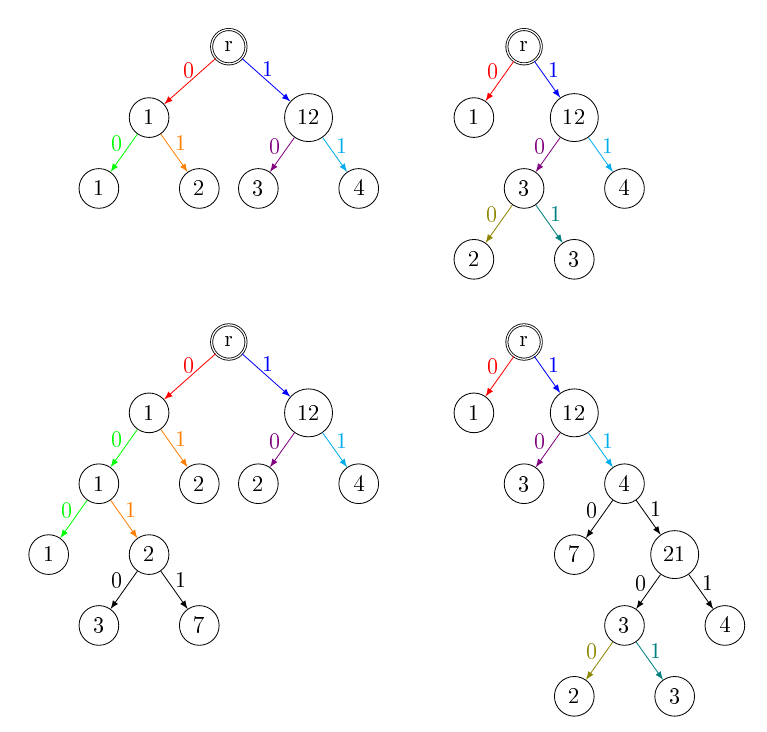}
	\end{center}
	\caption{The olive edges are now identified, the teal edges are now identified and their intial vertices are also identified.}
	\label{fig:first2folding}
\end{figure}

\begin{figure}
	\begin{center}
			\includegraphics{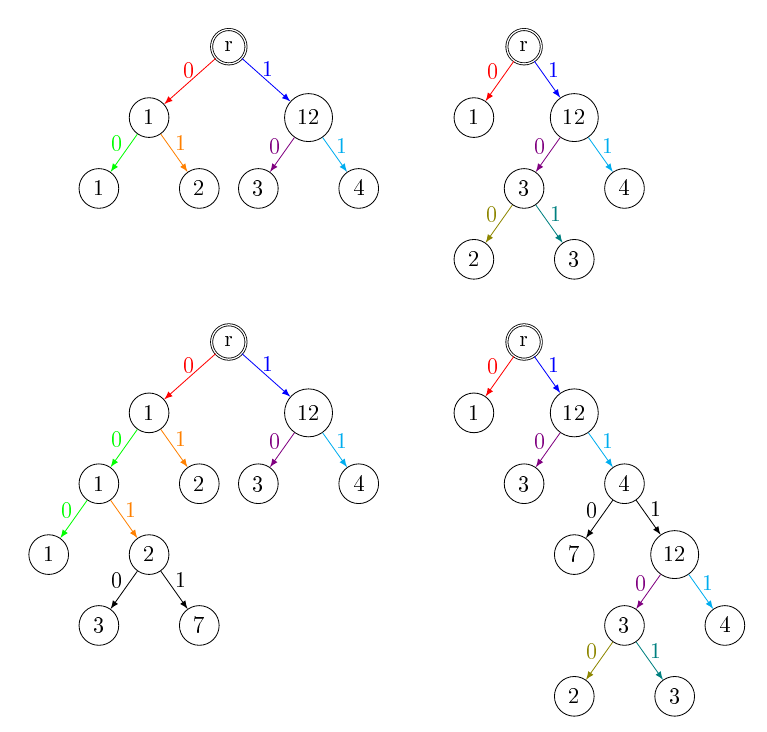}
	\end{center}
	\caption{All the violet edges are now identified, all the cyan edges are now identified and their initial vertices are also identified.}
	\label{fig:second2folding}
\end{figure}

\begin{figure}
	\begin{center}
			\includegraphics{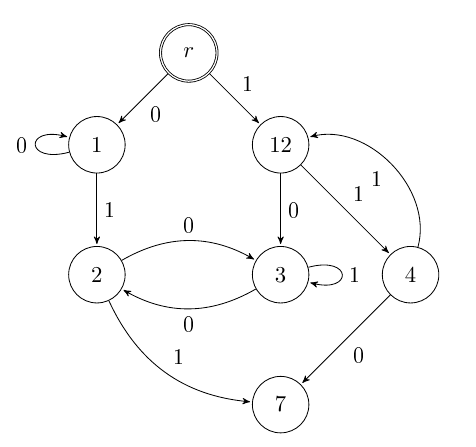}
	\end{center}
	\caption{The core of $H$.}
	\label{fig:coreH}
\end{figure}


As noted in \cite{GS1}, the core of $H$ does not depend on the chosen generating set nor on the order of foldings applied.

It follows from the definition of the core of $H$ that the core $\mathcal C(H)$ accepts the generators of $H$ and hence, by Lemma \ref{product}, the entire subgroup $H$. 

\begin{Definition}
	Let $H$ be a subgroup of $F$. Let $\mathcal C(H)$ be the core of $H$. The \emph{closure} of $H$, denoted $\Cl(H)$, is the subgroup of $F$ of all (reduced) tree-diagrams accepted by $\mathcal C(H)$. In other words, the closure of $H$ is the diagram group $\ddd(\mathcal C(H))$. 
\end{Definition}

Let $H$ be a subgroup of $F$. Then $\Cl(H)$ is a closed subgroup of $F$ which contains $H$. By \cite[Theorem 5.6]{G}, the closure of $H$ is the subgroup of $F$ of all piecewise-$H$ functions. Hence, by Lemma \ref{dyadic piecewise}, the closure of $H$ is the minimal closed subgroup of $F$ which contains $H$. In particular, the subgroup $H$ is closed if and only if $H=\Cl(H)$.

\subsection{On the core and closure of subgroups of $F$}

Below, we recall some useful results about the core and the closure of subgroups of $F$. But first, we will need the following lemma.

 A \emph{trail} in a tree-automaton $\A_r$ is a finite sequence of directed edges $e_1,\dots,e_n$ such that for each $i=1,\dots,n-1$ we have ${e_i}_+={e_{i+1}}_-$ (that is, a trail is a ``path'' which does not necessarily start from the root). Clearly, every trail has a finite binary label. Note that if $x$ is a vertex in $\A_r$, then for every finite binary word $u$, there is at most one trail in $\A_r$ labeled $u$ with initial vertex $x$. 

\begin{Lemma}\label{lem:paths in automaton}
	Let $\A_r$ be a tree-automaton such that $u$ and $v$ label paths in $\A_r$. Assume that there is a function $f$ in the diagram group $\ddd(\A_r)$ such that $u\to v$ is a pair of branches of $f$. Then $u^+=v^+$ in $\A_r$. 
\end{Lemma}

\begin{proof}
	Let $(T_+,T_-)$ be the reduced tree-diagram of $f$. Since $f$ belongs to $\ddd(\A_r)$, the tree-diagram $(T_+,T_-)$ is accepted by $\A_r$. Hence, if $u\to v$ is a pair of branches of the reduced tree-diagram $(T_+,T_-)$, we are done. Otherwise, by Remark \ref{rem:branches}, there are finite binary words $p,q,w$ such that $u\equiv pw$, $v\equiv qw$ and such that $p\to q$ is a pair of branches of the reduced tree-diagram of $f$. In that case, since $(T_+,T_-)$ is accepted by $\A_r$, we have that $p^+=q^+$ in $\A_r$. Since the word $u\equiv pw$ labels a path in the core, the word $w$ labels a trail in the core with initial vertex $p^+=q^+$. This trail ends at the vertex $(pw)^+=(qw)^+$. Hence, $u^+=(pw)^+=(qw)^+=v^+$ as necessary. 
\end{proof}

Let $H$ be a subgroup of $F$ and assume that $u$ and $v$ are paths in the core $\mathcal C(H)$.   By Lemma \ref{lem:paths in automaton}, if there is a function in $\Cl(H)$ with the pair of branches $u\to v$, then in the core, we have $u^+=v^+$. The following lemma says that the other direction is also true.

\begin{Lemma}[{\cite[Lemma 6.1]{G}}]\label{identified vertices in the core}
	Let $H$ be a subgroup of $F$ and let $\mathcal C(H)$ be the core of $H$. Let $u$ and $v$ be paths in the core $\mathcal C(H)$. Then $u^+=v^+$ if an only if there is an element $h\in \Cl(H)$ such that $h$ has the pair of branches $u\to v$. 
\end{Lemma}

Intuitively, Lemma \ref{identified vertices in the core} says that two paths in the core of $H$ terminate on the same vertex if and only if they ``have to'' in order for the core to accept the subgroup $\Cl(H)$.

The following lemma follows from Lemma \ref{identified vertices in the core} and the fact the  closure of $H$ is the subgroup of $F$ of all piecewise-$H$ functions.

\begin{Lemma}[{\cite[Lemma 4.6]{G}}]\label{identified vertices H}
	Let $H$ be a subgroup of $F$ and let $\mathcal C(H)$ be the core of $H$. Let $u$ and $v$ be finite binary words which label paths in $\mathcal C(H)$. Then $u^+=v^+$ if and only if there is $k\in\mathbb{N}$ such that for any finite binary word $w$ of length $\geq k$, there is an element $h\in H$ with the pair of branches $uw\to vw$. 
\end{Lemma}

Recall that if $\mathcal C(H)$ is the core of a subgroup $H$ of $F$, then 
 for every vertex $x$ of the core, there is a directed path in the core from the root to $x$.

\begin{Definition}
	Let $H$ be a subgroup of $F$ and let $\mathcal C(H)$ be the core of $H$. Let $x$ be a vertex of $\mathcal C(H)$. 
	\begin{enumerate}
		\item[$(1)$] If there is $n\in\mathbb{N}$ such that $u\equiv 0^n$ is a path in the core such that $u^+=x$  then the vertex $x$ is called a \emph{left vertex} of the core. 
		\item[$(2)$] If there is $n\in\mathbb{N}$ such that $u\equiv 1^n$ is a path in the core such that $u^+=x$  then the vertex $x$ is  a \emph{right vertex} of the core. 
		\item[$(3)$] If there is a path $u$ in the core which contains both digits $0$ and $1$ such that $u^+=x$, then $x$ is a \emph{middle} vertex of the core.
	\end{enumerate}
\end{Definition}

\begin{Remark}\label{rem:different types}
	Let $\mathcal C(H)$ be the core of a subgroup $H$ of $F$. Then each vertex of $\mathcal C(H)$ is exactly one of the following: $(1)$ the root, $(2)$ a left vertex, $(3)$ a right vertex, or $(4)$ a middle vertex.
\end{Remark}

Indeed, since each vertex in the core of $H$ is the end-vertex of some directed path in the core, each vertex in the core is of one of the four mentioned types. Lemma \ref{identified vertices in the core} implies that a vertex cannot be of two different types (for example, a vertex cannot be both a left vertex and a right vertex because that would imply that there is an element in $\Cl(H)$ with a pair of branches of the form, $0^n\to 1^m$ for some $n,m\in\mathbb{N}$). Note also that if $x$ is a middle vertex of the core of $H$ and $x$ has two outgoing edges, then its children are also middle vertices of the core.

\begin{Example}
	The core of Thompson's group $F$ is given in Figure \ref{fig:coreF}. Note that it has exactly four vertices: the root, a unique left vertex, a unique right vertex and a unique middle vertex.  
\end{Example}

One can verify that this is the core of $F$ using the construction in Definition \ref{core} (for example, starting with the generating set $\{x_0,x_1\}$). Alternatively, it follows from Lemma \ref{identified vertices in the core}. Indeed, since the core of $F$ accepts every reduced tree-diagram in $F$, every finite binary word $u$ labels a path in the core of $F$. Then, the fact that there is a unique middle vertex in the core follows from Lemma \ref{identified vertices in the core}, since for every pair of finite binary words $u$ and $v$  which contain both digits $0$ and $1$ there is an element in $F$ with the pair of branches $u\to v$. Similarly, Lemma \ref{identified vertices in the core} implies that there is a unique left vertex and a unique right vertex in the core of $F$.

\begin{figure}[ht]
	\centering
	\includegraphics[width=.4\linewidth]{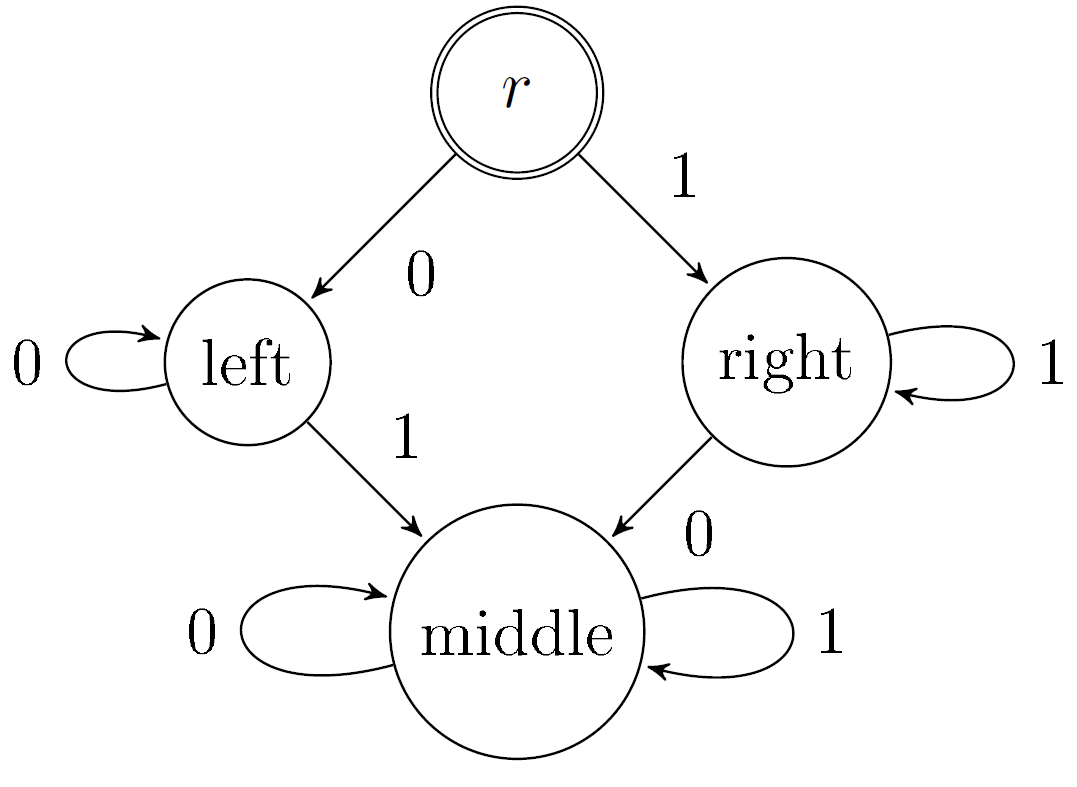}
		\caption{The core of Thompson's group $F$. The root of the core is labeled $r$. There is a unique left vertex, a unique right vertex and a unique middle vertex in the core.}\label{fig:coreF}
\end{figure}

Now, let $H$ be a subgroup of $F$. Since the closure of $H$ is the subgroup of $F$ of all piecewise-$H$ functions, the orbits of the action of $H$ on the set of dyadic fractions $\mathcal D$ coincide with the orbits of the action of $\Cl(H)$. The following lemma follows from (the more general) \cite[
Theorem 6.5]{G}. 
Recall that an inner vertex of a tree-automaton is a vertex which has two outgoing edges. 

\begin{Lemma}\label{lem:finitely many orbits}
	Let $H$ be a subgroup of $F$ and assume that the core $\mathcal C(H)$ is finite (i.e, that there are finitely many vertices in $\mathcal C(H)$). Then the action of $H$ on the set of  dyadic fractions $\mathcal D$ has finitely many orbits if and only if every vertex in $\mathcal C(H)$ is an inner vertex. 
\end{Lemma}

\subsection{The derived subgroup of $F$}\label{der_sub}

The derived subgroup of $F$ is an infinite simple group \cite{CFP}. It can be characterized as the subgroup of $F$ of all functions $f$ with slope $1$ both at $0^+$ and at $1^-$ (see \cite{CFP}). In other words, it is the subgroup of all functions in $F$ supported in the interval $(0,1)$.
In particular, the derived subgroup of $F$ acts transitively on the set of dyadic fraction $\mathcal D$. (Indeed, for every pair of dyadic fractions $\alpha,\beta\in\mathcal D$ there is a function $f\in F$ such that $f(\alpha)=\beta$ and such that $f$ is supported in $(0,1)$.)

Since $[F,F]$ is infinite and simple, every finite index subgroup of $F$ contains the derived subgroup of $F$.
Hence, there is a one-to-one correspondence between finite index subgroups of $F$ and finite index subgroups of the abelianization $F/[F,F]$. 

Recall that the abelianization of $F$ is isomorphic to $\mathbb{Z}^2$.  The 
standard map from $F$ to its abelianization $\pi_{ab}\colon F\to \mathbb{Z}^2$ sends an element $f\in F$ to $(\log_2(f'(0^+)),\log_2(f'(1^-)))$ (see, for example, \cite{CFP}). 
Below, when we refer to the image of a subgroup $H$ of $F$ in the abelianization of $F$, we refer to its image in $\mathbb{Z}^2$ under $\pi_{ab}$. 
 The following remark will be useful. 

\begin{Remark}\label{rem:H[F,F]}
	Let $H$ be a subgroup of $F$. Then $H$ is contained in a proper finite index subgroup of $F$ if and only if $H[F,F]<F$  (in other words, if and only if $\pi_{\ab}(H)$ is a strict subgroup of $\mathbb{Z}^2$). 
\end{Remark}

\begin{proof}
	Follows from the fact that every finite index subgroup of $F$ contains the derived subgroup of $F$ and the fact that every strict subgroup of $\mathbb{Z}^2$ is contained in a finite index subgroup of $\mathbb{Z}^2$.
\end{proof}

As noted, there is a one-to-one correspondence between finite index subgroups of $\mathbb{Z}^2$ and finite index subgroups of $F$. 
More generally, there is a one-to-one correspondence between subgroups of $\mathbb{Z}^2$ and subgroups of $F$ which contain the derived subgroup of $F$.  We will be particularly interested in subgroups of $\mathbb{Z}^2$ whose preimage under $\pi_{ab}$ is a closed subgroup of $F$. 

\begin{Definition}
	Let $K$ be a subgroup of $\mathbb{Z}^2$. We say that $K$ is a \emph{closed} subgroup of $\mathbb{Z}^2$ if its preimage under $\pi_{ab}$ is a closed subgroup of $F$. 
\end{Definition}

The following lemma follows easily from the characterization of closed subgroups of $F$ as subgroups $H$ that are closed under taking piecewise-$H$ functions. 

\begin{Lemma}\label{closed subgroups of Z^2}
	Let $K$ be a  subgroup of $\mathbb{Z}^2$. Then $K$ is closed if and only if there exist integers $p,q\geq 0$ 
	such that $K=p\mathbb{Z}\times q\mathbb{Z}$.
\end{Lemma}

\begin{proof}
	Assume that there exist $p,q\geq 0$ such that $K=p\mathbb{Z}\times q\mathbb{Z}$. We denote by $F_{p,q}$ the preimage in $F$ of $p\mathbb{Z}\times q\mathbb{Z}$. In other words, $F_{p,q}$ is the subgroup of $F$ of all functions $f$ such that $\log_2(f'(0^+))$ is an integer multiple of $p$ and $\log_2(f'(1^-))$ is an integer multiple of $q$. The subgroup $F_{p,q}$  clearly contains every piecewise-$F_{p,q}$ function. Hence, it is a closed subgroup of $F$. Hence, by definition, $K$ is a closed subgroup of $F$.
	
	In the opposite direction, assume that $K$ is a closed subgroup of $\mathbb{Z}^2$.
	We let $p=\gcd\{a>0 \mid (a,b)\in K\}$ 
	 and $q=\gcd\{b>0\mid (a,b)\in K\}$ (where $\gcd(\emptyset)$ is taken to be zero).
	 Clearly, $K\subseteq p\mathbb{Z}\times q\mathbb{Z}$. We claim that the inverse inclusion also holds. Indeed, let $H$ be the preimage of $K$ under $\pi_{ab}$ and note that $H$ is a closed subgroup of $F$. By the choice of $p$, there exists $b$ such that $(p,b)\in K$. Let $f=x_0^p\oplus x_0^{-b}$ and note that $f'(0^+)=2^p$ and $f'(1^-)=2^b$ (since the slope of $x_0$ at $0^+$ is $2$ and at $1^-$ is $2^{-1}$). Since $\pi_{ab}(f)=(p,b)\in K$, the function $f\in H$. Let $g=x_0^p\oplus 1$ and note that $g$ is a piecewise-$H$ functions. Since $H$ is  closed, the function $g\in H$. Hence, $\pi_{ab}(g)=(p,0)\in K$. In a similar way, one can show that $(0,q)\in K$. Hence, $p\mathbb{Z}\times q\mathbb{Z}\subseteq K$, as required. 
\end{proof}

Let $p,q\geq 0$. As in the proof of Lemma \ref{closed subgroups of Z^2}, we denote by $F_{p,q}$ the preimage in $F$ of $p\mathbb{Z}\times q\mathbb{Z}$. 
 Note that $F_{p,q}$ is of finite index in $F$ if and only if $p,q\geq 1$. In \cite{BW}, Bleak and Wassink proved that for every $p,q\geq 1$, the subgroup $F_{p,q}$ of $F$ (which they denote by $K_{(p,q)}$ and call a \emph{rectangular subgroup of $F$}) is isomorphic to $F$. They also prove that every finite index subgroup of $F$ which is not of this form, is not isomorphic to $F$.

\subsection{Subgroups of $F$ whose closure contains $[F,F]$}

In this paper, we will be interested in subgroups of $F$ whose closure contains the derived subgroup of $F$. In \cite{G}, we gave a characterization of such subgroups in terms of their core.

\begin{Lemma}[{\cite[Lemma 7.1]{G}}]\label{core of derived subgroup}
	Let $H$ be a subgroup of $F$. Then $\Cl(H)$ contains the derived subgroup of $F$ if and only if the core $\mathcal C(H)$ has a unique middle vertex and that middle vertex is an inner vertex (i.e., it has two outgoing edges, necessarily to itself). 
\end{Lemma}

It is not difficult to check that if the core of $H$ is as described in Lemma \ref{core of derived subgroup} then it accepts every tree-diagram in $[F,F]$ and thus the closure of $H$ contains $[F,F]$. The opposite direction in the lemma follows from Lemma \ref{identified vertices in the core} and the fact that for every pair of finite binary words $u,v$ which contain both digits $0$ and $1$ there is an element in $[F,F]$ with the pair of branches $u\to v$.

In \cite{G}, we proved that if the closure of a subgroup $H$ of $F$ contains the derived subgroup of $F$, then the subgroup $H$ must be ``big'' in the sense that for any pair of branches of an element of $[F,F]$, if the branches in the pair  are ``extended'' a little, then there must be an element in $H$ with that pair of branches. More accurately, the following was proved in \cite{G}.

\begin{Lemma}[{\cite[Corollary 7.8]{G}}]\label{lm1}
	Let $H$ be a subgroup of $F$ such that $\Cl(H)$ contains the derived subgroup of $F$.
	Let $v_1$ and $v_2$ be a pair of finite binary words which contain both digits $0$ and $1$. Then there exists $k\in\mathbb{N}$ such that for any pair of finite binary words $w_1,w_2$ of length $\ge k$ there is an element $h\in H$ with the pair of branches $v_1w_1\rightarrow v_2w_2$.
\end{Lemma}

As a corollary from Lemma \ref{lm1} we have the following. 

\begin{Corollary}\label{cy1}
	Let $H$ be a subgroup of $F$ such that $\Cl(H)$ contains the derived subgroup of $F$. Then there is a finite binary word $u$ which contains both digits $0$ and $1$ such that for every finite binary word $w$ there is an element $h\in H$ with the pair of branches $u\rightarrow uw$.  
\end{Corollary}

\begin{proof}
	Let $v_1$ be a finite binary word which contains both digits $0$ and $1$ and let $v_2\equiv v_1$. By Lemma \ref{lm1}, there exists $k\in\mathbb{N}$ such that for every pair of finite binary words $w_1,w_2$ of length $\geq k$, there is an element in $H$ with the pair of branches $v_1w_1\to v_2w_2$. Then the finite binary word $u\equiv v_10^k$ satisfies the result. Indeed, for any finite binary word $w$, if one lets $w_1\equiv 0^k$ and $w_2\equiv 0^kw$, then there is an element in $H$ with the pair of branches $v_1w_1\to v_2w_2$, i.e., with the pair of branches $u\to uw$, as required. 
\end{proof}

The following lemma also shows that if the closure of $H$ contains the derived subgroup of $F$ then $H$ is ``big'' in the sense that it must contain elements with  certain properties. 

\begin{Lemma}[{\cite[Lemma 7.12]{G}}]\label{lm2}
	Let $H$ be a subgroup of $F$ such that $\Cl(H)$ contains the derived subgroup of $F$.
	Let $a<b$ in $(0,1)$ be finite dyadic fractions. Let $u$ be a finite binary word which contains both digits $0$ and $1$. Then there is an element $g\in H$ such that $g$ maps the interval $[a,b]$ into the dyadic interval $[u]$.
\end{Lemma}

\subsection{The generation problem in Thompson's group $F$}

Recall that in \cite{G}, we gave a solution for the generation problem in $F$. That is, we gave an algorithm such that given a finite subset $X$ of $F$ determines whether $X$ generates $F$. In fact, we gave an algorithm such that given a finite subset $X$ of $F$ determines whether the subgroup it generates contains the derived subgroup of $F$ (equivalently, whether the subgroup it generates is a normal subgroup of $F$ \cite{CFP}).

\begin{Theorem}[{\cite[Theorem 1.3]{G}}]\label{thm:der}
	Let $H$ be a subgroup of $F$. Then $H$ contains the derived subgroup of $F$ if and only if the following conditions hold. 
	\begin{enumerate}
		\item[$(1)$] $[F,F]\subseteq\Cl(H)$.
		\item[$(2)$]  There is an element $h\in H$ which fixes a dyadic fraction $\alpha\in(0,1)$ such that the slope $h'(\alpha^-)=2$ and the slope $h'(\alpha^+)=1$. 
	\end{enumerate}
\end{Theorem}

Given a finite subset $X$ of $F$, we let $H$ be the subgroup generated by $X$. Then it is decidable if condition $(1)$ holds for $H$ (see Lemma \ref{core of derived subgroup}). In \cite[Section 8]{G}, we gave an algorithm for deciding if $H$ satisfies condition $(2)$, given that $H$ satisfies condition $(1)$. Hence, Theorem \ref{thm:der} gives an algorithm for determining if $H$ contains $[F,F]$. 

Now, given a finite subset $X$ of $F$, let $H$ be the subgroup of $F$ generated by $X$. Clearly, if $H[F,F]\neq F$, then $H$ is a strict subgroup of $F$. If $H[F,F]=F$ (which can be checked easily using the abelianization map), then to determine if $H=F$, it suffices to check if $H$ contains the derived subgroup $[F,F]$. 
Hence, Theorem \ref{thm:der} gives a solution to the generation problem in $F$.

\begin{Corollary}\label{cor_int}
	Let $H$ be a subgroup of $F$. Then $H=F$ if and only if the following conditions hold. 
	\begin{enumerate}
		\item[$(1)$] $H[F,F]=F$.
		\item[$(2)$] $[F,F]\subseteq \Cl(H)$
		\item[$(3)$] There is a function $h\in H$ which fixes a finite dyadic fraction $\alpha\in \mathcal D$ such that the slope $h'(\alpha^-)=2$ and the slope $h'(\alpha^+)=1$. 
	\end{enumerate}
\end{Corollary}

\section{Improved solution to the generation problem in $F$}\label{main section}

In this section, we prove that if the image of $H$ in the abelianization of $F$ is closed, then the second condition in Theorem \ref{thm:der} is superfluous. More specifically, we prove the following.

\begin{Proposition}\label{main_pro}
	Let $H$ be a subgroup of $F$ such that the following conditions hold. 
	\begin{enumerate}
		\item[$(1)$] The image of $H$ in the abelianization of $F$ is a closed subgroup of $\mathbb{Z}^2$. 
		\item[$(2)$] $[F,F]\subseteq \Cl(H)$. 
	\end{enumerate}
	Then there is an element $h\in H$ which fixes a dyadic fraction $\alpha\in(0,1)$ such that the slope $h'(\alpha^-)=2$ and the slope $h'(\alpha^+)=1$. 
\end{Proposition}

The proof of Proposition \ref{main_pro} relies on some ideas from \cite{G}. Recall that in \cite[Section 8]{G}, we give an algorithm for determining if a subgroup $H$ of $F$ whose closure contains the derived subgroup of $F$ satisfies Condition $(2)$ of Theorem \ref{thm:der}. Proposition \ref{main_pro} claims that if the image of $H$ in the abelianization of $F$ is closed then the algorithm from \cite[Section 8]{G} necessarily returns ``Yes''. While we are not going to consider the algorithm itself, we will use the ``setting" of the algorithm from \cite{G} with some modifications. Until Lemma \ref{belongs to S_H}, we follow \cite[Section 8]{G} with small modifications.

\begin{Definition}
	Let $H$ be a subgroup of $F$. We let
	$$\mathcal S_H=\{(a,b)\in\mathbb{Z}^2 \mid \exists h\in H:\exists\alpha\in\mathcal D: h(\alpha)=\alpha,h'(\alpha^-)=2^a,h'(\alpha^+)=2^b
	\}$$
	That is, we denote by $\mathcal S_H$ the subset of $\mathbb{Z}^2$ of all vectors $(a,b)$ such that there is an element $h\in H$ and a finite dyadic fraction $\alpha\in(0,1)$ such that $h$ fixes $\alpha$, and such that the slope $h'(\alpha^-)=2^a$ and the slope $h'(\alpha^+)=2^b$.
\end{Definition}

Recall that if $H$ is a subgroup of $F$ such that $\Cl(H)$ contains the derived subgroup of $F$, then $H$ acts transitively on the set of dyadic fractions $\mathcal D$ (since $\Cl(H)$ acts transitively on $\mathcal D$). Hence, we have the following. 

\begin{Lemma}
	Let $H$ be a subgroup of $F$ such that $\Cl(H)$ contains the derived subgroup of $F$. Then $\mathcal S_H$ is a subgroup of $\mathbb{Z}^2$.
\end{Lemma}

\begin{proof}
	Let $(a_1,b_1),(a_2,b_2)\in\mathcal S_H$. We claim that $(a_1+a_2,b_1+b_2)\in\mathcal S_H$. By assumption, there exist $h_1,h_2\in H$ and $\alpha_1,\alpha_2\in\mathcal D$ such that for $i=1,2$, $h_i(\alpha_i)=\alpha_i$, $h_i'(\alpha_i^-)=2^{a_i}$ and $h_i'(\alpha_i^+)=2^{b_i}$. Since $H$ acts transitively on $\mathcal D$, there is an element $h\in H$ such that $h(\alpha_1)=\alpha_2$. Consider the element $g=h_1^h\in H$. The element $g$ fixes $\alpha_2$ and $g'(\alpha_2^-)=h_1'(\alpha_1^-)=2^{a_1}$, $g'(\alpha_2^+)=h_1'(\alpha_1^+)=2^{b_1}$. Hence the element $k=h_2g\in H$ fixes the dyadic fraction $\alpha_2$, has slope $2^{a_1+a_2}$ at $\alpha_2^-$ and slope $2^{b_1+b_2}$ at $\alpha_2^+$. Hence, $(a_1+a_2,b_1+b_2)\in\mathcal S_H$. 
\end{proof}

Let $u$ be a finite binary word. We define $\ell_0(u)$ to be the length of the longest suffix of zeros of $u$ and $\ell_1(u)$ to be the length of the longest suffix of ones of $u$. Note that for every finite binary word $u$, $\ell_0(u)=0$ or $\ell_1(u)=0$.  We make the following definition.

\begin{Definition}
	Let $(T_+,T_-)$ be a tree-diagram of an element in $F$. 
	Let $u_1$ and $u_2$ be a pair of consecutive branches of $T_+$ and $v_1$ and $v_2$ be the corresponding pair of consecutive  branches of $T_-$, so that $u_1\rightarrow v_1$ and $u_2\rightarrow v_2$ are pairs of branches of $(T_+,T_-)$. 
	Then the \textit{$2$-tuple associated with these consecutive pairs of branches of $(T_+,T_-)$} is defined to be 
	$$t=(\ell_1(u_1)-\ell_1(v_1),\ell_0(u_2)-\ell_0(v_2)).$$
\end{Definition}

\begin{Remark}\label{obvious remark}
	Let $(T_+,T_-)$ be a tree-diagram of an element in $F$ and let $u_1\to v_1$ and $u_2\to v_2$ be two consecutive pairs of branches of $(T_+,T_-)$.
	Let $u$ be the longest common prefix of $u_1$ and $u_2$. Then, $$u_1\equiv u01^{m_1}\ \mbox{ and }\ u_2\equiv u10^{n_1}\ \mbox{ for some } \ m_1,n_1\ge 0.$$ 
	Let $v$ be the longest common prefix of $v_1$ and $v_2$. Then, $$v_1\equiv v01^{m_2}\ \mbox{ and }\ v_2\equiv v10^{n_2}\ \mbox{ for some }\ m_2,n_2\ge 0.$$ 
	Then the $2$-tuple associated with this consecutive pair of branches of $(T_+,T_-)$ is 
	$$t=(m_1-m_2,n_1-n_2).$$
\end{Remark}

\begin{Lemma}\label{belongs to S_H}
	Let $(T_+,T_-)$ be a tree-diagram of an element $h\in H$. Assume that $(T_+,T_-)$ has two consecutive pairs of branches $u_1\to v_1$ and $u_2\to v_2$ and let $t$ be the $2$-tuple associated with these pairs of branches. 	
	Let $u$ be the longest common prefix of $u_1$ and $u_2$ and let $v$ be the longest common prefix of $v_1$ and $v_2$. If there is an element $g\in H$ with the pair of branches $v\to u$ then the tuple $t\in \mathcal S_H$. 
\end{Lemma}

\begin{proof}
	Since $u$ is the longest common prefix of $u_1$ and $u_2$ and they are consecutive branches of the tree $T_+$,
	$$u_1\equiv u01^{m_1}\ \mbox{ and }\ u_2\equiv u10^{n_1}\ \mbox{ for some } \ m_1,n_1\ge 0.$$ 
	Similarly, 
	$$v_1\equiv v01^{m_2}\ \mbox{ and }\ v_2\equiv v10^{n_2}\ \mbox{ for some }\ m_2,n_2\ge 0.$$ 
	As noted in Remark \ref{obvious remark}, the $2$-tuple $t=(m_1-m_2,n_1-n_2)$.
	Consider the element $f=hg\in H$. 
	The element $f$ has the pairs of branches 
	$$u01^{m_1}\to u01^{m_2} \mbox{ and } u10^{n_1}\to u10^{n_2}.$$
	(Indeed, $h$ has the pair of branches $u01^{m_1}\to v01^{m_2}$ and $g$ has the pair of branches $v01^{m_2}\to u01^{m_2}$. Hence, $f=hg$ has the pair of branches $u01^{m_1}\to u01^{m_2}$. Similarly, for the second pair of branches.)
	Let $\alpha=.u01^{\mathbb{N}}=.u1$. 
	The above pairs of branches of $f$ imply that $f$ fixes $\alpha$. In addition, they imply that the slope of $f$ at $\alpha^-$ is $2^{m_1-m_2}$ and the slope of $f$ at $\alpha^+$ is $2^{n_1-n_2}$. 
	Hence, $$t=(m_1-m_2,n_1-n_2)\in \mathcal S_H.$$
\end{proof}

\begin{Lemma}\label{num_caret_left}
	Let $T$ be a finite binary tree with branches $u_1,\dots,u_n$. 
	Then the sum $\sum_{i=1}^{n}\ell_1(u_i)$ is equal to the number of carets in $T$. In other words, 
	$$\sum_{i=1}^{n}\ell_1(u_i)=n-1.$$
\end{Lemma}

\begin{proof}
	The proof is by induction on the number $n$ of leaves of the tree (where a tree with no edges has one leaf). For $n=1$ the claim is obvious. Hence assume that the claim holds for some $n\geq 1$ and let $T$ be a finite binary tree with $n+1$ leaves. Let $u_1,\dots,u_{n+1}$ be the branches of $T$ and let $i\in\{1,\dots,n\}$ be such that the leaves at the end of branches $u_i,u_{i+1}$ have a common father (note such an $i$ must exist). Removing the caret formed by the $i$ and $i+1$ leaves and their father, results in a tree $T'$ with branches $u_1,\dots,u_{i-1},v,u_{i+2},\dots,u_n$, where the finite binary word $v$ is such that $u_i\equiv v0$ and $u_{i+1}\equiv v1$. In particular, $\ell_1(u_i)=0$ and $\ell_1(u_{i+1})=\ell_1(v)+1$. 
	By the induction hypothesis, the sum of $\ell_1(w)$, where $w$ runs over all the branches of $T'$ is $n-1$. Hence, the sum of $\ell_1(w)$ where $w$ runs over all the branches of $T$ is $(n-1)+1=n$ as required. 	
\end{proof}

Similarly, we have the following righ-left analogue. 

\begin{Lemma}\label{num_caret_right}
	Let $T$ be a finite binary tree with branches $u_1,\dots,u_n$. 
	Then the sum $\sum_{i=1}^{n}\ell_0(u_i)$ is equal to the number of carets in $T$. In other words, 
	$$\sum_{i=1}^{n}\ell_0(u_i)=n-1.$$
\end{Lemma}

\begin{Corollary}
	Let $(T_+,T_-)$ be a tree-diagram of an element $h\in F$ with pairs of branches $u_i\rightarrow v_i$, $i=1,\dots,n$. Let $t_i$, $i=1,\dots,n-1$ be the tuples associated with the consecutive pairs of branches of $(T_+,T_-)$.	Then 
	$$\sum_{i=1}^{n-1} t_i=(-\log_2(h'(1^-)),-\log_2(h'(0^+))).$$ 
\end{Corollary}

\begin{proof}
	By definition, for each $i=1,\dots,n-1$, $$t_i=(\ell_1(u_i)-\ell_1(v_i),\ell_0(u_{i+1})-\ell_0(v_{i+1}))$$
	Hence, 
	$$\sum_{i=1}^{n-1}t_i=
	\big(\sum_{i=1}^{n-1}\ell_1(u_i)-\sum_{i=1}^{n-1}\ell_1(v_i), \sum_{i=2}^n \ell_0(u_{i})-\sum_{i=2}^n \ell_0(v_i)  \big).$$
	We note that by Lemma \ref{num_caret_left},
	$$\sum_{i=1}^{n-1}\ell_1(u_i)=\sum_{i=1}^{n}\ell_1(u_i)-\ell_1(u_n)=(n-1)-\ell_1(u_n).$$
	Similarly, by Lemma \ref{num_caret_left},	$$\sum_{i=1}^{n-1}\ell_1(v_i)=(n-1)-\ell_1(v_n).$$
	Hence		
	$$\sum_{i=1}^{n-1}\ell_1(u_i)-\sum_{i=1}^{n-1}\ell_1(v_i)=\ell_1(v_n)-\ell_1(u_n)$$
	Note that $u_n\equiv 1^{\ell_1(u_n)}$ and $v_n\equiv 1^{\ell_1(v_n)}$. Since $h$ has the pair of branches  $u_n\equiv 1^{\ell_1(u_n)}\rightarrow v_n\equiv 1^{\ell_1(v_n)}$, the slope of $h$ at $1$ (from the left) satisfies $$\log_2(h'(1^-))={\ell_1(u_n)}-{\ell_1(v_n)}.$$ Hence, $$\sum_{i=1}^{n-1}\ell_1(u_i)-\sum_{i=1}^{n-1}\ell_1(v_i)=-\log_2(h'(1^-)).$$ 
	Similarly, using Lemma \ref{num_caret_right}, one can get that
	$$\sum_{i=2}^n \ell_0(u_{i})-\sum_{i=2}^n \ell_0(v_i)=-\log_2(h'(0^+)).$$
	Hence, 	
	$$\sum_{i=1}^{n-1} t_i=(-\log_2(h'(1^-)),-\log_2(h'(0^+))),$$
	as required.	
\end{proof}

\begin{Lemma}\label{main lemma}
	Let $(T_+,T_-)$ be a tree-diagram of an element $h\in F$ and let $u_i\to v_i$, $i=1,\dots,n$, be the pairs of branches of $(T_+,T_-)$. Let $t_i$, $i=1,\dots,n-1$, be the tuple associated with the $i$ and $i+1$ pairs of branches of $(T_+,T_-)$.  Assume that there are finite binary words $u$, $w_1$ and $w_2$ and indexes $k<\ell$ in $\{1,\dots,n-1\}$ for which the $k^{th}$ and $\ell^{th}$ pairs of branches of $(T_+,T_-)$ are the pairs
	$$uw_1\rightarrow uw_1 \ \mbox{ and }\ uw_2\rightarrow uw_21,$$
	respectively.
	Then $$\sum_{i=k}^{\ell-1}t_i=(1,0).$$ 
\end{Lemma}

\begin{proof}
	Let $a=.uw_1$ and $b=.uw_21^{\mathbb{N}}$ and note that $h$ fixes the  dyadic fractions $a$ and $b$. Hence, the element 
	\[
	g(t) =
	\begin{cases}
	t &  \hbox{ if }  t\in [0,a]\cup[b,1] \\
	h(t)       & \hbox{ if } t\in[a,b] 
	\end{cases} 
	\]
	Belongs to $F$. Let $(T'_+,T'_-)$ be a tree-diagram of $g$ with pairs of branches  $u_i'\rightarrow v_i'$, $i=1,\dots,m$. We can assume (by passing to an equivalent tree-diagram if necessary) that $m=n$ and that the $k$ to $\ell$ pairs of branches of $(T'_+,T'_-)$ coincide with the $k$ to $\ell$ pairs of branches of $(T_+,T_-)$. Note that all other pairs of branches of $(T'_+,T'_-)$ are pairs of branches of the identity; i.e, pairs of branches of the form $w\rightarrow w$ for some finite binary words $w$. 
	
	We let $t'_i$, 
	 $i=1,\dots,n-1$ be the tuples 
	  associated with the tree-diagram $(T'_+,T'_-)$ and note that for $i=k\dots,\ell-1$, we have $t'_i=t_i$. Indeed, for each $i=k,\dots,\ell-1$ the $i^{th}$ and $(i+1)^{th}$ pairs of branches of $(T_+,T_-)$ and $(T'_+,T'_-)$ coincide. 
	
	We note also that for each $i\in\{1,\dots,k-1\}\cup\{\ell+1,\dots,n-1\}$ we have $t'_i=(0,0)$. Indeed, for each such $i$, the $i^{th}$ and $(i+1)^{th}$ pairs of branches of $(T'_+,T'_-)$ are pairs of branches of the form $w\rightarrow w$ for finite binary words $w$. It follows easily from the definition, that the corresponding tuple is $(0,0)$. 
	
	Finally, 
	we consider the tuple $t'_\ell$
	$$t'_\ell =(\ell_1(u_2w_2)-\ell_1(u_2w_11),\ell_0(u'_{\ell+1})-\ell_0(v'_{\ell+1}))
	=(-1,0)$$
	
	We note that by the previous corollary, since the slope of $g$ at $0^+$ and at $1^-$ is $1$, 
	$$\sum_{i=1}^{n-1}t'_i=(0,0).$$

	Hence, 
	\begin{equation*}
	\begin{split}
	\sum_{i=k}^{\ell-1}t_i&=\sum_{i=k}^{\ell-1}t'_i\\
	&=\sum_{i=1}^{n-1}t'_i-\sum_{i=1}^{k-1}t_i'-t'_\ell-\sum_{i=\ell+1}^{n-1}t_i'\\
	&=(0,0)-(0,0)-(-1,0)-(0,0)=(1,0),
	\end{split}
	\end{equation*}
	as required.
\end{proof}

\begin{Lemma}\label{belong to S_H}
	Let $H$ be a subgroup of $F$. 
	Let $u$ be a finite binary word such that for every finite binary word $w$, there is an element in $H$ with the pair of branches $uw\to u$. 
	Let $h\in H$ be an element supported in the interval $[u01]$. Then for every tuple $t$ associated with consecutive pairs of branches of $h$, we have $t\in\mathcal S_H$. 
\end{Lemma}

\begin{proof}
	Let $u_i\rightarrow v_i$, $i=1,\dots,n$ be the pairs of branches of $h$. Let $i\in\{1,\dots,n-1\}$ and consider the tuple $t$ associated with the $i$ and $i+1$ pairs of branches. We consider two cases:
	
	$(1)$ Both intervals $[u_i]$ and $[u_{i+1}]$ do not intersect the interior of the support of $h$. In that case, $u_i\equiv v_i$ and $u_{i+1}\equiv v_{i+1}$ and $t=(0,0)\in \mathcal S_H$ as required. 
	
	$(2)$ At least one of the intervals $[u_i]$ or $[u_{i+1}]$ intersects the interior of the support of $h$. We consider the case where $[u_i]$, and hence $[v_i]$, intersects the interior of the support of $h$ (the other case being similar). In that case, $u01$ must be a prefix of $u_i$ and of $v_i$. 
	Since $u_i$ and $u_{i+1}$ are consecutive branches in a full finite binary tree and $u01$ is a prefix of $u_i$, the word $u$ must also be a prefix of $u_{i+1}$. Similarly, the word $u$ must be a prefix of $u_{i+1}$ as well. Hence, the word $u$ is a prefix of $u_i,u_{i+1},v_i,v_{i+1}$. 
	Let $w_1$ be such that $uw_1$ is the longest common prefix of $u_i$ and $u_{i+1}$ and let $w_2$ be such that $uw_2$ is the longest common prefix of $v_i$ and $v_{i+1}$. 
	By Assumption, there is an element $h_1\in H$ with the pair of branches $uw_1\to u$ and an element $h_2\in H$ with the pair of branches $uw_2\to u$. Then the element $h_2h_1^{-1}\in H$ has the pair of branches $uw_2\to uw_1$. 
	Hence, by Lemma \ref{belongs to S_H} the tuple $t\in\mathcal \mathcal S_H$. 
\end{proof}

Now, we are ready to give the proof of Proposition \ref{main_pro}.

\begin{proof}[Proof of Proposition \ref{main_pro}]
	Let $H$ be a subgroup of $F$ such that the closure of $H$ contains the derived subgroup of $F$. Assume also that $\pi_{\ab}(H)$ is a closed subgroup of $\mathbb{Z}^2$ and let $p,q\in\mathbb{Z}$ be such that $\pi_{\ab}(H)=p\mathbb{Z}\times q\mathbb{Z}$. We need to prove that $(1,0)\in \mathcal S_H$. Indeed, in that case, Condition (2) of Theorem \ref{thm:der} holds for $H$.  
	
	Since $\Cl(H)$ contains $[F,F]$, by Corollary \ref{cy1}, there exists a finite binary word $u$ such that for every finite binary word $w$, there is an element in $H$ with the pair of branches $u\to uw$. 
	In particular, there is an element $h\in H$ with the pair of branches $u\to u1$. 
	Let $\alpha=.u1^{\mathbb{N}}$ and note that $h$ fixes $\alpha$ and that $h'(\alpha^-)=\frac{1}{2}$. 
	
	Since $\pi_{\ab}(H)=p\mathbb{Z}\times q\mathbb{Z}$, there is an element $f_1\in H$ such that $\pi_{\ab}(f_1)=(p,0)$. Note that $f_1'(0^+)=2^p$ and that $f_1$ fixes pointwise a right neighborhood of $1$. That is, for some  dyadic number $c\in (0,1)$, the element $f_1$ fixes the interval $[c,1]$ pointwise. Since $\Cl(H)\supseteq [F,F]$, the subgroup $H$ acts transitively on the set of  dyadic fractions $\mathcal D$. Hence, conjugating $f_1$ by an element of $H$ if necessary, we can assume that $c<\alpha$. 
	
	Recall that the element $h$ fixes $\alpha$ and has slope $\frac{1}{2}$ at $\alpha^-$.
	Let ${m_1}$ be such that $h'(0^+)=2^{m_1}$ and note that $p|m_1$. Hence, the element $$h_1=hf_1^{-\frac{m_1}{p}}\in H.$$ Note that $h_1$ has slope $1$ at $0^+$ and as such, it fixes a right neighborhood of $0$. It also coincides with $h$ on $[c,1]$. In particular, $h_1$ fixes $\alpha$ and has slope $\frac{1}{2}$ at $\alpha^-$.

	Let $m_2$ be the slope of $h_1$ at $1^-$. 	
	Since $\pi_{\ab}(H)=p\mathbb{Z}\times q\mathbb{Z}$, $q|m_2$. In addition, there must exist an element $f_2\in H$ which fixes a right neighborhood of $0$ and has slope $2^q$ at $1^{-}$. Using conjugation if necessary, we can assume that $f_2$ fixes the interval $[0,d]$ for some $d>\alpha$. 
	Now, consider the element  $$h_2=h_1f_2^{-\frac{m_2}{q}}\in H.$$
	The element $h_2$ fixes some left neighborhood of $1$. Since it coincides with $h_1$ on $[0,d]$, it also fixes a right neighborhood of $0$. In addition $h_2(\alpha)=\alpha$ and $h_2'(\alpha^-)=\frac{1}{2}$. 
	
	Let $a<b$ be dyadic fractions in $(0,1)$ such that $h_2$ is supported in the interval $[a,b]$. 
	By Lemma \ref{lm2} there is an element $g\in H$ such that $g([a,b])\subseteq [u01]$. We consider the element $f=h_2^g$. First, we note that $f$ belongs to $H$ and is supported in $[u01]$. In addition, for $\beta=g(\alpha)$ we have that $f$ fixes $\beta$ and $f'(\beta^-)=\frac{1}{2}$.

	Let $(T_+,T_-)$ be the reduced tree diagram of $f$. Since $f$ is supported in the interval $[u01]$, by Lemma \ref{belong to S_H}, every tuple $t$ associated with the tree-diagram $(T_+,T_-)$ belongs to $\mathcal S_H$.  
	
	Since $f$ fixes the interval $[u00]$ pointwise but does not fix the interval $[u01]$ pointwise, the tree-diagram $(T_+,T_-)$ must have the pair of branches $u00\to u00$. Recall that $f(\beta)=\beta$ and let $v$ be the finite binary word such that $\beta=.v1=.v01^\mathbb{N}$. Since $f$ fixes $\beta$ and $f'(\beta^-)=\frac{1}{2}$, by Lemma \ref{4parts}, the tree-diagram $(T_+,T_-)$ must have a pair of branches of the form $v01^n\to v01^{n+1}$ for some $n\geq 0$. Since $\beta=.v1$ is in $[u01]$, the word $u01$ must be a prefix of $v1$. Since $\beta\neq.u01$ (because $f$ does not fix a left neighborhood of $\beta$), the word $u01$ must be a strict prefix of $v1$, and as such, it is a prefix of $v$. Let $w_3$ be such that $v\equiv u01w_3$. Then for $w_4\equiv w_301^n$, the tree-diagram $(T_+,T_-)$ has the pair of branches $u01w_4\rightarrow u01w_41$. 
	In particular, for some $k<\ell$ the $k$ and $\ell$ pairs of branches of $(T_+,T_-)$ are 
	$$(*) u00\rightarrow u00\ \mbox{ and }\ u01w_4\rightarrow u01w_41,$$
	respectively.
	
	Hence, by Lemma \ref{main lemma}, and the fact that every tuple associated with $(T_+,T_-)$ belongs to the additive group $\mathcal S_H$ it follows from $(*)$ that $(1,0)\in \mathcal S_H$, as necessary. That completes the proof of the proposition.
\end{proof}

Proposition \ref{main_pro} implies the following. 

\begin{Theorem}\label{main}
	Let $H$ be a subgroup of $F$, which satisfies the following conditions. \begin{enumerate}
		\item[$(1)$] The image of $H$ in the abelianization of $F$ is closed. 
		\item[$(2)$] $\Cl(H)$ contains the derived subgroup of $F$.  
	\end{enumerate}
	Then $H=\Cl(H)=H[F,F]$. 
\end{Theorem}

\begin{proof}
	Let $H$ be a subgroup of $F$ which satisfies Conditions $(1)$ and $(2)$ from the theorem. Then by Theorem \ref{thm:der} and Proposition \ref{main_pro}, the subgroup $H$ contains the derived subgroup of $F$. Hence $H[F,F]=H$. 
	Similarly, $\Cl(H)[F,F]=\Cl(H)$. 
	Thus, it suffices to prove that $H[F,F]=\Cl(H)[F,F]$, or equivalently, that the image of $H$ in the abelianization of $F$ coincides with the image of $\Cl(H)$. It is clear that the image of $H$ is contained in the image of $\Cl(H)$. In the other direction, let $p,q\geq 0$ be such that the image of $H$ in the abelianization of $F$ is $p\mathbb{Z}\times q\mathbb{Z}$. Let $f\in \Cl(H)$. It suffices to prove that its image in the abelianization of $F$ belongs to $p\mathbb{Z}\times q\mathbb{Z}$. But since $f\in \Cl(H)$, it is a piecewise-$H$ function. Hence, its slope at $0^+$ coincides with the slope at $0^+$ of some element $h_1\in H$ and its slope at $1^-$ coincides with the slope at $1^-$ of some element $h_2\in H$. Since $h_1,h_2\in H$, we have that $\log_2 f'(0^+)=\log_2 h_1'(0^+)$ is an integer multiple of $p$ and $\log_2 f'(1^-)=\log_2 h_2(1^-)$ is an integer multiple of $q$. Hence, the image of $f$ in the abelianization of $F$ is in $p\mathbb{Z}\times q\mathbb{Z}$, as required. 
\end{proof}

As a corollary from Theorem \ref{main}, we get the following. 

\begin{Corollary}\label{main cor}
	Let $H$ be a subgroup of $F$. Then $H=F$ if and only if the following conditions hold. 
	\begin{enumerate}
		\item[$(1)$] $H[F,F]=F$.
		\item[$(2)$] $[F,F]\subseteq \Cl(H)$.
	\end{enumerate}
\end{Corollary}

\begin{proof}
	One direction is obvious. In the other direction, note that if $H[F,F]=F$, then the image of $H$ in the abelianization of $F$ is $\mathbb{Z}^2$, and in particular it is a closed subgroup. Hence, if Conditions $(1)$ and $(2)$ hold for $H$, then by Theorem \ref{main}, $H=H[F,F]=F$. 
\end{proof}

Note that Corollary \ref{main cor} gives a simple solution for the generation problem in $F$. Indeed, given a finite set $X$ of elements in $F$, to determine if $X$ generates $F$ one has to $(1)$
find the image of $X$ in the abelianization of $F$ and check whether it generates $\mathbb{Z}^2$ and $(2)$ construct the core of the subgroup $H$ generated by $X$ and use Lemma \ref{core of derived subgroup} to check if $\Cl(H)$ contains $[F,F]$.
This gives us a linear-time algorithm in the sum of sizes of elements in $X$ (where the \emph{size} of an element in $X$ is the number of carets in its reduced tree-diagram).

Corollary \ref{main cor} implies that every maximal subgroup of $F$ which has infinite index in $F$ is closed. 

\begin{Corollary}\label{max closed}
	Let $M$ be a maximal subgroup of $F$ and assume that $M$ has infinite index in $F$. Then $M$ is a closed subgroup of $F$.  
\end{Corollary}

\begin{proof}
	Assume by contradiction that $M$ is not closed. Then $\Cl(M)$ strictly contains $M$. Since $M$ is maximal, $\Cl(M)=F$. Since $M$ is maximal and has infinite index in $F$, it is not contained in any proper finite index subgroup of $F$. Hence, by Remark \ref{rem:H[F,F]}, $M[F,F]=F$. 
	Therefore, $M$ satisfies Conditions $(1)$ and $(2)$ of  Corollary \ref{main cor} which implies that $M=F$, in contradiction to $M$ being a maximal subgroup of $F$. 
\end{proof}

In Section \ref{sec:max}, we derive more results regarding maximal subgroups of $F$. But first, in Section \ref{sec:morph}, we study tree-automata and morphisms between tree-automata. 

\section{Reduced tree-automata and morphisms of tree-automata}\label{sec:morph}

Let $\mathcal T$ be a finite or infinite planar binary tree. Then $\mathcal T$ can be naturally viewed as a rooted tree-automaton, where the root of the automaton is the root of $\mathcal T$, all edges are directed away from the root and every left edge is labeled ``0'' while every right edge is labeled ``1''. Below, we will often consider a binary tree $\mathcal T$ as a rooted tree-automaton  without explicitly saying so. Clearly, the diagram group $\ddd(\mathcal T)$ is the trivial subgroup of $F$.

\begin{Definition}
	Let $\A_r$ be a rooted tree-automaton. An \emph{extention} of $\A_r$ is a rooted tree-automaton obtained from $\A_r$ as follows. Let $\mathcal L=\{\ell_i\mid i\in\mathcal I\}$ be the set of leaves of $\A_r$. For each $i\in\mathcal I$, let $\mathcal T_i$ be a finite or infinite binary tree viewed as a rooted tree-automaton. Then the rooted tree-automaton $\A_r'$ obtained from $\A_r$ by identifying the root of $\mathcal T_i$ with the  leaf $\ell_i$ of $\A_r$, for each $i$, is called an extension of $\A_r$.   
\end{Definition}

It is easy to check that if $\A_r$ is a rooted tree-automaton and $\A_r'$ is an extension of $\A_r$ then $\A_r$ and $\A_r'$ accept the same reduced tree-diagrams in $F$.

A rooted tree-automaton is said to be \emph{full} if it has no leaves. In that case, every finite binary word $u$ labels a unique path in $\A_r$. Note that every rooted tree-automaton can be extended to a full rooted tree-automaton, by attaching a distinct copy of the complete infinite binary tree $\mathcal T$ to each leaf of $\A_r$. 

If $\A_r'$ is an extension of $\A_r$ we will also say that $\A_r$ is a reduction of $\A_r'$. A tree-automaton is said to be \emph{reduced} if it has no reduction other than itself. In other words, $\A_r$ is reduced, if it is not the extension of any tree-automaton other than itself. 

Let $\A_r$ be a rooted tree-automaton and let $x$ and $y$ be vertices of $\A_r$. We say that $y$ is a descendant of $x$ if there is a non-empty  trail in  $\A_r$ with initial vertex $x$ and terminal vertex $y$. Notice that in a rooted tree-automaton, it is possible for two vertices to be descendants of each other and for a vertex to be a descendant of itself. We make the following observation. 

\begin{Lemma}\label{reduced aut}
	Let $\A_r$ be a rooted tree-automaton. Then $\A_r$ is reduced if and only if every vertex $x$ in $\A_r$ satisfies at least one of the following conditions. 
	\begin{enumerate}
		\item[$(1)$] $x$ is a leaf (that is, $x$ has no outgoing edges). 
		\item[$(2)$] $x$ is a descendant of itself.
		\item[$(3)$] $x$ has a descendant $y$ which has two distinct incoming edges in $\A_r$.
	\end{enumerate}
\end{Lemma}

\begin{proof}
	Assume first that $\A_r$ is a tree-automaton such that every vertex in $\A_r$ satisfies at least one of Conditions $(1)$-(3). We claim that $\A_r$ must be reduced. Indeed, if $\A_r$ is not reduced then it is the extension of some tree-automaton other than itself. In that case, 
	there exist
	some tree-automaton $\mathcal B_r$ which has a leaf $\ell$ and a non-empty binary tree $\mathcal T$, such that $\mathcal A_r$ can be obtained from $\mathcal B_r$ by identifying the root of 
	 $\mathcal T$ with the leaf $\ell$. 
	 Then, the vertex $\ell$, viewed as a vertex of $\A_r$, does not satisfy any of Conditions $(1)$-(3). Indeed, in $\A_r$ the vertex $\ell$ is not a leaf and thus does not satisfy Condition $(1)$. In addition, the descendants of $\ell$ in $\A_r$ are the vertices of $\mathcal T$, other than its root. As each of them has a unique incoming edge in $\A_r$ and none of them coincides with $\ell$, the vertex $\ell$ does not satisfy Conditions $(2)$ and (3). Hence $\ell$ does not satisfy any of the three conditions in the lemma, in contradiction to the assumption.

	In the other direction, assume that $\A_r$ is reduced, but that there is a vertex $x$ in $\A_r$ which does not satisfy any of the conditions in the lemma. Let $\mathcal T_x$ 
 be the	tree-automaton with root $x$ obtained from $\A_r$ as follows.	
	The vertex set of $\mathcal T_x$ consists of the vertex $x$ as well as all of its descendants in $\A_r$. The labeled directed edges in $\mathcal T_x$ are the edges of $\A_r$ whose end-vertices belong to the vertex set of $\mathcal T_x$. It is easy to check that $\mathcal T_x$ is a tree-automaton with root $x$. We claim that $\mathcal T_x$ is a binary tree. Indeed, its root $x$ has no incoming edges in $\mathcal T_x$ as it is not a descendant of itself in $\A_r$. In addition, every other vertex in $\mathcal T_x$ has exactly one incoming edge (as it is a descendant of $x$ it has an incoming edge and by assumption, it cannot have more than one). Every non-leaf vertex in $\mathcal T_x$ also has two outgoing edges, since it has two outgoing edges in $\A_r$ and their end-vertices are  clearly vertices in $\mathcal T_x$. Hence, $\mathcal T_x$ is a rooted binary-tree.
	Now, let $\mathcal C_r$ be the rooted tree-automaton obtained from $\A_r$ be removing from $\A_r$ all the edges of $\mathcal T_x$ as well as all of the descendants of $x$. One can verify that $\mathcal C_r$ is a rooted tree-automaton with root $r$. Clearly, the vertex $x$ is a leaf of $\mathcal C_r$. Note also that $\A_r$ is an extension of $\mathcal C_r$, where the leaf $x$ of $\mathcal C_r$ is identified with the root of $\mathcal T_x$. Hence, $\A_r$ is not reduced, in contradiction to the assumption. 	
\end{proof}

Note that Lemma \ref{reduced aut} can be simplified in the case where the tree-automaton $\A_r$ is such that the root $r$ is not a descendant of itself (as is always the case in the core of a subgroup of $F$). Indeed, we have the following.

\begin{Remark}\label{simplify}
	Let $\A_r$ be a rooted tree-automaton and assume that the root $r$ is not a descendant of itself. 
	Let $x$ be a vertex of $\A_r$ and assume that $x$ is a descendant of itself. Then $x$ has a descendant $y$ which has two distinct incoming edges in $\A_r$. 
\end{Remark}

\begin{proof}
	Let $x$ be a vertex in $\A_r$ which is a descendant of itself. Let $\mathcal B$ be the set of all vertices in $\mathcal A_r$ which are descendants of $x$. 
	Note that $\mathcal B$ is not empty since $x\in\mathcal B$ and that $r\notin B$. Note also that every vertex in $\mathcal B$ has a father which also belongs to $\mathcal B$.  For every vertex in $\mathcal B$ there is at least one path in $\A_r$ from the root to the vertex. Let $u$ be a finite binary word of minimal length which labels a path in $\A_r$ which starts from the root and terminates in a vertex belonging to $\mathcal B$. Clearly, $u$ is not empty. Let $v$ be a finite binary word such that $u\equiv va$ for a letter $a\in\{0,1\}$  and consider the vertex $v^+$. From the minimality of $u$ it follows that the vertex $v^+$ is not a descendant of $x$.  Hence, the vertex $u^+$ has a father $v^+$ which does not belong to $\mathcal B$. Since $u^+\in\mathcal B$ it also has a father which belongs to $\mathcal B$. Hence $u^+$ (which is a descendant of $x$) has at least two distinct fathers in $\A_r$, and in particular, at least two distinct incoming edges. 
\end{proof}

Lemma \ref{reduced aut} and Remark \ref{simplify} imply the following.

\begin{Corollary}\label{cor:red aut}
	Let $\A_r$ be a tree-automaton such that the root $r$ is not a descendant of itself.  Then $\A_r$ is reduced if and only if every inner vertex $x$ in $\A_r$ has a descendant $y$ with at least two distinct incoming edges. 
\end{Corollary}

It follows from \cite[Lemma 10.9]{G} that if $H$ is a subgroup of $F$ then every inner vertex of the core $\mathcal C(H)$  has a descendant with two distinct incoming edges. Hence, we have the following. 

\begin{Lemma}\label{lem:core reduced}
	Let $H$ be a subgroup of $F$. Then the core of $H$ is a reduced tree-automaton.
\end{Lemma}

Let $H$ be a subgroup of $F$. Since the core of $H$ accepts the subgroup $H$, it accepts any reduced tree-diagram in $H$. In particular, for every reduced tree-diagram in $H$, all the branches in the tree-diagram are readable on $\mathcal C(H)$. The fact that $\mathcal C(H)$ is reduced implies that the only finite binary words readable on $\mathcal C(H)$ are branches of reduced tree-diagrams in $H$ and prefixes of such branches.
 Indeed, we have the following.

\begin{Lemma}\label{lem:minimal core}
	Let $H$ be a subgroup of $F$ and let $\mathcal C(H)$ be the core of $H$. Let $u$ be a finite binary word. Then $u$ labels a path in $\mathcal C(H)$ if and only if there is an element in $H$
	with reduced tree-diagram $(T_+,T_-)$ such that $u$ is the prefix of some branch of $T_+$. 
\end{Lemma}

\begin{proof}
	If there exists an element in $H$ with reduced tree-diagram $(T_+,T_-)$ such that $u$ is a prefix of some branch $v$ of $T_+$, then the branch $v$ (and in particular, it's prefix $u$) is readable on $\mathcal C(H)$.

	In the other direction, let $u$ be a finite binary word which labels a path in the core of $H$. We can assume that $u$ is non-empty.  	Let $u_1$ be the prefix of $u$   such that $u\equiv u_1a$, where $a\in\{0,1\}$.
	Note, that it suffices to prove that there is an element in $H$ with  reduced tree-diagram $(T_+,T_-)$ such that $u_1$ is a strict prefix of some branch of $T_+$. Indeed, in that case, both $u_10$ and $u_11$ are prefixes of  branches of $T_+$. Let $x$ be the vertex $u_1^+$ of $\mathcal C(H)$ and note that $x$ is an inner vertex of $\mathcal C(H)$. Since the core $\mathcal C(H)$ is reduced,   the vertex $x$ has a descendant $y$ with two distinct incoming edges.
	Since $y$ is a descendant of $x$ there is a non-empty finite binary word $w$ which labels a trail from $x$ to $y$. Then $u_1w$ labels a path in the core such that $(u_1w)^+=y$. Let $w_1$ be the prefix of $w$ such that  $w\equiv w_1b_1$ for a letter $b_1\in\{0,1\}$ and let $v_1\equiv u_1w_1$. Then $v_1$ labels a path in the core which terminates in a vertex $z_1=v_1^+$. Note that the vertex $z_1$ is a father of the vertex $y$ and that there is an edge $e_1$ labeled $b_1$ from the vertex $z_1$ to the vertex $y$. 
	
	Since $y$ has at least two distinct incoming edges in $\A_r$, it has an incoming edge $e_2\neq e_1$ in $\A_r$. Let $b_2$ be the label of $e_2$, let $z_2={e_2}_-$ and let $v_2$ be a path in the core $\mathcal C(H)$ such that $v_2^+=z_2$. Note that $v_2b_2$ labels a path in $\A_r$ such that $(v_2b_2)^+=y$. Since $v_1b_1$ also labels a path in $\A_r$ such that $(v_1b_1)^+=y$, by Lemma \ref{identified vertices H},	
 there exists $k\in\mathbb{N}$ such that for every finite binary word $w$ of length $\geq k$, there is an element in $H$ with the pair of branches $v_1b_1w\to v_2b_2w$. 
	In particular, for $w\equiv 0^k$, there is an element $h\in H$ which has the pair of branches $v_1b_10^k\to v_2b_20^k$. Hence, $h$ has a (not necessarily reduced) tree-diagram 
	which has the pair of branches $v_1b_10^k\to v_2b_20^k$. 
	Let $(T_+,T_-)$ be the reduced tree-diagram of $h$. 
	By Remark \ref{rem:branches}, there are finite binary words $p,q,s$ such that  $v_1b_10^k\equiv ps$ and $v_2b_20^k\equiv qs$ and such that $p\to q$ is a pair of branches of $(T_+,T_-)$.  
	We claim that the word $s$ is of length at most $k$. In other words, we claim that $s$ is a suffix of $0^k$. Indeed, assume by contradiction that the length of $s$ is greater than $k$. Then $b_10^k$ and $b_20^k$ are suffixes of $s$ and in particular $b_1\equiv b_2$. 
		In that case, the  vertices $z_1={e_1}_-$ and $z_2={e_2}_-$ must be distinct (otherwise, the edges $e_1$ and $e_2$ coincide since they have the same initial vertex and the same label). 
Now, since $b_10^k$ is a suffix of $s$, there exists a finite binary word $s_1$ such that 
	   $s\equiv s_1b_10^k$. Note that $$v_1b_10^k\equiv ps\equiv ps_1b_10^k \mbox{ and } v_2b_20^k\equiv qs\equiv qs_1b_20^k.$$
	Hence, $ps_1\equiv v_1$ and $qs_1\equiv v_2$. By Assumption, the reduced tree-diagram of $h$ has the pair of branches $p\to q$. Hence, since $H$ is accepted by $\mathcal C(H)$, we have $p^+=q^+$ in $\mathcal C(H)$. But that implies that $(ps_1)^+=(qs_1)^+$ in the core (indeed, the word $s_1$ labels a unique trail in the core with initial vertex $p^+=q^+$). Hence, $v_1^+=v_2^+$, in contradiction to $v_1^+=z_1$ and $v_2^+=z_2$ being distinct vertices of the core. Hence, the suffix $s$ is a suffix of $0^k$. Since $ps\equiv v_1b_10^k$, we get that $v_1b_1$ is a prefix of $p$. Recall that $u_1$ is a prefix of $v_1$. Hence, $u_1$ is  a strict prefix of $p$ which is a branch of the tree $T_+$ of the reduced tree-diagram $(T_+,T_-)$ of $h\in H$, as required. 
\end{proof}

Intuitively, Lemma \ref{lem:minimal core} says that the only finite binary words readable on $\mathcal C(H)$ are those that must be readable on the core, for it to accept the subgroup $H$.

\begin{Definition}
	Let $\A_r$ and $\A'_{s}$ be two rooted tree-automata. A \textit{morphism of rooted tree-automata} from $\A_r$ to $\A'_{s}$ is a mapping from $\A_r$ to $\A'_{s}$ which maps the root $r$ of $\A_r$ to the root $s$ of $\A'_{s}$, maps each vertex of $\A_r$ to a vertex of $\A'_{s}$ and each edge of $\A_r$ to an edge of $\A'_{s}$, while preserving adjacency of vertices and edges as well as the labels and direction of the edges. 
\end{Definition}

\begin{Lemma}\label{lem:obvious}
	Let $\A_r$ and $\A_s'$ be rooted tree-automata. Then there is a morphism of rooted tree-automata from $\A_r$ to $\A_s'$ if and only if the following conditions hold. 
	\begin{enumerate}
		\item[$(1)$] Every finite binary word $u$ readable on $\A_r$ is also readable on $\A_s'$. 
		\item[$(2)$] If $u$ and $v$ are readable on $\A_r$ such that $u^+=v^+$ in $\A_r$ then ($u$ and $v$ are readable on $\A_s'$ and) $u^+=v^+$ in $\A_{s}'$ as well. 
	\end{enumerate}
If Conditions $(1)$ and $(2)$ hold, then there is a unique morphism $\phi$ from $\A_r$ to $\A_s'$. In addition, the morphism $\phi$ is surjective if and only if every finite binary word readable on  $\A_s'$ is also readable on $\A_r$.
\end{Lemma}

\begin{proof}
	Assume that there is a morphism $\phi$ from $\A_r$ to $\A_s'$.
	We claim that Conditions $(1)$ and $(2)$ hold. Let $u$ be a word readable on $\A_r$. Then $u$ labels a path $e_1,\dots,e_n$ in $\A_r$. The morphism $\phi$ maps the directed edges $e_1,\dots,e_n$ to directed edges $e_1',\dots,e_n'$ in $\A_s'$. Note that since $\phi$ is a morphism, the directed edges $e_1',\dots,e_n'$ form a path in $\A_s'$ whose label is $u$. Hence, $u$ is readable on $\A_s'$ and Condition $(1)$ holds. 
	
	Now assume that $u$ and $v$ are finite binary words readable on $\A_r$ such that $u^+=v^+$ on $\A_r$. By Condition $(1)$, $u$ and $v$ are also readable on $\A_s'$. Since $\phi$ is a morphism, it must map the end vertex of the path $u$ (resp. $v$) in $\A_r$ to the end vertex of the path $u$ (resp. $v$) in $\A_s'$. Hence, since $u^+=v^+$ in $\A_r$ and $\phi$ is well-defined, we have that $u^+=v^+$ in $\A_s'$ and Condition $(2)$ holds. 
	
	Now assume that Conditions $(1)$ and $(2)$ hold. We claim that there is a  morphism $\phi$ from $\A_r$ to $\A_s'$. Indeed, one can define the action of  $\phi$ on vertices of $\A_r$ as follows. Let $x$ be a vertex of $\A_r$. Then there exists a path $u$ in $\A_r$ such that $u^+=x$. By Condition $(1)$, the word $u$ also labels a path in $\A_s'$. Hence, we can define $\phi(x)$ to be the vertex $u^+$ of $\A_s'$. Condition $(2)$ guarantees that the action of $\phi$ on vertices is well defined (note that this definition also guarantees that the root $r=\emptyset^+$ of $\A_r$  is mapped to the root $s=\emptyset^+$ of $\A_s'$). Next, we define the action of $\phi$ on edges. Let $e$ be an edge of $\A_r$ and let $b$ be its label. The end-vertices $e_-$ and $e_+$ of $e$ are mapped by $\phi$ onto vertices $\phi(e_-)$ and $\phi(e_+)$. We claim that in $\A_s'$ there is a (necessarily unique) directed edge from $\phi(e_-)$ to $\phi(e_+)$ labeled $b$. Indeed, 
	 let $w$ be a path in $\A_r$ such that $w^+=e_-$ and note that $wb$ labels a path in $\A_r$ such that $(wb)^+=e_+$. By the definition of the action of $\phi$ on vertices of $\A_r$, in the tree-automaton $\A_s'$ we have $w^+=\phi(e_-)$ and $(wb)^+=\phi(e_+)$. Hence, $b$ labels a directed edge $e'$ from $\phi(e_-)$  to $\phi(e_+)$ in $\A_s'$. We define $\phi(e)$ to be $e'$. Clearly, $\phi$ preserves adjacency of edges and vertices, as well as the direction and label of edges. Hence, $\phi$ is a morphism of rooted tree-automata. It is easy to see that $\phi$ is the unique morphism from $\A_r$ to $\A_s'$. Indeed,  if $\phi_1$ is a morphism from $\A_r$ to $\A_s'$, it must map a vertex $u^+$ of $\A_r$ to the vertex $u^+$ of $\A_s'$. Hence, its action on vertices of $\A_r$ coincides with the action of $\phi$. But the action of a morphism on vertices of $\A_r$, determines uniquely its action on the edges of $\A_r$ and thus, $\phi$ and $\phi_1$ must coincide.

	  Finally, we note that if every finite binary word readable on $\A_s'$ is readable on $\A_r$,  then the morphism $\phi$ is surjective. Indeed, to show that $\phi$ is surjective on vertices, let $y$ be a vertex of $\A_s'$. Then there is a path $u$ is $\A_s'$ such that $u^+=y$. By assumption, the word $u$ is readable on $\A_r$. Then, by definition, $\phi$ maps the vertex $u^+$ of $\A_r$ onto the vertex $y$ of $\A_r$. Similarly, let $e'$ be an edge of $\A_s'$ labeled $b$. Let $v$ be a path in $\A_s'$ such that $v^+=e'_-$ and note that $vb$ labels a path in $\A_s'$ such that $(vb)^+=e_+$. Then, $vb$ also labels a path in $\A_r$. Let $e$ be the last edge in that path. Then $e$ is labeled $b$, its initial vertex is the vertex $v^+$ and its terminal vertex is the vertex $(vb)^+$. The morphism $\phi$ maps $e$ onto $e'$. In the other direction, assume that $\phi$ is surjective and let $u$ be a finite binary word readable on $\A_s'$. Then $u$ labels a path $e_1',\dots,e_n'$ in $\A_s'$. Surjectivity implies that the path is the image of a path in $\A_r$ with the same label. Hence, $u$ is readable on $\A_r$ as well. 
\end{proof}

\begin{Lemma}\label{lem: morphism from core}
	Let $H$ be a subgroup of $F$ and let $\mathcal A_r$ be a rooted tree-automaton which accepts $H$. Then the following assertions hold. 
	\begin{enumerate}
		\item[$(1)$] There is a unique morphism of rooted tree-automata from $\mathcal C(H)$ to $\A_r$. 
		\item[$(2)$] If the core of $H$ has no leaves, then the unique morphism from $\mathcal C(H)$ to $\A_r$ is surjective.
	\end{enumerate}  
\end{Lemma}

\begin{proof}
	$(1)$ We claim that Conditions $(1)$ and $(2)$ from Lemma \ref{lem:obvious} hold for the tree-automata $\mathcal C(H)$ and $\A_r$. 
	Assume that $u$ labels a path in the core $\mathcal C(H)$. Then by Lemma \ref{lem:minimal core}, there is a reduced tree-diagram $(T_+,T_-)$ of an element in $H$ such that $u$ is a prefix of some branch of $T_+$. Since the automaton $\A_r$ accepts $H$, the tree $T_+$ is readable on $\A_r$ and as such, $u$ labels a directed path in $\A_r$. Hence, Condition $(1)$ from Lemma \ref{lem:obvious} holds. Now, let $u$ and $v$ be two finite binary words readable on $\mathcal C(H)$ such that $u^+=v^+$ in $\mathcal C(H)$. Hence, by Lemma \ref{identified vertices in the core}, there is an element $h$ in $\Cl(H)$ with the pair of branches $u\to v$. We claim that $h\in\ddd(\A_r)$. Indeed, since $\A_r$ accepts $H$, the diagram group $\ddd(\A_r)$ is a closed subgroup of $F$ which contains $H$. Hence,  $\ddd(\A_r)$ contains the closure of $H$ and in particular, the element $h$. Then, by Lemma \ref{lem:paths in automaton}, since $h$ has the pair of branches $u\to v$ and $u$ and $v$ label paths in $\A_r$ (since Condition $(1)$ holds) we have $u^+=v^+$ in $\A_r$. Hence, Condition $(2)$ from Lemma \ref{lem:obvious} holds as well. Therefore, by Lemma \ref{lem:obvious}, there is a unique morphism from $\mathcal C(H)$ to $\A_r$. 
	
	$(2)$ This follows immediately from Lemma \ref{lem:obvious}, since by assumption, every finite binary word is readable on $\C(H)$.
\end{proof}

\begin{Corollary}\label{cor:no-leaves}
	Let $H$ and $G$ be subgroups of $F$ such that $H$ is contained in $G$. Then the following assertions hold. 
	\begin{enumerate}
		\item There is a unique morphism from $\mathcal C(H)$ to $\mathcal C(G)$. 
		\item If there are no leaves in $\mathcal C(H)$, then the unique morphism from $\mathcal C(H)$ to $\mathcal C(G)$ is surjective.  
	\end{enumerate}
\end{Corollary}

\section{Maximal subgroups of Thompson's group $F$}\label{sec:max}

Let $u$ be a finite binary word. Recall that $F_{[u]}$ is the subgroup of $F$ of all functions supported in the interval $[u]$. The subgroup $F$ is isomorphic to $F_{[u]}$ and there is a natural isomorphism from $F$ to $F_{[u]}$ mapping each element $g\in F$ to the $[u]$-copy of $g$, denoted $g_{[u]}$ (see Section \ref{sec:copies}). We will need the following simple remark. 

\begin{Remark}\label{rem:obvious}
	Let $u$ be a finite binary word readable on a rooted tree-automaton $\A_r$. Let $T_u$ be the minimal finite binary tree with branch $u$. Then $T_u$ is readable on $\A_r$. 
\end{Remark} 

\begin{proof}
	Let $v$ be a branch of $T_u$. It suffices to prove that $v$ is readable on $T_u$. If $v\equiv u$, we are done. Otherwise, the word $v'$ obtained from $v$ by changing its last letter is a prefix of $u$ and as such, readable on $\A_r$. But that implies that $v$ is also readable on $\A_r$ (indeed, if two words differ only in their last letter and one of them is readable on $\A_r$, the other one is also readable on $\A_r$.)
\end{proof}

\begin{Lemma}\label{Core of H,G}
	Let $H$ and $G$ be subgroups of $F$. Assume that in the core of $H$ there is a leaf $\ell$ and let $u$ be a finite binary word which labels a path in the core $\mathcal C(H)$ such that $u^+=\ell$. Let $G$ be a subgroup of $F$ and consider its $[u]$-copy $G_{[u]}$. Let $K$ be the subgroup of $F$ generated by $H$ and $G_{[u]}$. Then the core of the subgroup $K$ can be obtained from the core $\mathcal C(H)$ and the core $\mathcal C(G)$ by identifying the  root of $\mathcal C(G)$ with the leaf $\ell$ of $\mathcal C(H)$ (where the root of the obtained automaton is taken to be the root of $\mathcal C(H)$).
\end{Lemma}

\begin{proof}
	Let $\mathcal A_r$ be the rooted tree-automaton obtained by identifying the root of $\mathcal C(G)$ with the leaf $\ell$ of $\mathcal C(H)$ (where the root of $\A_r$ is taken to be the root of $\mathcal C(H)$). It suffices to prove that $\A_r$ is isomorphic to the core $\mathcal C(K)$. 
	
	First, note that the automaton $\A_r$ accepts every element accepted by $\mathcal C(H)$, and as such, every element of $H$. We claim that it also accepts every element in $G_{[u]}$. Indeed, let $g\in G$ and consider its $[u]$-copy $g_{[u]}\in G_{[u]}$. We claim that $g_{[u]}$ is accepted by $\A_r$.  Indeed, let  $(T_+,T_-)$ be the reduced tree-diagram of $g$ and note that $(T_+,T_-)$ is accepted by $\mathcal C(G)$. 
	Let $u_i\to v_i$, $i=1,\dots,n$ be the pairs of branches of $(T_+,T_-)$, and let $T_u$ be the minimal finite binary tree with branch $u$. Note that by Remark \ref{rem:obvious}, the tree $T_u$ is readable on $\mathcal C(H)$ and as such, on $\A_r$.
	Now, the pairs of branches of the reduced tree-diagram of $g_{[u]}$ are pairs of branches of the form $b\to b$ for every branch $b$ of $T_u$, other than the branch $u$, as well as the pairs of branches $uu_i\to uv_i$ for $i=1,\dots,n$. 	
	It suffices to prove that for each of these pairs of branches, the branches in the pair label paths in $\A_r$ which terminate on the same vertex. For the branches of the form $b\to b$, where $b$ is a branch of $T_u$ distinct from $u$, this is clear. For each $i=1,\dots,n$, it follows from the construction of $\A_r$ that $uu_i$ and $uv_i$ label paths in $\A_r$ whose end vertex is the vertex $u_i^+$ and $v_i^+$ of $\mathcal C(G)$, respectively. Since $(T_+,T_-)$ is accepted by $\mathcal C(G)$,  the vertices $u_i^+$ and $v_i^+$ of $\mathcal C(G)$ coincide, as necessary. 
	
	Since $\A_r$ accepts the subgroup $H$ as well as the subgroup $G_{[u]}$, it accepts the subgroup $K$ generated by $H\cup G_{[u]}$. 
	Therefore, by Lemma \ref{lem: morphism from core} there is a unique morphism $\phi$ from the core $\mathcal C(K)$ to $\A_r$.
	For every path $v$ in $\mathcal C(K)$, the morphism $\phi$ maps the vertex $v^+$ of $\mathcal C(K)$ onto the vertex $v^+$ of $\A_r$ (see the proof of Lemma \ref{lem:obvious}). It suffices to prove that $\phi$ is bijective. 
	
	First, we prove that $\phi$ is surjective. Let $v$ be a path in $\A_r$. By Lemma \ref{lem:obvious}, it suffices to prove that $v$ labels a path in $\C(K)$. First, we consider the case where the vertex $v^+$ in $\A_r$ is a vertex of $\mathcal C(H)$. In that case, the path $v$ must be a path in $\mathcal C(H)$ (indeed, if the path $v$ passes through a vertex in $\A_r$ which does not belong to $\C(H)$, it cannot return later to vertices of $\mathcal C(H)$). Hence, by Lemma \ref{lem:minimal core}, there is an element  $h\in H$ such that $v$ is the prefix of one of the branches of its reduced tree-diagram. Then, since $H\leq K$, $v$ must also label a path in $\mathcal C(K)$, as necessary. Now, assume that the vertex $v^+$ in $\A_r$ is a vertex $y$ of the copy of $\C(G)$, other than its root. Then the path $v$ divides into two subpaths  $v\equiv v_1v_2$ such that $v_1$ is a path in $\mathcal C(H)$ from the root $r$ to $\ell$ and $v_2$ labels a path in $\mathcal C(G)$ from its root to the vertex $y$. Note that by the previous case, $v_1$ labels a path in $\mathcal C(K)$. 
	Since $v_2$ labels a path in $\mathcal C(G)$, by Lemma \ref{lem:minimal core}, there is an element $g\in G$ such that $v_2$ is a prefix of some branch of its reduced tree-diagram. Then, $uv_2$ is a prefix of some branch of the reduced tree-diagram of the element $g_{[u]}$. Since $g_{[u]}\in K$, the word $uv_2$ labels a path in $\mathcal C(K)$. Since $v_1$ and $u$ are both paths in $\mathcal C(H)$ such that $v_1^+=\ell=u^+$, there is an element $h\in\Cl(H)$ with the pair of branches $v_1\to u$. But the element $h$ also belongs to $\Cl(K)$ and $u$ and $v_1$ label paths in $\mathcal C(K)$. Hence, by Lemma \ref{identified vertices in the core}, in $\mathcal C(K)$, we also have $v_1^+=u^+$. Since $uv_2$ labels a path in $\mathcal C(K)$ and $u^+=v_1^+$ in $\mathcal C(K)$, the word $v_1v_2\equiv v$ is also readable on $\mathcal C(K)$, as necessary. 
	
	Next, we prove that $\phi$ is injective on vertices (injectivity of $\phi$ on edges follows easily from that). Let $v_1$ and $v_2$ be two paths in the core $\mathcal C(K)$, such that $\phi(v_1^+)=\phi(v_2^+)$. We need to prove that in $\mathcal C(K)$, we also have $v_1^+=v_2^+$. Let us denote by $z$ the vertex $\phi(v_1^+)$. Again, we consider two cases. 
	
	$(1)$ The vertex $z$ belongs to $\mathcal C(H)$. In that case, any path in $\A_r$ which terminates in $z$ passes only through vertices of $\mathcal C(H)$. Hence, $v_1$ and $v_2$ label paths in $\mathcal C(H)$ which terminate in the same vertex $z$. Hence, by Lemma \ref{identified vertices in the core}, there is an element $f$ in $\Cl(H)$ with the pair of branches $v_1\to v_2$. Note that $\Cl(H)\leq \Cl(K)$, since $H\leq K$. Hence, $f\in \Cl(K)$. Then, since $f$ has the pair of branches $v_1\to v_2$ and $v_1$ and $v_2$ are readable on $\mathcal C(K)$, by Lemma \ref{identified vertices in the core}, we have $v_1^+=v_2^+$ in $\mathcal C(K)$, as required.

	$(2)$ The vertex $z$ does not belong to $\mathcal C(H)$. In that case, it is a vertex of the copy of $\mathcal C(G)$ in $\A_r$, distinct from its root. As noted above, every path in $\A_r$ which terminates in such a vertex can be divided into two subpaths $w_1,w_2$ such that $w_1$ is a path in $\mathcal C(H)$ from $r$ to $\ell$ and $w_2$ labels a path in $\mathcal C(G)$ from its root to the vertex. Hence, there are finite binary words $p_1,p_2$, $q_1,q_2$ such that $v_1\equiv p_1p_2$ and $v_2\equiv q_1q_2$ and such that $p_1$ and $q_1$ label paths in $\mathcal C(H)$ from the root to $\ell$ and $q_1$ and $q_2$ label paths in $\mathcal C(G)$ from its root to the vertex $z$. Note that in $\mathcal C(H)$ we have $u^+=p_1^+=q_1^+$. Hence, by Lemma \ref{identified vertices in the core}, there are elements $h_1,h_2\in \Cl(H)$ such that $h_1$ has the pair of branches $p_1\to u$ and $h_2$ has the pair of branches $u\to q_1$. 
	Similarly, since $q_1^+=q_2^+$ in $\mathcal C(G)$, there is an element  $g\in\Cl(G)$ such that $g$ has the pair of branches $p_2\to q_2$. Then, the element $k=g_{[u]}$, has the pair of branches $up_2\to uq_2$. 
	Note that the elements $k,h_1,h_2$ all belong to the closure of $K$. Hence, $h_1kh_2\in\Cl(K)$. But the element $h_1kh_2$ has the pair of branches $p_1p_2\to q_1q_2$ (indeed, $h_1$ takes the branch $p_1p_2$ to the branch $up_2$, then $k$ takes the branch $up_2$ to the branch $uq_2$ and $h_2$ takes the branch $uq_2$ to the branch $q_1q_2$).  Hence, the paths $p_1p_2\equiv v_1$ and $q_1q_2\equiv v_2$ terminate on the same vertex of $\mathcal C(K)$, as necessary. 
\end{proof}

\begin{Lemma}\label{full}
	Let $H$ be a finitely generated proper subgroup of $F$. Then $H$ is contained in a finitely generated proper subgroup $G$ of $F$ such that the core of $G$ has no leaves. 
\end{Lemma}

\begin{proof}
	Since $H$ is finitely generated, its core is finite. If the core of $H$ has no leaves, we are done. Hence, assume that $H$ has $n\geq 1$ leaves: $\ell_1,\dots,\ell_n$ and for each $i=1,\dots,n$, let $u_i$ be a finite binary word which labels a path in $\mathcal C(H)$ from the root to $\ell_i$. 
	Let $\A_r$ be the tree-automaton obtained from $\mathcal C(H)$, by identifying each leaf $\ell_i$ with the root of a distinct copy of the core of Thompson's group $F$, $\mathcal C(F)$ (where the root of $\A_r$ is taken to be the root of $\mathcal C(H)$). Repeated applications of Lemma \ref{Core of H,G}, show that $\A_r$ is isomorphic to the core of the subgroup $G$  generated by $H\cup F_{[u_1]}\cup \dots\cup F_{[u_n]}$. Since each of these subgroups is finitely generated, the subgroup $G$ is finitely generated. In addition, since the core of Thompson's group $F$ has no leaves, the core $\mathcal C(G)\cong \A_r$ has no leaves. Therefore, the subgroup $G$ is as required (note that $G\neq F$ since its core is not isomorphic to the core of $F$, indeed, in the core of $F$ there is a unique middle vertex and it is easy to see that in the core of $G$ there is more than one middle vertex).
\end{proof}

\begin{Theorem}
	Let $H$ be a finitely generated proper subgroup of $F$. Then the
	following assertions hold. 
	\begin{enumerate}
		\item[$(1)$] There exists a finitely generated maximal subgroup $M\leq F$ which contains $H$. 
		\item[$(2)$] If the action of  $H$  on the set of dyadic fractions $\mathcal D$ has finitely many orbits then every maximal subgroup of $F$ which contains $H$ is finitely generated. Moreover, there are only finitely many maximal subgroups of infinite index in $F$ which contain $H$. 
	\end{enumerate}
\end{Theorem}

\begin{proof}
	We begin by proving (2). Since $H$ is finitely generated, its core is finite. Hence, by Lemma \ref{lem:finitely many orbits}, the fact that the action of $H$ on the set of dyadic fractions $\mathcal D$ has finitely many orbits implies that $\mathcal C(H)$ has no leaves. We claim that every maximal subgroup of $F$ which contains $H$ is finitely generated. Indeed, let $M$ be a maximal subgroup of $F$ which contains $H$. If $M$ has finite index in $F$ then it is finitely generated and we are done. Hence, we can assume that $M$ has infinite index in $F$. Hence, by Corollary \ref{max closed}, $M$ is closed. Let $\C(M)$ be the core of $M$. Since $H\leq M$ and since $\C(H)$ has no leaves, by Corollary \ref{cor:no-leaves}, there is a surjective morphism from $\mathcal C(H)$ to $\mathcal C(M)$. Since the core of $H$ is finite, its surjective image $\C(M)$ must also be finite. In \cite[Corollary 5.14]{GS3}, 
	we proved that the core of a closed subgroup of $F$ is finite if and only if the subgroup is finitely generated. Hence, the closed subgroup $M$ is finitely generated, as required. 
	
	Note that the above argument shows that every maximal subgroup $M$ of infinite index in $F$ which contains $H$ is a closed subgroup whose core $\C(M)$ is a surjective image of $\C(H)$. Since $\C(H)$ is finite, it only has finitely many non-isomorphic surjective images. Hence, there are only finitely many maximal subgroups of $F$ of infinite index which contain $H$ 
	(note that if $G_1$ and $G_2$ are distinct closed subgroups of $F$ then $\C(G_1)$ and $\C(G_2)$ are not isomorphic, since $G_1=\ddd(\C(G_1))$ and $G_2=\ddd(\C(G_2))$).
	
	Now, we are ready to prove $(1)$. By Lemma \ref{full}, there exists a finitely generated proper subgroup $G\leq F$ which contains $H$ such that the core $\mathcal C(G)$ has no leaves. It suffices to show that $G$ is contained in some finitely generated maximal subgroup of $F$. Note that by Zorn's lemma, $G$ must be contained inside some maximal subgroup of $F$. We claim that every such subgroup is finitely generated. Indeed, since $G$ is finitely generated its core $\C(G)$ is finite. Since in addition $\C(G)$ has no leaves, by Lemma \ref{lem:finitely many orbits}, the action of $G$ on the set of dyadic fractions $\mathcal D$ has finitely many orbits. Hence, by part (2), every maximal subgroup of $F$ which contains $G$ is finitely generated. 
\end{proof}

We finish this section by giving a characterization of maximal subgroups of $F$ of infinite index.  
Note that by Corollary \ref{max closed}, we only need to consider closed subgroups of $F$.

\begin{Definition}
	Let $\A_r$ be a rooted tree-automaton. We say that $\A_r$ is a \emph{core automaton} if there is a subgroup $H$ of $F$ such that $\A_r$ is isomorphic to the core $\mathcal C(H)$ (i.e., such that there exists a bijective morphism of rooted tree-automata from $\A_r$ to $\mathcal C(H)$).
\end{Definition}

\begin{Lemma}\label{lem:max char}
	Let $H$ be a closed subgroup of $F$. Then $H$ is a maximal subgroup of infinite index in $F$ if and only if the following conditions hold. 
	\begin{enumerate}
		\item[$(1)$] $H[F,F]=F$. 	
		\item[$(2)$] The core $\mathcal C(H)$ is full. 
		\item[$(3)$] There is more than one middle vertex in the core $\C(H)$. 
		\item[$(4)$] Each core-automaton $\A_r$ that is a surjective image of $\C(H)$ but non isomorphic to $\C(H)$ is isomorphic to $\C(F)$.  
	\end{enumerate}
\end{Lemma}

\begin{proof}
	Assume that $H$ is a maximal subgroup of infinite index in $F$, then it cannot be contained in any proper finite index subgroup of $F$. Hence, by Remark \ref{rem:H[F,F]}, $H[F,F]=F$. 	
	 We claim that the core of $H$ has no leaves. Otherwise, let $\ell$ be a leaf of $H$ and let $u$ be a path in the core $\mathcal C(H)$ such that $u^+=\ell$. As in the proof of Lemma \ref{full}, one can show that the subgroup $K$ of $F$ generated by $H$ and $F_{[u]}$ is a strict subgroup of $F$. It strictly contains $H$ (indeed, every non-trivial element of $F_{[u]}$ is not accepted by $\mathcal C(H)$ and as such it does not belong to $H$). Hence, $H$ is not a maximal subgroup of $F$, a contradiction. Now, we claim that there is more than one middle vertex in $\mathcal C(H)$. Otherwise, there is a unique middle vertex in $\C(H)$. Since $\C(H)$ is full, the unique middle vertex is an inner vertex. In that case, by Lemma \ref{core of derived subgroup}, the closed subgroup $H=\Cl(H)$ contains the derived subgroup of $F$. But then $H=H[F,F]=F$, in contradiction to $H$ being a maximal subgroup of $F$. Finally, let $\A_r$ be a core-automaton that is a surjective  image of $\C(H)$ but not isomorphic to $\C(H)$. Since $\A_r$ is a core-automaton, there exists a subgroup $G$ of $F$ such that $\A_r$ is isomorphic to $\mathcal C(G)$. The surjective morphism from $\C(H)$ to $\C(G)$ implies that every tree-diagram accepted by $\C(H)$ is accepted by $\C(G)$. Hence $H=\Cl(H)\leq \Cl(G)$. If $H=\Cl(G)$, then the core $\mathcal C(H)$ is isomorphic to the core $\mathcal C(G)$ (indeed, for every subgroup $K$ of $F$, the core of $K$ coincides with the core of $\Cl(K)$, see \cite{G}). Since by assumption, $\A_r\cong \C(G)$ is not isomorphic to $\mathcal C(H)$, we have that $H<\Cl(G)$. Since $H$ is a maximal subgroup of $F$, it follows that $\Cl(G)=F$. Hence $\A_r\cong \mathcal C(G)\cong \mathcal C(F)$, as required.

	In the other direction, assume that $H$ is a closed subgroup of $F$ which satisfies Conditions $(1)-(4)$. We claim that $H$ is a maximal subgroup of infinite index in $F$. First note that Condition (3) and Lemma \ref{core of derived subgroup} imply that $H$ does not contain the derived subgroup of $F$. Hence, $H$ has infinite index in $F$. Now, let $G$ be a subgroup of $F$ which strictly contains $H$. It suffices to prove that $G=F$. Condition $(1)$ implies that $G[F,F]=F$. Condition $(2)$ and Corollary \ref{cor:no-leaves} imply that there is a unique surjective morphism from $\mathcal C(H)$ to $\mathcal C(G)$. Since $G$ strictly contains $H=\Cl(H)$, there is an element in $G$ that is not accepted by the core of $H$. Hence,  the core  $\mathcal C(G)$ is not isomorphic to $\C(H)$. Therefore, by Condition $(4)$, the core $\C(G)$ is isomorphic to $\C(F)$. In particular, $\Cl(G)=\Cl(F)=F\supseteq [F,F]$. Hence, by Corollary \ref{main cor}, $G=F$, as necessary.  
\end{proof}

Note that if $K$ is a finitely generated subgroup of $F$, then given a finite generating set of $K$ one can construct its core $\mathcal C(K)$. By \cite[Corollary 5.14]{GS3}, the subgroup $H=\Cl(K)$ is also finitely generated. Moreover, in \cite{GS3}, we give an algorithm for finding a finite generating set of $H=\Cl(K)$ (given the core  $\mathcal C(K)$). Now, in order to determine if the closed subgroup $H$ is a maximal subgroup of $F$ of infinite index, one can attempt to verify if Conditions $(1)-(4)$ of Lemma \ref{lem:max char} hold for $H$. Note that the first $3$ conditions are simple to verify (indeed, to check if Condition $(1)$ holds for $H$ one can consider the image of its finite generating set in the abelianization of $F$, and in order to check if Conditions $(2)$ and $(3)$ hold one only has to consider the core of $H$, which coincides with the core of $K$). The only condition we do not have an algorithm for verifying is Condition $(4)$. Note that Condition $(4)$ 
can  also be formulated as follows: every rooted tree-automaton $\A_r$ that is a surjective image of $\mathcal C(H)$ is either isomorphic to $\mathcal C(H)$ or to $\mathcal C(F)$ or is not a core-automaton. 
In the next section, we give a characterization of rooted tree-automata which are core-automata.
The characterization relies on results of \cite[Section 10]{G}.
While the characterization does not yield a complete algorithm for verifying if a finitely generated closed subgroup $H$ of $F$ satisfies Condition (4) of Lemma \ref{lem:max char}, it often enables us to check if the condition holds or not (see Remarks \ref{rem:core-automaton} and \ref{rem:delete?} at the end of the next section) and thus, to verify if $H$ is a maximal subgroup of infinite index in $F$.

\section{Core Automata}\label{sec:core}

The following lemma gives a characterization of rooted tree-automata which are core-automata. It follows from \cite[Lemmas 10.9 and 10.6]{G} and from Corollary \ref{cor:red aut} and Lemma \ref{identified vertices in the core}, but we prefer to give a self-contained proof. 

\begin{Lemma}\label{core automaton}
	Let $\A_r$ be a rooted tree-automaton. Then $\A_r$ is a core automaton if and only if the following conditions hold. 
	\begin{enumerate}
		\item[$(1)$] $\A_r$ is reduced. 
		\item[$(2)$] For every pair of finite binary words $u$ and $v$ which label paths in $\A_r$ such that $u^+=v^+$, there is an element in the diagram group $\ddd(\A_r)$ with the pair of branches $u\to v$. 	
	\end{enumerate} 
\end{Lemma}

\begin{proof}
	For every subgroup $H$ of $F$, the core $\mathcal C(H)$ satisfies conditions $(1)$ and $(2)$ by Lemmas \ref{lem:core reduced} and \ref{identified vertices in the core}.  Hence, if $\A_r$ is a core automaton, it satisfies Conditions $(1)$ and $(2)$.

	In the opposite direction, let $\A_r$ be a rooted tree-automaton which satisfies conditions $(1)$ and $(2)$. Let $H$ be the subgroup of $F$ accepted by $\A_r$ (i.e., $H=\ddd(\A_r)$). We claim that there is a bijective morphism from $\mathcal C(H)$ to $\A_r$. First, by Lemma \ref{lem: morphism from core}, there is a unique morphism $\phi$ from $\mathcal C(H)$ to $\A_r$. For each finite binary word $u$ which labels a path in $\C(H)$, the morphism $\phi$ maps the vertex 
	$u^+$ of $\mathcal C(H)$ to the vertex $u^+$ of $\A_r$.  It suffices to prove that $\phi$ is bijective.
	
	First, we claim that $\phi$ is injective on vertices (from that it follows that it is also injective on edges). Indeed, let $u$ and $v$ be two paths in the core $\mathcal C(H)$ and assume that $\phi$ maps $u^+$ and $v^+$ onto the same vertex in $\A_r$. It suffices to show that in $\mathcal C(H)$ we have $u^+=v^+$. But since $\phi(u^+)=\phi(v^+)$, the paths $u$ and $v$ on $\A_r$ terminate on the same vertex. Hence, it follows from Condition $(2)$ that there is an element $h$ in the subgroup $H=\ddd(\A_r)$ with the pair of branches $u\to v$. Since $u$ and $v$ label paths in the core of $H$, the existence of the element $h$ implies by Lemma \ref{identified vertices in the core} that $u^+=v^+$ in the core $\mathcal C(H)$, as required. 
	
	Hence, it suffices to show that $\phi$ is surjective. 
	 To do so, we first note that in $\A_r$ the root $r$ is not a descendant of itself. Otherwise, there is a non-empty finite binary word $u$ which labels a path in $\A_r$ such that $u^+=r=\emptyset^+$. Hence, by Condition $(2)$, there is an element in $H$ (and in particular in $F$) with the pair of branches $u\to\emptyset$, which is impossible. 
	
	Now, to prove that $\phi$ is surjective, 
	it suffices to show that for any path $u$ in $\A_r$, the word $u$ labels a path in $\mathcal C(H)$. Assume by contradiction that this is not the case, and let $u$ be a finite binary word of minimal length such that $u$ labels a path in $\A_r$ but does not label a path in $\mathcal C(H)$. Let $v$ be a finite binary word such that $u\equiv va$ for a letter $a\in\{0,1\}$. By assumption, the word $v$ labels a path in the core $\mathcal C(H)$. The vertex $v^+$ is necessarily a leaf of $\mathcal C(H)$ since $va$ does not label a path in $\mathcal C(H)$. Let us consider the vertex $v^+$ in $\A_r$. The vertex is not a leaf of $\A_r$. Hence, by Corollary \ref{cor:red aut}, the vertex $v^+$ has a descendant in $\A_r$ with two distinct incoming edges. Hence there exists a non-empty finite binary word $w$ such that the vertex $y=(vw)^+$ has two distinct incoming edges in $\A_r$.

	Let $p$ be the prefix of $w$ such that $w\equiv pb$ for some letter $b\in\{0,1\}$ and note that $vpb$ is a path in $\A_r$ such that $(vpb)^+=y$. Let $e_1$ be the last edge of the path $vpb$ in $\A_r$.  
By assumption, the vertex $y$ has an incoming edge $e_2\neq e_1$ in $\A_r$. Let $q$ be a path in $\A_r$ such that $q^+={e_2}_-$ and let $c$ be the label of $e_2$. Then $qc$ labels a path in $\A_r$ such that $(qc)^+=y$. Since $(vpb)^+=(qc)^+$ in $\A_r$, 	
 by Condition $(2)$, there is an element $h$ in $H$ with the pair of branches $vpb\to qc$. We claim that the reduced tree-diagram of $H$ also has this pair of branches. Indeed, assume by contradiction that this is not the case, then the letters $b$ and $c$ must coincide. In addition, by reducing one relevant common caret, we get that $h$ has the pair of branches $vp\to q$. But since $H$ is accepted by $\A_r$ and $vp$ and $q$ label paths in $\A_r$, we must have $(vp)^+=q^+$ in $\A_r$. Hence, the edges $e_1$ and $e_2$ have the same initial vertex (since ${e_1}_-=(vp)^+=q^+={e_2}_-$) and the same label. Hence, the edges $e_1$ and $e_2$ coincide, in contradiction to the assumption. Hence, $vpb\to qc$ is a pair of branches of the reduced tree-diagram of $h$. Therefore, $vpb$ labels a path in $\mathcal C(H)$, in contradiction to $v^+$ being a leaf of $\C(H)$. Hence, $\phi$ is bijective on vertices. 
\end{proof}

Note that even if $\A_r$ is a finite tree-automaton, if some vertex in $\A_r$ is a descendant of itself (equiv. if there is a directed cycle in $\A_r$), then there are infinitely many finite binary words which label paths in $\A_r$ and there are infinitely many pairs of words $u$ and $v$ which label paths in $\A_r$ such that $u^+=v^+$. However, it turns  out \cite{G} that to verify if condition $(2)$ of Lemma \ref{core automaton} holds for a finite tree-automaton $\A_r$ it suffices to consider only finitely many pairs of finite binary words.

Indeed, following \cite[Section 10.2]{G}, given a rooted tree-automaton $\A_r$, we associate labeled binary trees $T_{\A_r}$ and $T_{\A_r}^{\min}$ with the tree-automaton $\A_r$. 

Given a labeled binary tree $T$, a \emph{path} $p$ in $T$ is always a simple path starting from the root. Every path is labeled by a finite binary word $u$. As for paths in tree-automata, we  rarely distinguish between the path $p$ and its label $u$. Similarly, we  denote by $p^+$ or $u^+$ the terminal vertex of the path $p$ in $T$ and by $\lab(u^+)$ or $\lab(p^+)$   the label of this terminal vertex. An \emph{inner} vertex of $T$ is a vertex which is not a leaf. 

Now, let $\A_r$ be a rooted tree-automaton. The labeled binary tree $T_{\A_r}$ associated with $\A_r$ is defined as follows. The labels of vertices in $T_{\A_r}$ are the vertices of $\A_r$.  Recall that each finite binary word $u$ labels at most one path in $\A_r$. We let $T_{\A_r}$ be the maximal binary tree such that for every finite path $u$ in $T_{\A_r}$, the finite binary word $u$ labels a path in $\A_r$. For example, if there are no leaves in $\A_r$ then $T_{\A_r}$ is the complete infinite binary tree. The label of each vertex $u^+$ of $T_{\A_r}$ is the vertex $u^+$ of $\A_r$.

Notice that every caret in $T_{\A_r}$ is labeled with accordance with some father and his children (sometimes called a \emph{caret})  in the tree-automaton $\A_r$. In fact, $T_{\A_r}$ can be constructed inductively as follows. One starts with a root labeled by the root $r$ of $\A_r$. Whenever there is a leaf in the tree whose label is a father $x$ in  $\A_r$,  one attaches a caret to the leaf and labels the left (resp. right) leaf of the caret by the left (resp. right) child  of $x$. 

Now let $T$ be a rooted subtree of $T_{\A_r}$, maximal with respect to the property that there is no pair of distinct inner vertices in $T$ which have the same label. If $\ell$ is a leaf of $T$ and $\ell$ does not share a label with any inner vertex in $T$, then $\ell$ must be a leaf of $T_{\A_r}$. Indeed, otherwise one could attach the caret of $T_{\A_r}$ with  root $\ell$ to the  subtree $T$ and get a larger subtree where no pair of distinct inner vertices share a label. 

If the leaf $\ell$ shares a label with some inner vertex $x$ of $T$, then in $T_{\A_r}$, $\ell$ has two children. Each  child of $\ell$ is labeled as the respective child of $x$. Continuing in this manner, we see that it is possible to get $T_{\A_r}$ from $T$, by inductively attaching carets to leaves which share their label with inner vertices of $T$ and labeling the new leaves appropriately. It follows that $T$ and $T_{\A_r}$ have the same set of labeled carets. Since no labeled caret appears in $T$ more than once, $T$ is a minimal subtree of $T_{\A_r}$ with respect to the property that the sets of labeled carets of $T$ and $T_{\A_r}$ coincide.

We let a \emph{minimal tree associated with $\A_r$}, denoted $T_{\A_r}^{\min}$, be a tree $T$ as described in the preceding paragraph. We note that a minimal tree associated with $\A_r$ is not unique. However, the label of the root of $T_{\A_r}^{\min}$ (which is $r$) and the set of labeled carets of $T_{\A_r}^{\min}$ are determined uniquely by $\A_r$. Thus, we can consider different minimal trees associated with $\A_r$ to be equivalent. Clearly, a minimal tree associated with $\A_r$ also determines $\A_r$ uniquely.

The following lemma follows immediately from \cite[Lemma 10.6]{G}. 

\begin{Lemma}\label{T_min}
	Let $\A_r$ be a rooted tree-automaton and let $T_{\A_r}^{\min}$ be an associated minimal tree. Assume that for every pair of finite binary words $u$ and $v$ which label paths in $T_{\A_r}^{\min}$ such that $u^+$ and $v^+$ share a label in $T_{\A_r}^{\min}$ 
	 there is a tree-diagram $(T_+,T_-)$ which has the pair of branches $u\to v$ and is accepted by $\A_r$. Then for every pair of finite binary words $u_1$ and $v_1$ which label paths in $\A_r$ such that $u_1^+=v_1^+$, there is a tree-diagram accepted by $\A_r$ which has the pair of branches $u_1\to v_1$. 
\end{Lemma}

Note that Lemma \ref{T_min} shows that if $\A_r$ is a finite rooted tree-automaton then to check if Condition (2) from Lemma \ref{core automaton} holds for $\A_r$ it suffices to check finitely many pairs of finite binary words. Following \cite{G}, we make the following definition.

\begin{Definition}
	Let $\A_r$ be a tree-automaton. We define a semigroup presentation $\mathcal P^{\A_r}$ associated with $\A_r$ as follows. The alphabet of $\mathcal P^{\A_r}$ is the vertex set $V$ of $\A_r$. The set of relations $R$ is the set of all relations of the form $a=bc$ where $a,b,c\in V$ and $b$ and $c$ are the left and right child of $a$, respectively. The semigroup presentation $\mathcal P^{\A_r}=\la V\mid R\ra$ is the \emph{presentation associated with $\A_r$}.
\end{Definition}

Note that if $\mathcal A_r$ is a rooted tree-automaton, then the associated semigroup presentation together with the root $r$ completely determine $\mathcal A_r$. 

Let $\mathcal A_r$ be a rooted tree-automaton and let $V$ be the vertex set of $\mathcal A_r$. 
Let $u$ be a finite binary word which labels a path in $\mathcal A_r$ and let $T_u$ be the minimal binary tree with branch $u$. Note that by Remark \ref{rem:obvious}, $T_u$ must be readable on $\mathcal A_r$, so in particular, $T_u$ is a rooted subtree of $T_{\A_r}$. Let us label each vertex of $T_u$ by its label when it is viewed as a rooted subtree of $T_{\A_r}$ and note that if $u_1,\dots,u_n$ are the branches of $T_u$, then the labels of the leaves of $T_u$, read from left to right, are the vertices $u_1^+,\dots,u_n^+$ of $\A_r$. Now, assume that $u$ is the $k^{th}$ branch of $T_u$. The words 
$p_u\equiv u_1^+\cdots u_{k-1}^+$ and $q_u\equiv u_{k+1}^+\cdots u_n^+$ (where for each $i$, $u_i^+$ is the end vertex of the path $u_i$ in $\mathcal A_r$) are words over the alphabet $V$. 
Note that $p_u$ or $q_u$ might be empty (note also that the labels of the leaves of $T_u$ read from left to right spell the word $p_u u^+q_u$). The pair $(p_u,q_u)$ is called the \emph{pair of words associated with the path $u$  in $\mathcal A_r$}. Note that (if they are not empty) $p_u$ and $q_u$ represent elements of the semigroup $S^{\A_r}$ represented by the semigroup presentation $\P^{\A_r}$.

The following lemma is a slight modification of \cite[Lemma 10.10]{G} (and follows immediately from the proof in \cite{G}). 

\begin{Lemma}\label{lem:pupv}
	Let $\A_r$ be a rooted tree-automaton and let $H$ be the subgroup of $F$ accepted by $\A_r$.  Let $u,v$ be finite binary words which label paths in $\A_r$ and let $(p_u,q_u)$ and $(p_v,q_v)$ be the associated pairs of words. Then there is an element $h\in H$ with the pair of branches $u\to v$ if and only if the following assertions hold. 
	\begin{enumerate}
		\item[$(1)$] $u^+=v^+$ in $\A_r$. 
		\item[$(2)$] $p_u= q_u$ and $p_v=q_v$ in the semigroup $S^{\A_r}$ defined by the presentation $\mathcal P^{\A_r}$ associated with the automaton $\A_r$. 	
	\end{enumerate}	
\end{Lemma}

Note that when we write $p_u=p_v$ in the semigroup $S^{\A_r}$ defined by the presentation $\mathcal P^{\A_r}$, the meaning is that either both  $p_u$ and $p_v$ are empty or that they are both non-empty and $p_u=p_v$ in $S^{\A_r}$. Similarly, for $q_u$ and $q_v$. 

\begin{Remark}\label{rem:core-automaton}
	Let $\A_r$ be a finite rooted tree-automaton. To check if $\A_r$ is a core automaton one has to check whether it is reduced (which can be easily done) and whether it satisfies Condition (2) from Lemma \ref{core automaton}. For that, one can construct the (finite) tree $T_{\A_r}^{\min}$. Then one has to check for each pair of finite binary words $u$ and $v$ which label paths on $T_{\A_r}^{\min}$ such that $u^+$ and $v^+$ share a label in $T_{\A_r}^{\min}$ whether Condition (2) from Lemma \ref{core automaton} holds for $u$ and $v$. Lemma \ref{lem:pupv} shows that to check if Condition (2) from Lemma \ref{core automaton} holds for $u$ and $v$ one has to check equality of words  in the semigroup $S^{\A_r}$ given by the presentation $\P^{\A_r}$. Hence, if the word problem for $\P^{\A_r}$ is decidable (for example, if $\P^{\A_r}$ has a finite completion), we get a method for checking if $\A_r$ is a core automaton.
	Note that in general, the word problem for finite semigroup presentations is undecidable. However, not every semigroup presentation is associated with a rooted tree-automaton and to verify if Condition (2) from Lemma \ref{lem:pupv} holds for every relevant  pair of finite binary words, we do not need a complete solution for the word problem over the presentation $\P^{\A_r}$. Hence, it is possible that the problem of deciding if a given rooted tree-automaton is a core automaton is decidable. 
\end{Remark}

\begin{Example}
	The rooted tree-automaton $\A_r$ from Figure \ref{fig:NotCore} is not a core automaton.

	\begin{figure}[ht]
		\centering
		\includegraphics[width=.38\linewidth]{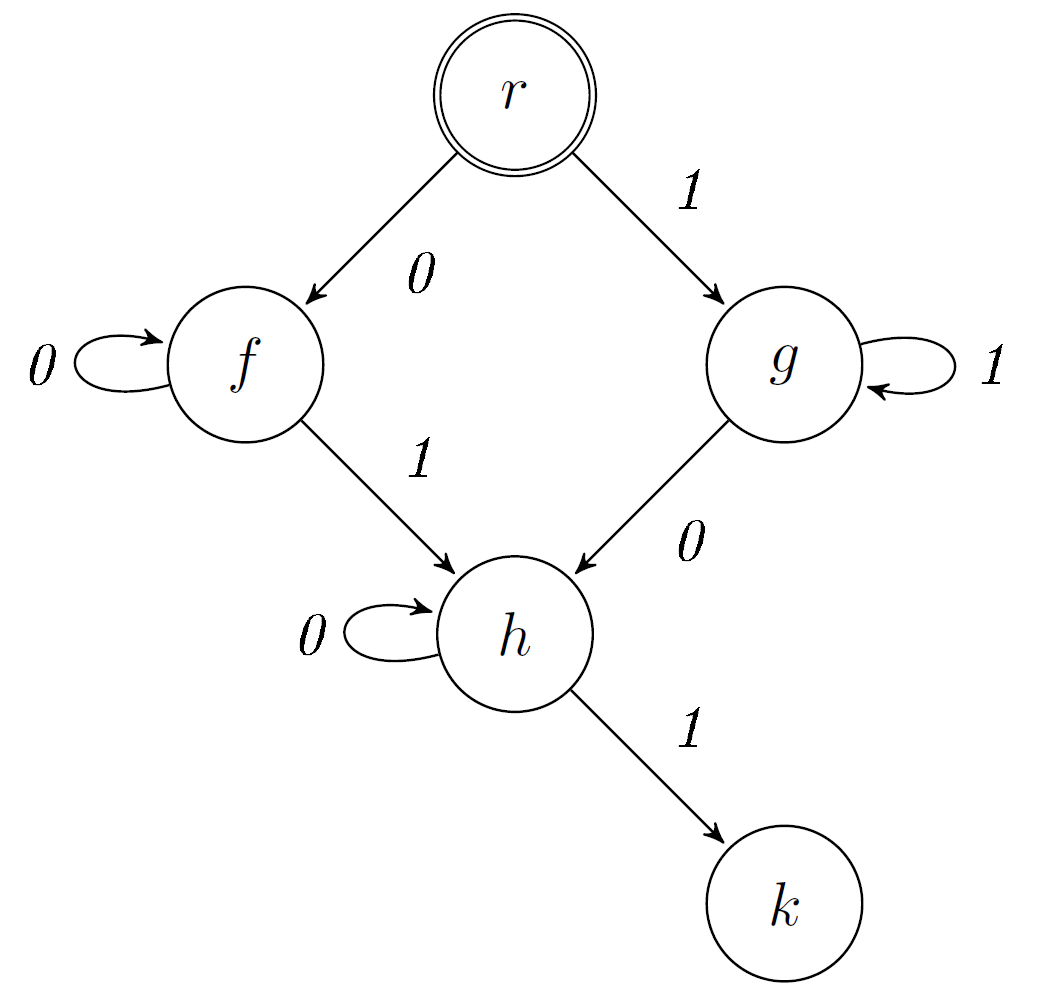}
		\caption{A rooted tree-automaton which is not a core automaton.}\label{fig:NotCore}
	\end{figure}

\begin{proof}
Note that the following is a minimal tree associated with $\A_r$: 

\Tree[.$r$ [.$f$  [.$f$ ] [.$h$ [.$h$ ] [.$k$ ] ] ] [.$g$ [.$h$ ] [.$g$ ] ] ] 

The semigroup presentation associated with $\A_r$ is: 
$$\mathcal P^{\A_r}=\la r,f,g,h,k\mid r=fg,f=fh,g=hg,h=hk \ra$$
Now, consider the words $u\equiv 01$ and $v\equiv 010$. Note that $u^+=v^+$ in $\A_r$ and let $(p_u,q_u)$ and $(p_v,q_v)$ be the pairs of words associated with the paths $u$ and $v$ in $\A_r$, respectively. Then $q_u\equiv g$ and $q_v\equiv kg$. 
Since no relation in $\P^{\A_r}$ has $k$ as its first letter (on either side of the relation), every word over the alphabet of $\P^{\A_r}$ that is equal to the word $q_v\equiv kg$ in $S^{\A_r}$ must start with $k$. Hence, 
 $q_u\neq q_v$ in the semigroup $S^{\A_r}$. Hence, by Lemma \ref{lem:pupv}, there is no element in $\ddd(\A_r)$ with the pair of branches $u\to v$ and therefore, by Lemma \ref{core automaton}, $\A_r$ is not  a core automaton. 
 \end{proof}
\end{Example}

\begin{Remark}\label{rem:delete?}
	Let $H$ be a finitely generated closed subgroup of $F$ with core $\mathcal C(H)$. Recall that $H$ satisfies Condition (4) from Lemma \ref{lem:max char} if and only if for every rooted tree-automaton $\A_r$ that is a surjective image of $\C(H)$, the tree-automaton $\A_r$ is either isomorphic to $\C(H)$ or to $\C(F)$ or is not a rooted tree-automaton.  
	Hence, to check if Condition (4) from Lemma \ref{lem:max char} holds for $H$, it suffices to consider all surjective images $\A_r$ of $\C(H)$, which are not isomorphic to $\C(H)$ nor to $\C(F)$ and check whether any of them is a core-automaton. This is often doable using Remark \ref{rem:core-automaton}. 
\end{Remark}

\begin{Example}\label{exm}
	The subgroup  $H=\la x_0,x_1^2x_3^{-1}x_2^{-1}x_1^{-1},x_1x_2^2x_3^{-1}x_1^{-2}\ra$ is a maximal subgroup of infinite index in $F$. 
	\begin{proof}
	First, we note that $H$ is a closed subgroup of $F$. Indeed, the core of $H$ is given in Figure \ref{fig:CoreMaximal}.

		\begin{figure}[ht]
		\centering
		\includegraphics[width=.4\linewidth]{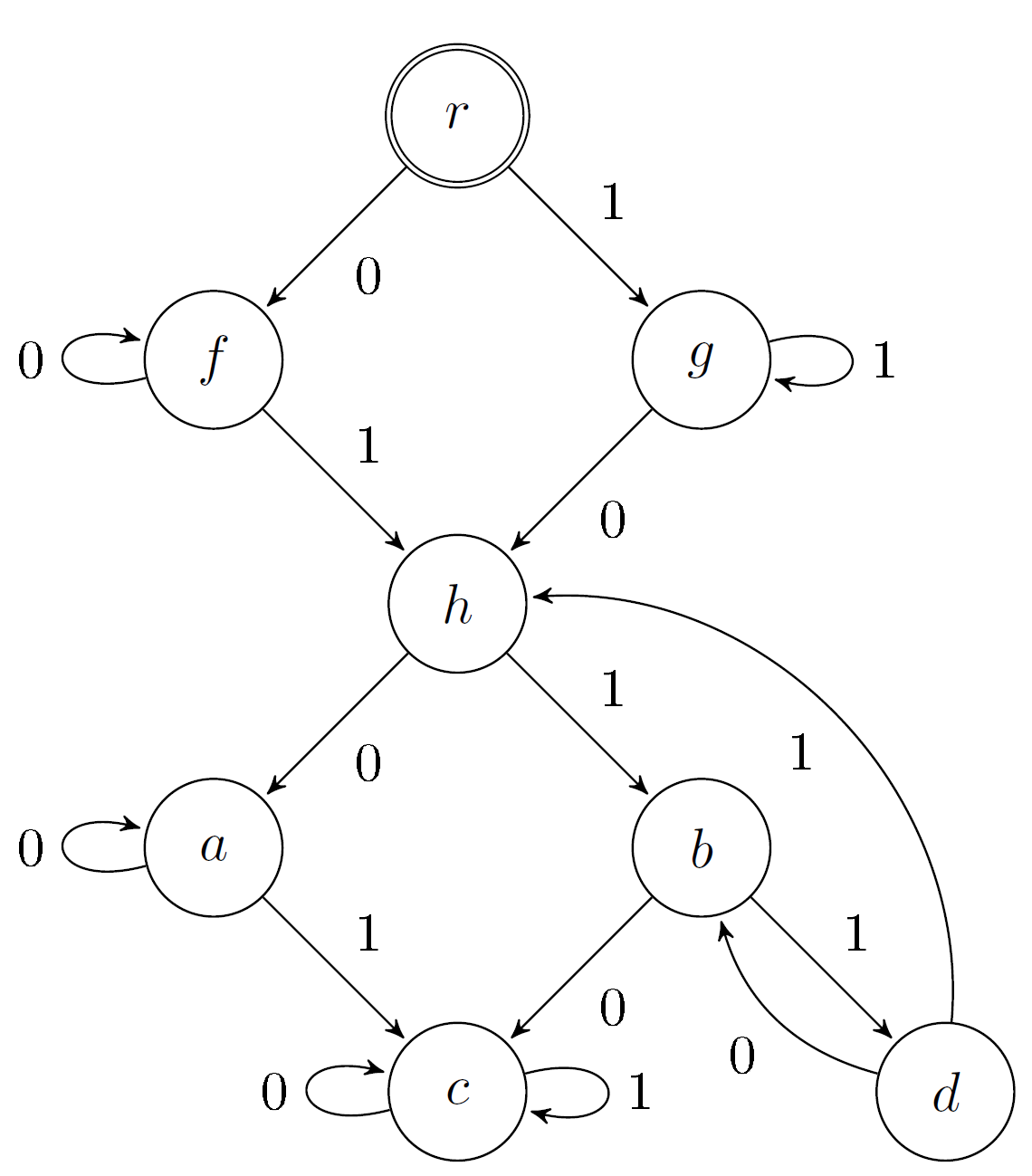}
		\caption{The core of $H$.}\label{fig:CoreMaximal}
	\end{figure}

If one applies the algorithm from \cite{GS3} for finding a generating set of $\Cl(H)$ one gets the given generating set of $H$\footnote{More accurately, the algorithm gives a generating set such that after a few elementary Nielsen transformations coincides with the given generating set of $H$.}. Hence $\Cl(H)=H$. Therefore to show that $H$ is a maximal subgroup of infinite index in $F$ it suffices to  verify that it satisfies Conditions (1)-(4) from Lemma \ref{lem:max char}. 
Computing the images of the generators of $H$ in the abelianization of $F$ shows that $\pi_{\ab}(H)=\mathbb{Z}^2$. Hence, Condition (1) from Lemma \ref{lem:max char} holds. Conditions (2) and (3) are clear from the above core. Hence, it suffices to check that Condition (4) from the lemma holds. Let $\A_r$ be a rooted tree-automaton that is a surjective image of $\C(H)$ but non-isomorphic to $\C(H)$. It suffices to prove that $\A_r$ is either not a core automaton or isomorphic to $\C(F)$. Let $\phi$ be the (unique) surjective morphism from $\C(H)$ to $\A_r$. Since $\phi$ is not injective, there are at least two distinct vertices $x\neq y$ in $\C(H)$ such that $\phi(x)=\phi(y)$. We note that if $x$ or $y$ is not a middle vertex of $\C(H)$ (i.e., if $x$ or $y$ belong to $\{r,f,g\}$) 
then by Remark \ref{rem:different types} the rooted tree-automaton $\A_r$ is not a core-automaton, and we are done. Hence, we can assume that $\phi(r),\phi(f)$ and $\phi(g)$ are all distinct vertices of $\A_r$ and that for every vertex $z$ of $\C(H)$ such that $z\notin\{r,f,g\}$ we have $\phi(z)\notin\{\phi(r),\phi(f),\phi(g)\}$. In particular, the vertices $x$ and $y$ must both be middle vertices of $\C(H)$. 
 It is not difficult to check that if $\{x,y\}\neq \{a,c\}$, then the rooted tree-automaton $\A_r$ must be isomorphic to $\C(F)$. Indeed, assume for  example that $\{x,y\}=\{b,c\}$, so that $\phi(b)=\phi(c)$. Then since $d$ and $c$ are the right children of $b$ and $c$ respectively, we must have $\phi(d)=\phi(c)$. Similarly, since $h$ and $c$ are the right children of $d$ and $c$, respectively, we have $\phi(h)=\phi(c)$. 
  Continuing in this manner we get that $\phi(h)=\phi(a)=\phi(b)=\phi(c)=\phi(d)$. That clearly implies that $\A_r$ is isomorphic to $\C(F)$ (recall that $\C(F)$ appears in Figure \ref{fig:coreF}). Hence,
  it suffices to consider the case where $\{x,y\}=\{a,c\}$ so that $\phi(a)=\phi(c)$. Moreover, we can assume that that $\phi$ is injective on the vertices $h$, $b$ and $d$ and that $\phi(c)\notin\{\phi(h),\phi(b),\phi(d)\}$ (otherwise, we are done by the previous case). It follows that $\A_r$ is the rooted tree-automaton in Figure \ref{fig:NotCoreMaximal}.

  	\begin{figure}[ht]
  	\centering
  	\includegraphics[width=.42\linewidth]{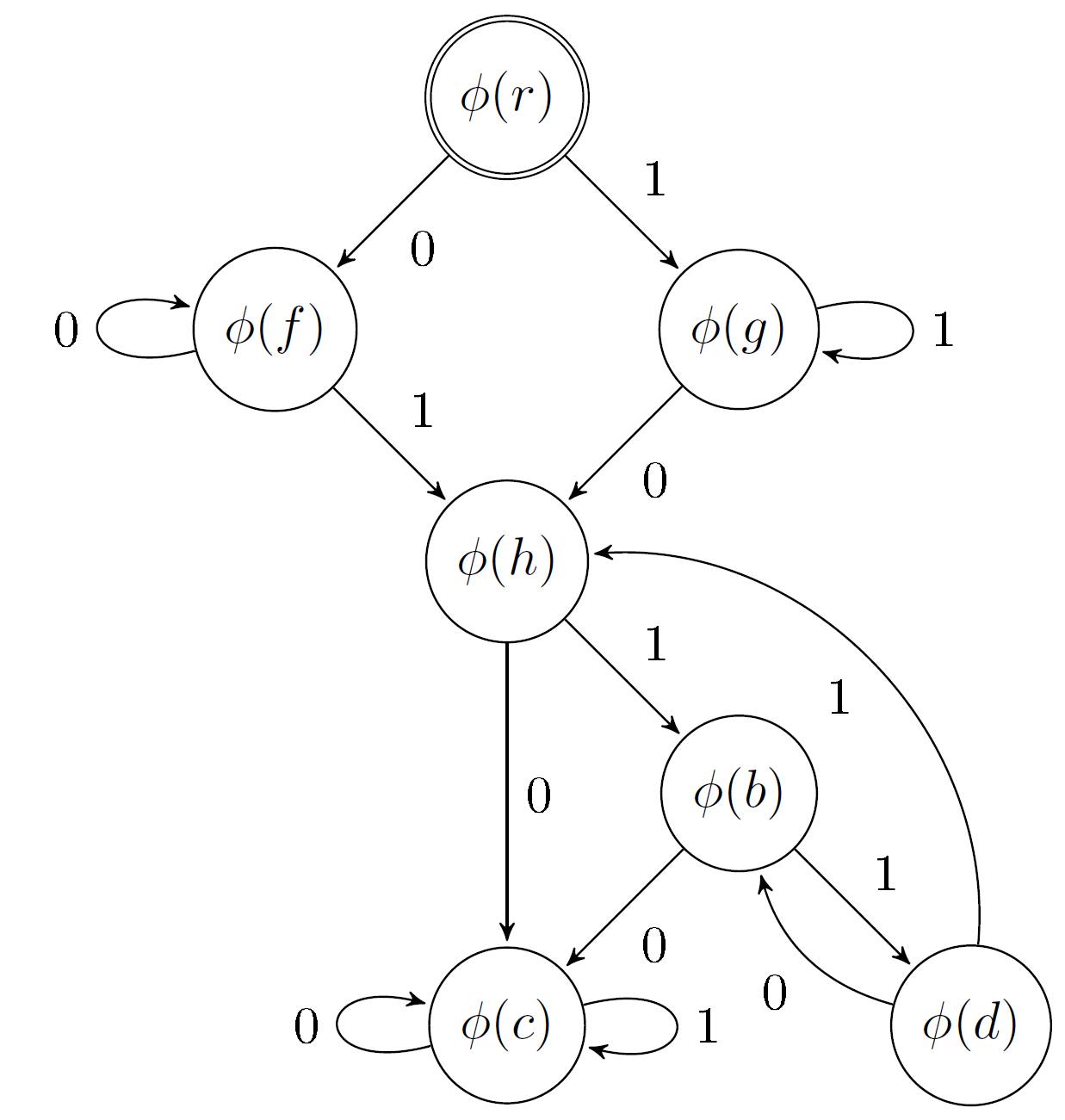}
  	\caption{The rooted tree-automaton $\A_r$.}\label{fig:NotCoreMaximal}
  \end{figure}

This rooted tree-automaton is not a core automaton. Indeed, let $u\equiv 010$ and $v\equiv 0101$ and note that $u^+=v^+=\phi(c)$ in $\A_r$. Let $(p_u,q_u)$ and $(p_v,q_v)$ be the pairs of words associated with the paths $u$ and $v$ in $\A_r$. Then $p_u\equiv \phi(f)$ and $p_v\equiv \phi(f)\phi(c)$. It is easy to check that $p_u\neq p_v$ in the semigroup $S^{\A_r}$ represented by the semigroup presentation $\P^{\A_r}$ (indeed, in every relation of $\P^{\A_r}$ the last letter on the right hand-side of the relation is $\phi(c)$ if and only if the last letter on the left hand-side of the relation is $\phi(c)$). Hence $\A_r$ is not a core automaton and we are done. 
\end{proof}
\end{Example}

\section{An infinite family of non-isomorphic maximal subgroups of Thompson's group $F$}\label{sec:Jones}

\subsection{Jones' subgroups of Thompson's group $F$}

	Vaughan Jones \cite{Jones} defined a family of unitary representations of Thompson's group $F$ using planar algebras. These representations give rise to interesting subgroups of $F$ (the stabilizers of the vacuum vector in these representations). Jones' subgroup $\overrightarrow{F}$, defined in \cite{Jones} is particularly interesting. Indeed, Jones proved that elements of $F$ encode in a natural way all knots and links and elements of $\overrightarrow{F}$ encode all oriented links and knots. 
	
	In \cite{GS} we proved that Jones' subgroup $\overrightarrow{F}$ is isomorphic to
	the ``brother group'' $F_3$  of $F$ (Recall that for every $n\ge 2$ one can define a ``brother group'' $F_n$ of $F=F_2$ as the group of all piecewise linear increasing homeomorphisms of the unit interval where all slopes are powers of $n$ and all breaks of the derivative occur at $n$-adic fractions, i.e., points of the form $\frac a{n^k}$ where $a, k$ are positive integers \cite{B}. It is well known that $F_n$ is finitely presented for every $n$ (a concrete and easy presentation can be found in \cite{GuSa97})).  We also showed that $\overrightarrow F$ is the stabilizer of the set $S$ of all dyadic fractions such that the sum of digits in their finite binary representation is odd.
	
	In \cite{GS1}, we proved that the only subgroups of $F$ strictly containing Jones' subgroup $\arr{F}$ are $F_{1,2}$ and $F$
	 (recall that $F_{1,2}$ is the subgroup of $F$ of index $2$ of all functions whose slope at $1^-$ is an even power of $2$). In particular, $\arr{F}$ is a maximal subgroup of $F_{1,2}$.  Since by \cite{BW}, $F_{1,2}$ is isomorphic to $F$, it followed that $F$ has a maximal subgroup isomorphic to Jones' subgroup $\overrightarrow{F}$. This was the first example of a maximal subgroup of $F$ which is not the stabilizer of any number in $(0,1)$. 
	 
	In \cite{GS}, we also studied a family of subgroups, which we called Jones' subgroups $\overrightarrow F_n$, which can be defined in an analogous way to $\overrightarrow F$, where $\overrightarrow F_2=\overrightarrow F$ (for further details, see \cite[Section 5]{GS}). We showed that like $\overrightarrow{F}$, for each $n$, Jones' subgroup $\overrightarrow{F}_n$ is isomorphic to the brother group $F_{n+1}$ of $F$. We also showed that for each $n$, $\overrightarrow{F}_n$ is the intersection of stabilizers of certain sets of dyadic fractions. 
	
	\begin{Lemma}[{\cite[Theorem 5.11]{GS}}]\label{lem:stabilizers}
		Let $n\geq 2$. For each $i=0,\dots,n-1$, let $S_i$ be the set of all dyadic fractions such that the sum of digits in their finite binary representation is $i$ modulo $n$. 
		Then 
		$$\overrightarrow{F}_n=\bigcap_{i=0}^{n-1}\Stab(S_i).$$
		That is, Jones' subgroup $\overrightarrow{F}_n$ is the intersection of the stabilizers of $S_i$, $i=0,\dots,n-1$ under the natural action of $F$ on the interval $[0,1]$. 
	\end{Lemma}

	It follows from Lemma \ref{lem:stabilizers}, that for each $n$, $\arr{F}_n$ is a closed subgroup of Thompson's group  $F$. Indeed, Lemma \ref{lem:stabilizers} implies that every piecewise-$\arr{F}_n$ function belongs to $\arr{F}_n$. The following was also proved in \cite{GS}.
	
	\begin{Lemma}[See {\cite[Section 5.2]{GS}}]\label{lem:closed for addition}
		For every $n\in\mathbb{N}$, Jones' subgroup $\arr{F}_n$ is the  minimal subgroup of $F$ which contains the element $x_0x_1\dots x_{n-1}$ and is closed for addition. Moreover, the subgroup $\arr{F}_n$ is generated by the set $\{x_ix_{i+1}\cdots x_{i+n-1}\mid i=0,\dots, n\}$.
	\end{Lemma}

\subsection{Maximality of Jones subgroup $\protect\arr{F}_p$ inside $F_{1,p}$}

As noted above, in \cite{GS1}, we proved that the only subgroups of $F$ which strictly contain Jones' subgroup $\arr{F}$ are $F_{1,2}$ and $F$. In this section, we generalize this result for every prime number $p$. Namely, we show that the only subgroups of $F$ which strictly contain Jones' subgroup $\arr{F}_p$ are $F_{1,p}$ and $F$.

For the remainder of this section, let us fix a prime number $p$. We start by defining a rooted tree-automaton which we denote by $\Asum$ such that $\ddd(\Asum)=\arr{F}_p$. 

\begin{Definition}
	 The rooted tree-automaton $\Asum$ (see Figure \ref{fig:Asum}) is defined to be the rooted tree-automaton with vertices $a_0,\dots,a_{p-1}$ such that each vertex $a_i$ has two outgoing edges: a directed edge labeled $``0"$ from the vertex to itself and a directed edge labeled $``1"$ from the vertex to $a_{i+1}$, where $i+1$ is taken modulo $p$. The root of $\Asum$ is the vertex $a_0$. 
\end{Definition}

	\begin{figure}[h]
	\centering
	\includegraphics[width=.42\linewidth]{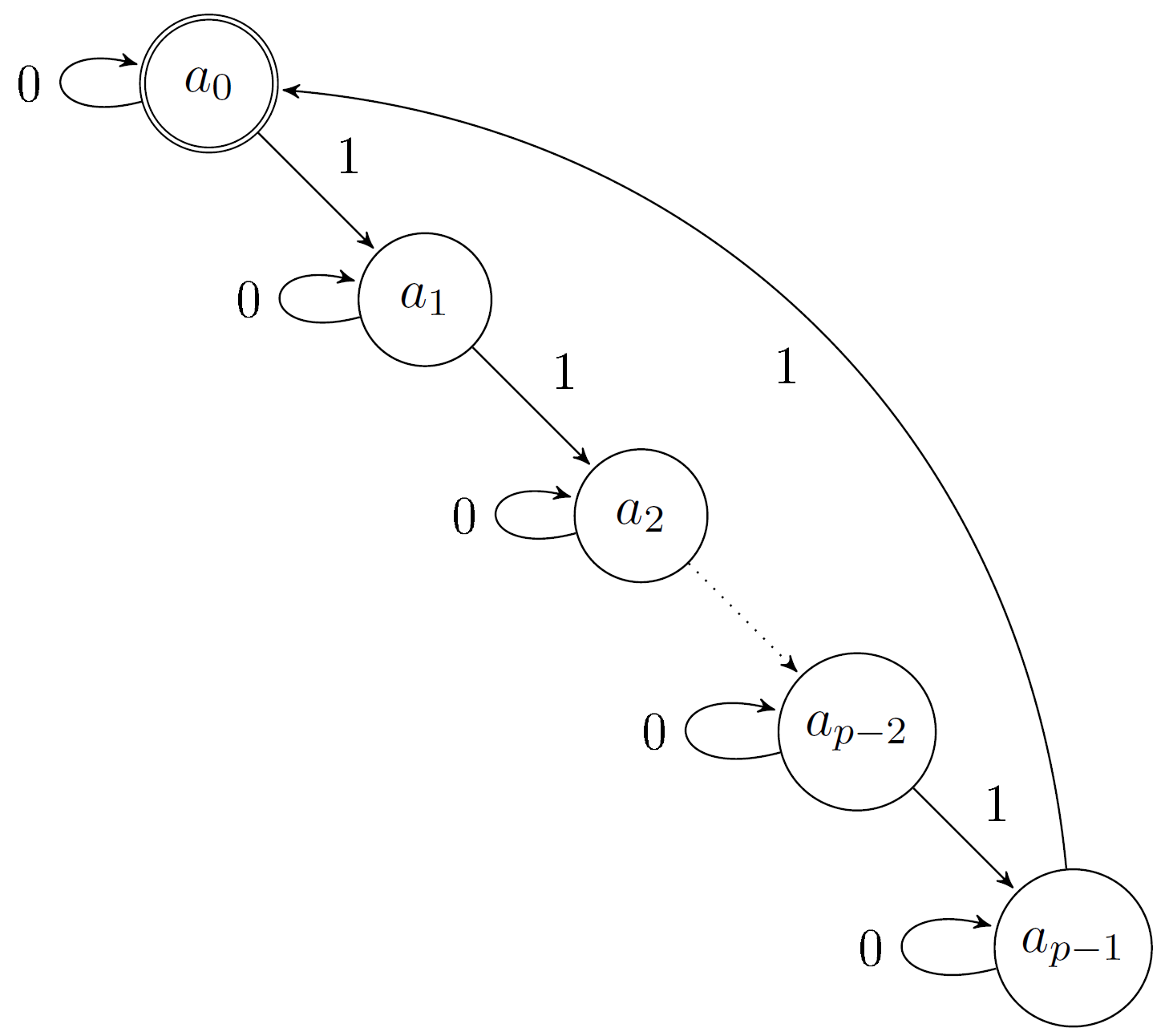}
	\caption{The rooted tree-automaton $\Asum$.}\label{fig:Asum}
\end{figure}

Note that $\Asum$ is a full rooted tree-automaton. In particular, every finite binary word $u$ is readable on $\Asum$. 
Let $u$ be a finite binary word. We denote by $\mathrm{sum}_p(u)$ the number in $\{0,\dots,p-1\}$ that is equal to the sum of digits of $u$  modulo $p$ (i.e., the number of appearances of the digit $1$ in $u$ modulo $p$). 
Similarly, if $\alpha$ is a dyadic fraction and $u$ is a finite binary word such that $\alpha=.u$, we let  $\mathrm{sum}_p(\alpha)=\mathrm{sum}_p(u)$. It is easy to prove by induction, that for every finite binary word $u$, the vertex $u^+$ in $\Asum$ is the vertex $a_{\sump(u)}$.

Lemma \ref{lem:stabilizers} implies the following. 

\begin{Lemma}\label{lem:dgfp}
	The diagram group $\ddd(\Asum)=\arr{F}_p$.
\end{Lemma}

\begin{proof} 
	Let $(T_+,T_-)$ be a reduced tree-diagram of some element $f\in F$. It suffices to prove that $(T_+,T_-)$ is accepted by $\Asum$ if and only if $f\in \Fp$. 
	
	Let $u_i\to v_i$, $i=1,\dots,n$ be the pairs of branches of $(T_+,T_-)$ and assume first that $(T_+,T_-)$ is accepted by $\Asum$. Then for each $i=1,\dots,n$, we have $u_i^+=v_i^+$ in $\Asum$. By the remarks preceding the lemma, that implies that for each $i=1,\dots,n$, we have $\sump(u_i)=\sump(v_i)$. We claim that $f\in \Fp$. By Lemma \ref{lem:stabilizers}, it suffices to prove that for each dyadic fraction $\alpha\in (0,1)$, the function  $f$ preserves the sum of digits modulo $p$ in the binary representation of $\alpha$.  
	 Let $\alpha\in(0,1)$ be a dyadic fraction. Then $\alpha$ has a finite binary representation of the form $.u_is$ for some $i\in\{1,\dots,n\}$ and some finite binary word $s$. Then $f(\alpha)=.v_is$. Since $\sump(u_i)=\sump(v_i)$, we have $\sump(f(\alpha))=\sump(.v_is)=\sump(.u_is)=\sump(\alpha)$, as necessary.

	In the other direction, assume that $f\in \Fp$. It suffices to prove that for each $i=1,\dots,n$, ${u_i}^+={v_i}^+$ in $\Asum$, or equivalently, that for each $i$, we have $\sump(u_i)=\sump(v_i)$. But for each $i$, we have $f(.u_i)=.v_i$. Hence, by Lemma \ref{lem:stabilizers}, $\sump(.u_i)=\sump(.v_i)$, as necessary. 
\end{proof}

Let $\mathcal P^{\mathrm{sum}}$ be the semigroup presentation associated with the rooted  tree-automaton $\Asum$. That is, 
$$\Psum=\la a_0,a_1,\dots,a_{p-1}\mid a_i=a_ia_{i+1}, i\in \{0,\dots,p-1\}\ra,$$ where $i+1$ is taken modulo $p$. The following is verified easily.

\begin{Lemma}
	 The semigroup $\Ssum$ represented by $\Psum$ is a left-zero semigroup of order $p$, such that $a_0,\dots,a_{p-1}$ are the distinct elements of $\Ssum$.
\end{Lemma}

Let $u$ be a finite binary word. We denote by $\mathrm{suf}_p(u)\in\{0,\dots,p-1\}$ the  length modulo $p$ of the longest suffix of $u$ where only ones appear.

\begin{Lemma}\label{sumsuf}
	Let $u$ and $v$ be paths in the automaton $\Asum$ and let $(p_u,q_u)$ and $(p_v,q_v)$ be the pairs of words associated with the paths $u$ and $v$ in $\Asum$, respectively. Let $\Ssum$ be the semigroup represented by  $\Psum$. Then the following assertions hold.  
	\begin{enumerate}
		\item If $u$ and $v$ contain the digit $1$, then $p_u=p_v$ in $\Ssum$. 
		\item If $u$ and $v$ contain the digit $0$, then $q_u=q_v$ in $\Ssum$ if and only if $$\sump(u)-\sufp(u)\equiv_p\sump(v)-\sufp(v)$$ 
	\end{enumerate}
\end{Lemma}

\begin{proof}
	Let $T_u$ (resp. $T_v$) be the minimal finite binary tree with branch $u$ (resp. $v$).
	
	(1) Assume that $u$ and $v$ contain the digit $1$.  In that case, $u$ (resp. $v$) is not the left-most branch of $T_u$ (resp. $T_v$). Hence, the word $p_u$ (resp. $p_v$) is not empty. Let $u_1$ (resp. $v_1$) be the first branch of $T_u$ (resp. $T_v$). Clearly, the words $u_1$ and $v_1$ only contain the digit $0$. Hence, $\sump(u_1)=\sump(v_1)=0$. Therefore, the paths $u_1$ and $v_1$ in $\Asum$ terminate in the vertex $u_1^+=a_0=v_1^+$. But the vertex $u_1^+$ (resp. $v_1^+$) is the first letter of $p_u$ (resp. $p_v$). Therefore, $a_0$ is the first letter of both  $p_u$ and $p_v$. Since $\Ssum$ is a left-zero semigroup, it follows that $p_u=a_0$ in $\Ssum$ and $p_v=a_0$ in $\Ssum$. Hence $p_u=p_v$ in $\Ssum$.

	$(2)$ Assume that $u$ and $v$ contain the digit $0$. Hence the words $q_u$ and $q_v$ are not empty. Let $u'$ and $v'$ be finite binary words such that $u\equiv u'01^{\sufp(u)}$ and $v\equiv v'01^{\sufp(v)}$ and note that the first branch in $T_u$ following the branch $u$ is the branch $u'1$. Hence, the first letter of $q_u$ is the vertex  $(u'1)^+$ of $\Asum$. Similarly, the vertex $(v'1)^+$ of $\Asum$ is the first letter of $q_v$. Since $\Ssum$ is a left-zero semigroup, we have $q_u=(u'1)^+=a_{\sump(u'1)}$ and  $q_v=(v'1)^+=a_{\sump(v'1)}$.
	Note that
	\begin{equation*}
	\begin{split}
\sump(u'1) &\equiv_p{\sump(u)-\sufp(u)+1},\\ 
\sump(v'1)&\equiv_p{\sump(v)-\sufp(v)+1}
	\end{split}
	\end{equation*} 
	Hence, $q_u=q_v$ in $\Ssum$ if and only if $$\sump(u)-\sufp(u)\equiv_p\sump(v)-\sufp(v)$$ 
\end{proof}

\begin{Lemma}\label{lem:coreFp1}
	Let $u$ and $v$ be finite binary words. 
	Then there is an element $h\in \arr{F}_p$ with the pair of branches $u\to v$ if and only if the following conditions hold. 
	\begin{enumerate}
		\item[$(1)$] Both $u$ and $v$ contain the digit $0$, or both $u$ and $v$ do not contain the digit $0$.  
		\item[$(2)$] Both $u$ and $v$ contain the digit $1$, or both $u$ and $v$ do not contain the digit $1$.
		\item[$(3)$] $\sump(u)=\sump(v)$.
		\item[$(4)$] $\sufp(u)=\sufp(v)$.  
	\end{enumerate}
\end{Lemma}

\begin{proof}
	Consider the words $u$ and $v$ as paths in the automaton $\Asum$ and let $(p_u,q_u)$ and $(p_v,q_v)$ be the associated pairs of words. Let $\Ssum$ be the semigroup represented by $\Psum$.

		Assume first that there is an element $h\in \arr{F}_p$ with the pair of branches $u\to v$. Then conditions $(1)$ and $(2)$ must hold (since they hold for every pair of branches of any element in $F$). Since $\arr{F}_p$ is the subgroup of $F$ accepted by $\Asum$, it follows from Lemma \ref{lem:pupv}, that $u^+=v^+$ in $\Asum$ and that $p_u=p_v$ and $q_u=q_v$ in $\Ssum$. Since $u^+=v^+$ in $\Asum$, we have $\sump(u)=\sump(v)$, so Condition $(3)$ holds. Finally, if $u$ and $v$  contain the digit $0$, then since $q_u=q_v$ in $\Ssum$ and $\sump(u)=\sump(v)$, Lemma \ref{sumsuf}$(2)$ implies that $\sufp(u)=\sufp(v)$. If both $u$ and $v$ do not contain the digit $0$, then for some $m,n\geq 0$, we have $u\equiv 1^m$ and $v\equiv 1^n$. Then $\sufp(u)=\sump(u)\equiv_p m$ and $\sufp(v)=\sump(v)\equiv_p n$. Since $\sump(u)=\sump(v)$, we have that $\sufp(u)=\sufp(v)$. Hence, Condition (4) holds as well.

	In the other direction, assume that Conditions $(1)$ to $(4)$ are satisfied and note that Condition $(3)$ implies that $u^+=v^+$ in $\Asum$.  We claim that $p_u=p_v$ and $q_u=q_v$ in the semigroup $\Ssum$. Then, by Lemma \ref{lem:pupv}, we would have that there is an element $h\in \ddd(\Asum)=\arr{F}_p$ with the pair of branches $u\to v$. 
	
	First, we show that $p_u=p_v$ in $\Ssum$. Indeed, if $u$ and $v$ both contain the digit $1$, then by Lemma \ref{sumsuf}$(1)$, $p_u=p_v$ in $\Ssum$. Otherwise, since $u$ and $v$ satisfy Condition $(2)$, both $u$ and $v$ do not contain the digit $1$. In that case, $u$ and $v$ are both (possibly empty) powers of $0$ and $p_u\equiv\emptyset\equiv p_v$.  
	
	Now, we prove that $q_u=q_v$ in $\Ssum$. If both $u$ and $v$ contain the digit $0$, then it follows immediately from Lemma \ref{sumsuf}$(2)$, since Conditions $(3)$ and $(4)$ hold for $u$ and $v$. Hence, assume that both $u$ and $v$ are (possibly empty) powers of $1$. In that case, $q_u\equiv q_v\equiv \emptyset$ and we are done. 
\end{proof}

Next, we want to consider the core of $\arr{F}_p$. First, we make the following simple observation. 

\begin{Lemma}
	Every finite binary word labels a path in the core $\mathcal C(\arr{F}_p)$. 
\end{Lemma}

\begin{proof}
	The lemma follows from the fact that $\Fp$ is closed for addition. Indeed, if $u$ labels a path in the core $\CFp$ then by Lemma  \ref{lem:minimal core}, there exists an element $f\in\Fp$ such that $u$ is a prefix of some branch of its reduced tree-diagram. But then $0u$ and $1u$ are prefixes of branches of the reduced tree-diagrams of $f\oplus 1$ and $1\oplus f$, respectively. 
	As such, by Lemma \ref{lem:minimal core}, they label paths in $\CFp$. 
\end{proof}

Recall that there are $4$ types of vertices in the core of a subgroup of $F$: the root, left vertices, right vertices and middle vertices.

\begin{Lemma}\label{lem:coreFp2}
	Let $u$ and $v$ be two finite binary words. Then $u^+=v^+$ in $\CFp$ if and only if the following assertions hold. 
	\begin{enumerate}
		\item $u^+$ and $v^+$ are vertices of the same type.
		\item $\sump(u)=\sump(v)$.
		\item $\sufp(u)=\sufp(v)$. 
	\end{enumerate}
\end{Lemma}

\begin{proof}
	By Lemma \ref{identified vertices in the core}, the vertices $u^+$ and $v^+$ coincide in $\CFp$ if and only if there is an element $h\in\Cl(\Fp)=\Fp$ which has the pair of branches $u\to v$. Hence, the result follows from Lemma \ref{lem:coreFp1} (note that Conditions $(1)$ and $(2)$ from Lemma \ref{lem:coreFp1} are equivalent to $u^+$ and $v^+$ being vertices of the same type). 
\end{proof}

Note that Lemma \ref{lem:coreFp2} gives a complete characterization of the core $\CFp$.

\begin{Corollary}\label{cor:core stru}
	The core $\mathcal C(\Fp)$ has $p^2+p+2$ vertices. It has a unqiue left vertex, $p$ distinct right vertices, $p^2$ distinct middle vertices and a root.
\end{Corollary}

\begin{proof}
	Clearly, the core has a unique root. We claim that there is a unique left vertex in the core. It suffices to show that for every $m,n\in\mathbb{N}$, the vertices $(0^m)^+$ and $(0^n)^+$ in $\CFp$ coincide. But, $\sump(0^n)=0=\sump(0^m)$ and $\sufp(0^n)=0=\sufp(0^m)$. Hence, by Lemma \ref{lem:coreFp2}, the left vertices $(0^n)^+$ and $(0^m)^+$  of the core coincide.

	Next, we claim that there are exactly $p$ distinct right vertices in the core. It suffices to prove that for every $m,n\in\mathbb{N}$, the vertices $(1^m)^+$ and $(1^n)^+$ in $\CFp$ coincide if and only if $m\equiv_p n$. Note that $\sump(1^m)=\sufp(1^m)\equiv_p m$ and $\sump(1^n)=\sufp(1^n)\equiv_p n$. Hence, by Lemma \ref{lem:coreFp2}, the right vertices $(1^m)^+$ and $(1^n)^+$ coincide in $\Fp$ if and only if $m\equiv_p n$. 
	
	Finally, we claim that there are exactly $p^2$ distinct middle vertices in the core. Let $u$ and $v$ be finite binary words which contain both digits $0$ and $1$. By Lemma \ref{lem:coreFp2} the middle vertices $u^+$ and $v^+$ coincide if and only if $\sump(u)=\sump(v)$ and $\sufp(u)=\sufp(v)$. Hence, there are exactly $p^2$ distinct middle vertices in the core (note that for each pair of $i,j\in\{0,\dots,p-1\}$, there is a finite binary word $w$ with $\sump(w)=i$ and $\sufp(w)=j$).
\end{proof}

To prove that the only subgroups of $F$ which strictly contain $\Fp$ are $F_{1,p}$ and $F$, we will  consider homomorphic images of the core $\CFp$. One such homomorphic image is the tree-automaton $\Asum$. Another one is the rooted tree-automaton $\Asuf$ defined as follows. 

\begin{Definition}
	The rooted tree-automaton $\Asuf$ (see Figure \ref{fig:Asuf}) is defined to be the rooted tree-automaton with vertices $b_0,\dots,b_{p-1}$ such that each vertex $b_i$ has two outgoing edges: a directed edge labeled $``0"$ to the vertex $b_0$ and a directed edge labeled $``1"$ from the vertex to $b_{i+1}$, where $i+1$ is taken modulo $p$. The root of $\Asuf$ is the vertex $b_0$. 
\end{Definition}

	\begin{figure}[h]
	\centering
	\includegraphics[width=.42\linewidth]{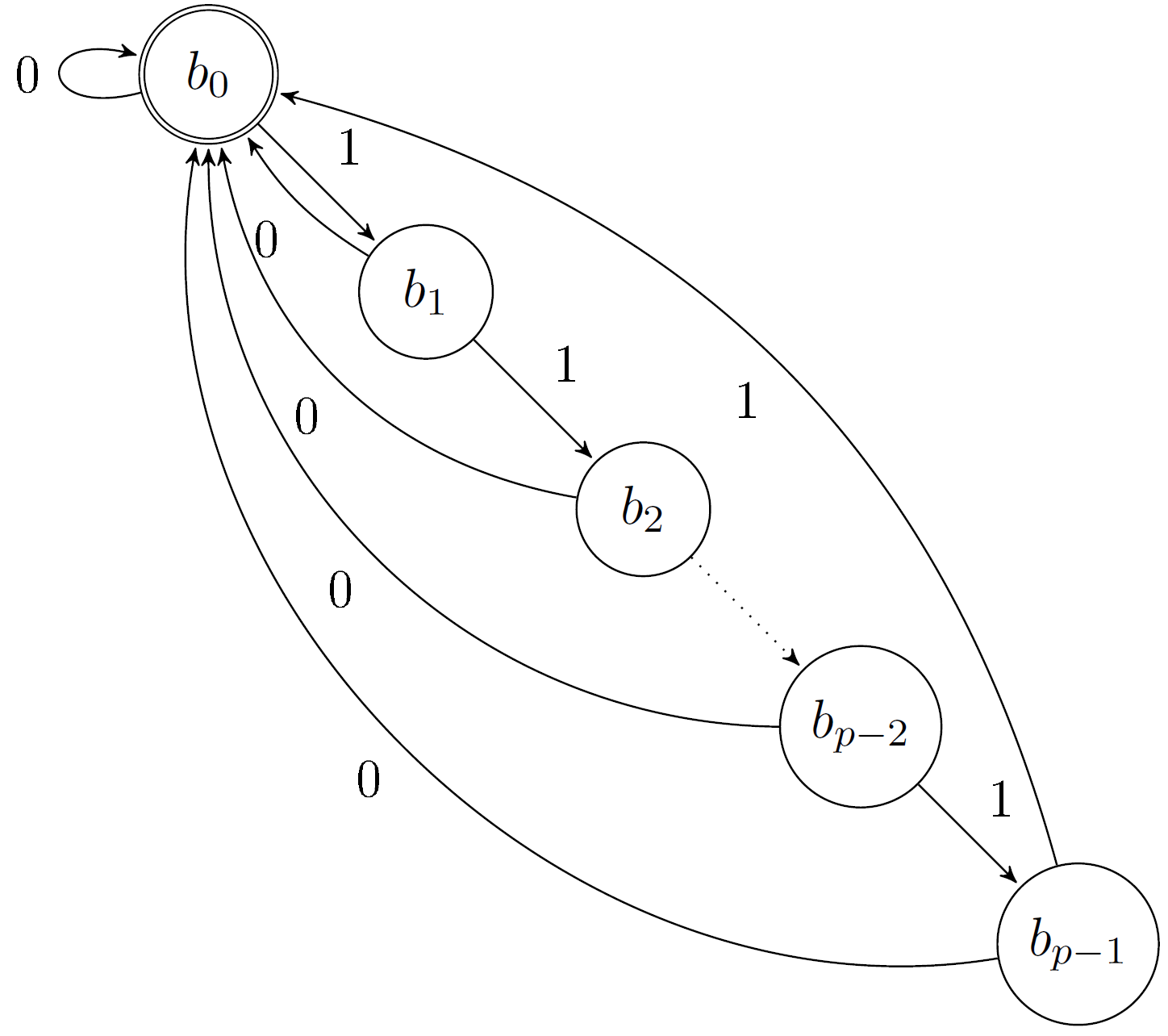}
	\caption{The rooted tree-automaton $\Asuf$.}\label{fig:Asuf}
\end{figure}

Note that $\Asuf$ is a full rooted tree-automaton. In particular, every finite binary word $u$ is readable on $\Asuf$. It is easy to prove by induction, that for every finite binary word $u$, the vertex $u^+$ in $\Asuf$ is the vertex $b_{\sufp(u)}$. It follows from Lemma \ref{lem:coreFp1}, that for every pair of branches $u\to v$ of a function in $\Fp$, we have $u^+=v^+$ in $\Asuf$. Hence, the subgroup $\Fp$ is accepted by $\Asuf$. In fact, below we prove that $\ddd(\Asuf)=\Fp$. To do so, we consider the semigroup presentation  $\mathcal P^{\mathrm{suf}}$ associated with the rooted  tree-automaton $\Asuf$. Note that 
$$\Psuf=\la b_0,b_1,\dots,b_{p-1}\mid b_i=b_0b_{i+1}, i\in \{0,\dots,p-1\}\ra,$$ where $i+1$ is taken modulo $p$. Note that the semigroup  $\Ssuf$ represented by $\Psuf$ is a cyclic group of order $p$, where $b_0,\dots,b_{p-1}$ are its distinct elements ($b_1$ is the identity element of the group).

\begin{Lemma}\label{pu}
	Let $u$ be a path in the rooted tree-automaton $\Asuf$ and let $(p_u,q_u)$ be the pair of words associated with the path $u$ in $\Asuf$. Let $\Ssuf$ be the semigroup represented by $\Psuf$. Assume that $u$ contains the digit $1$, then $p_u=b_0^{{\sump(u)}}$ in $\Ssuf$. 
\end{Lemma}

\begin{proof}
	Consider the structure of the minimal tree $T_u$ with branch $u$. The number of branches in $T_u$ to the left of the branch $u$ is equal to the number of digits $1$ in $u$ (i.e., to the sum of digits in $u$). Each branch $v$ of $T_u$ to the left of the branch $u$ terminates with the digit $0$ and therefore satisfies $v^+=b_0$ in $\Asuf$. Hence, the word $p_u\equiv b_0^{s(u)}$, where $s(u)$ is the sum of digits in $u$. Since $\Ssuf$ is a group of order $p$ generated by $b_0$ and $\sump(u)\equiv_p s(u)$, we have $p_u=b_0^{\sump(u)}$ in $\Ssuf$. 
\end{proof}

\begin{Lemma}\label{dgsuf}
	The diagram group $\ddd(\Asuf)=\Fp$. 
\end{Lemma}

\begin{proof}
	As noted above, $\Fp$ is contained in $\ddd(\Asuf)$. In the other direction, let $(T_+,T_-)$ be a tree-diagram in $\ddd(\Asuf)$. We claim that $(T_+,T_-)$ belongs to $\Fp$. Let $u\to v$ be a pair of branches of $(T_+,T_-)$. It suffices to show that $u$ and $v$ satisfy Conditions $(1)-(3)$ from Lemma \ref{lem:coreFp2}. Indeed, that would imply that  $(T_+,T_-)$ is accepted by the core of $\Fp$ and as such, belongs to the closed subgroup $\Fp$. 
	
	Since $u\to v$ is a pair of branches of the tree-diagram $(T_+,T_-)$, the vertices $u^+$ and $v^+$ of $\mathcal C(\Fp)$ must be vertices of the same type. Let $(p_u,q_u)$ and $(p_v,q_v)$ be the pairs of words associated with the paths $u$ and $v$ in $\Asuf$. Since $(T_+,T_-)$ is accepted by $\Asuf$, it follows from Lemma \ref{lem: morphism from core} that $u^+=v^+$ in $\Asuf$ and that $p_u=p_v$ in $\Ssuf$. Now, the fact that $u^+=v^+$ in $\Asuf$ implies that $\sufp(u)=\sufp(v)$. By Lemma \ref{pu}, the fact that $p_u=p_v$ 
	 implies that $\sump(u)=\sump(v)$ (if $u$ and $v$ do not contain the digit $1$ then  $\sump(u)=\sump(v)=0$). Hence, $u$ and $v$ satisfy Conditions $(1)-(3)$ from Lemma \ref{lem:coreFp2}, as necessary.
\end{proof}

\begin{Remark}
	Lemma \ref{dgsuf} can also be derived from \cite[Section 5.3]{GS}, where the isomorphism from Thompson's group  $F_{n+1}$ to Jones' subgroup $\arr{F}_n$ is described explicitly. 
\end{Remark}

\begin{Proposition}\label{prop:main2}
	Let $G$ be a subgroup of $F$ which strictly contains $\Fp$. Then there is a unique middle vertex in the core of $G$. 
\end{Proposition}

\begin{proof}
	Let us consider the automaton $\mathcal C(\Fp)$. It has $p^2$ middle vertices, which can be referred to as $x_{i,j}$, $i,j\in\{0,\dots,p-1\}$, where $x_{i,j}$ is the end vertex of all paths $u$ in $\mathcal C(\Fp)$ such that $\sump(u)=i$ and $\sufp(v)=j$.
	Since $G$ contains $\Fp$, by Corollary \ref{cor:no-leaves}, there is a surjective morphism $\phi$ from $\mathcal C(\Fp)$ to $\C(G)$.
	Note that for each finite binary word $u$ the morphism $\phi$ maps the vertex $u^+$ of $\mathcal C(\Fp)$ onto the vertex $u^+$ of $\C(G)$.
		It suffices to prove that for every $(i,j)\in\{0,\dots,p-1\}\times\{0,\dots,p-1\}$, we have $$\phi(x_{i,j})=\phi(x_{0,0}).$$ 
	We do it in several steps. In all of the steps below, when we write $x_{i,j}$ for some integers $i,j$, the integers are taken modulo $p$. 
	
	\textbf{Step 1:} There exist some $i_1< i_2$ in $\{0,\dots,p-1\}$ such  that $\phi(x_{i_1,0})=\phi(x_{i_2,0})$. 
	
	Indeed, since $G$ strictly contains $\Fp$, there is an element $g\in G\setminus\Fp$. Since $g\notin \Fp$, by Lemma \ref{lem:dgfp}, it is not accepted by the automaton $\Asum$. Hence, $g$ has a pair of branches $u\to v$ such that $\sump(u)\neq \sump(v)$. Let $i_1=\sump(u)$ and $i_2=\sump(v)$. We can assume that $i_1<i_2$, by replacing $g$ by $g^{-1}$ if necessary.  Note that $u0\to v0$ is also a pair of branches of $g$ and
	 that in $\mathcal C(\Fp)$ we have $(u0)^+=x_{i_1,0}$ and $(v0)^+=x_{i_2,0}$. Since $u0\to v0$ is a pair of branches of $g\in G$ we must have $(u0)^+=(v0)^+$ in $\C(G)$. Hence, 
	$$\phi(x_{i_1,0})=\phi(x_{i_2,0}),$$
	as required.  
	
	\textbf{Step 2:} For every $j,k\in\{0,\dots,p-1\}$ we have $\phi(x_{j,k})=\phi(x_{j+(i_2-i_1),k})$
	
	Indeed, recall that $g$ has the pair of branches $u\to v$ such that $i_1=\sump(u)$ and $i_2=\sump(v)$. Let $j,k\in\{0,\dots,p-1\}$ and note that $g$ has the pair of branches $u_1\to v_1$, where
	$$u_1\equiv u1^{2p-i_1-k+j}01^k\mbox{ and } v_1\equiv v1^{2p-i_1-k+j}01^k.$$
	Note also that $\sump(u_1)=j$,  $\sump(v_1)\equiv_p i_2-i_1+j$ and $\sufp(u_1)=\sufp(v_1)=k$. As in Step 1, the pair of branches $u_1\to v_1$ of $g$ implies that in $\mathcal C(G)$, we have $\phi(x_{j,k})=\phi(x_{j+(i_2-i_1),k})$, as necessary. 
	
	\textbf{Step 3:} For every $j,k\in\{0,\dots,p-1\}$ we have $\phi(x_{j,k})=\phi(x_{0,k})$
	
	Indeed, let $\ell=i_2-i_1$ and note that $\ell\in\{1,\dots,p-1\}$ is co-prime to $p$. Let $d\in\{1,\dots,p-1\}$ be such that $d\ell\equiv_p 1$ and let $k\in\{0,\dots,p-1\}$. 
	By Step 2, we have that for all $j$, $$\phi(x_{j,k})=\phi(x_{j+\ell,k})$$ 
	It follows that for every $j$, we have $$\phi(x_{j,k})=\phi(x_{j+(p-j)d\ell,k})$$
	Hence, for every $j$, we have 
	$$\phi(x_{j,k})=\phi(x_{0,k}).$$

	\textbf{Step 4:} There exists $\ell_1\in\{1,\dots,p-1\}$ such that for every $k\geq 0$ we have 
	 $\phi(x_{0,\ell_1+k})=\phi(x_{0,k})$.
	
	Indeed, since $g\notin \Fp$, by Lemma \ref{dgsuf}, the element $g$ is not accepted by $\Asuf$. Hence, it must have a pair of branches $w_1\to w_2$ such that  $\sufp(w_1)\neq \sufp(w_2)$. Let $m_1=\sufp(w_1)$ and $m_2=\sufp(w_2)$. We can assume without loss of generality that $m_1<m_2$ and let $\ell_1=m_2-m_1$. Now, let $k\geq 0$ and note that $g$ also has the pair of branches $w_11^{p-m_1+k}\to w_21^{p-m_1+k}$. Hence, in $\mathcal C(G)$, we have $(w_11^{p-m_1+k})^+= (w_21^{p-m_1+k})^+$.
	Note also that $\sufp(w_11^{p-m_1+k})\equiv_p k$ and $\sufp(w_2^{p-m_1+k})\equiv_p m_2-m_1+k\equiv_p \ell_1+k$. 
	 Let $j_1=\sump(w_11^{p-m_1+k})$ and $j_2=\sump(w_21^{p-m_1+k})$. Then since  $(w_11^{p-m_1+k})^+= (w_21^{p-m_1+k})^+$ in $\C(G)$, we have  
		$$\phi(x_{j_1,k})=\phi(x_{j_2,\ell_1+k}).$$
		Now, 
		 by Step 3 and the last equation we have 
		$$\phi(x_{0,k})=\phi(x_{j_1,k})=\phi(x_{j_2,\ell_1+k})=\phi(x_{0,\ell_1+k}),$$
		as necessary. 
		
		\textbf{Step 5:} For every $k\in\{0,\dots,p-1\}$, we have $\phi(x_{0,k})=\phi(x_{0,0})$.
		
		Indeed, since $\ell_1\in\{1,\dots,p-1\}$, it is co-prime to $p$. Let $d_1\in\{1,\dots,p-1\}$ be such that $d_1\ell_1\equiv_p 1$.  By step 5, we have that for all $k$, 
		$$\phi(x_{0,k})=\phi(x_{0,k+\ell_1}).$$
		It follows that for every $k$, we have 
		$$\phi(x_{0,k})=\phi(x_{0,k+(p-k)d_1\ell_1}).$$
		Hence, for every $k$, we have 
		$$\phi(x_{0,k})=\phi(x_{0,0}).$$
		
		\textbf{Step 6:} For every $j,k\in\{0,\dots,p-1\}$, we have $\phi(x_{j,k})=\phi(x_{0,0})$.
		
		Indeed, that follows immediately from Steps 3 and 5. 
		
		\noindent Hence, there is a unique middle vertex in the core of $G$, as required. 
\end{proof}

\begin{Theorem}\label{thm: main2}
	The only subgroups of Thompson's group $F$ which strictly contain $\Fp$ are $F_{1,p}$ and $F$. 
\end{Theorem}

\begin{proof}
	First, we note that the image of $\Fp$ in the abelianization of $F$ is  $\mathbb{Z}\times p\mathbb{Z}$.
	One can verify it, for example, by computing the image in the abelianization of the generating set of $\Fp$ from Lemma \ref{lem:closed for addition}.

	Now, let $G$ be a subgroup of $F$ which strictly contains $H$. Then by Proposition \ref{prop:main2}, the core of $G$ has a unique middle vertex. Clearly, this vertex has two outgoing edges.  Hence, by Lemma \ref{core of derived subgroup}, $\Cl(G)$ contains the derived subgroup of $F$. 
	 In addition, since $H\leq G$, it follows that $\pi_{\ab}(H)\leq \pi_{\ab}(G)$. Since $\pi_{\ab}(H)=\mathbb{Z}\times p\mathbb{Z}$ is a maximal subgroup of $\mathbb{Z}^2$, there are two options for $\pi_{\ab}(G)$: either $\pi_{\ab}(G)=\mathbb{Z}\times p\mathbb{Z}$ or $\pi_{\ab}(G)=\mathbb{Z}^2$. Note that in either case, $\pi_{\ab}(G)$ is a closed subgroup of $\mathbb{Z}^2$. Hence, by Theorem \ref{main}, in either case, $G=G[F,F]=\pi_{\ab}^{-1}(G)$. Hence, either $G=F_{1,p}$ or $G=F$, as required. 
\end{proof}

It follows from Theorem \ref{thm: main2} that $\Fp$ is a maximal subgroup of $F_{1,p}$. Since $F_{1,p}$ is isomorphic to $F$, we have the following. 

\begin{Corollary}\label{cor:max Jones}
	For every prime number $p$, Thompson's group $F$ has a maximal subgroup isomorphic to Jones' subgroup $\Fp$. 
\end{Corollary}

\section{Final Remarks and Open Problems}\label{sec:open}

Let $p$ be a prime number, For each $i=0,\dots,p-1$, let $S_i$ be the set of all dyadic fractions such that the sum of digits in their finite binary representation is $i$ modulo $p$. 

\begin{Remark}
	It follows from Theorem \ref{thm: main2} that for each $i=0,\dots,p-1$, Jones' subgroup $\Fp=\Stab(S_i)$. Indeed, by Lemma \ref{lem:stabilizers}, for each $i=0,\dots,p-1$, Jones' subgroup $\Fp$ is contained in $\Stab(S_i)$ and clearly, $\Stab(S_i)\notin\{ F_{1,p}, F\}$.
\end{Remark}

\begin{Lemma}
	The number of orbits of the action of $\Fp$ on the set of dyadic fractions $\mathcal D$ is $p$. 
\end{Lemma}

\begin{proof}
 As $\mathcal D$ is the disjoint union of the sets $S_i$ for $i=0,\dots,p-1$, it suffices to prove that for each $i$, the set $S_i$ is an orbit of the action of $\Fp$ on $\mathcal D$. Let $i\in\{0,\dots,p-1\}$ and let $\alpha\in S_i$. By Lemma \ref{lem:stabilizers}, the orbit of $\alpha$ is contained in $S_i$. Hence, it suffices to prove that if $\beta\in S_i$ then $\beta$ is in the orbit of $\alpha$.
 Let $u$ and $v$ be finite binary words such that $\alpha=.u$ and $\beta=.v$ and such that the last digit of both $u$ and $v$ is zero. We claim that there is an element in $\Fp$ with the pair of branches $u\to v$. To prove that it suffices to show that $u$ and $v$ satisfy Conditions $(1)-(4)$ from Lemma \ref{lem:coreFp1}. First, note that since $\alpha$ and $\beta$ are in $\mathcal D\subseteq(0,1)$, the finite binary words $u$ and $v$ must also contain the digit $1$. Hence, Conditions (1) and (2) from the lemma are satisfied.  
 In addition, since $\alpha$ and $\beta$ belong to $S_i$, we have $\sump(u)=i=\sump(v)$. Finally,  $\sufp(u)=0=\sufp(v)$. Hence Conditions $(3)$ and $(4)$ of Lemma \ref{lem:coreFp1} also hold and there is an element $h\in \Fp$ with the pair of branches $u\to v$. In particular, $h(\alpha)=h(.u)=.v=\beta$. Hence, $\beta$ is in the orbit of $\alpha$. 
\end{proof}

Let $\nu\colon F_{1,p}\to F$ be an isomorphism. Then $\nu(\arr{F}_p)$ is a maximal subgroup of $F$, isomorphic to $\Fp$. We claim that the action of $\nu(\arr{F}_p)$ on the set of dyadic fractions $\mathcal D$ also has exactly $p$ orbits. Indeed, since the action of $F_{1,p}$ on the interval $(0,1)$ is locally dense (see, for example \cite[Lemma 7.2]{BW}), Rubin's theorem (see \cite[Section 9]{B}) implies that 
there exists a homeomorphism $\phi\colon(0,1)\to(0,1)$ such that $\nu(f)=\phi^{-1}f\phi$ for every $f\in F_{1,p}$. The homeomorphism $\phi$ must map the set of dyadic fractions $\mathcal D$ onto itself (indeed, that follows from consideration of the groups of germs of $F_{1,p}$ and $F$ (see, for example \cite{L}) at dyadic fractions, rational non-dyadic fractions and irrational numbers in $(0,1)$).
It follows that $\phi(\mathcal D)=\mathcal D$ is the disjoint union of the sets $\phi(S_i)$, $i=0,\dots,p-1$. But these  sets are orbits of the action of $\nu(\arr{F}_p)$ on the interval $(0,1)$. Hence, the action of $\nu(\arr{F}_p)$ on $\mathcal D$ has exactly $p$ orbits. 

\begin{Corollary}
	For every prime number $p$, Thompson's group $F$ has a maximal subgroup whose action on the set of dyadic fractions $\mathcal D$ has exactly $p$ orbits. 
\end{Corollary} 

We believe the answer to the following problem is positive. 

\begin{Problem}\label{p1}
	Is it true that for every $n\in\mathbb{N}$, Thompson's group $F$ has a maximal subgroup whose action on the set of dyadic fractions $\mathcal D$ has exactly $n$ orbits? 
\end{Problem}

\begin{Remark}[Added in revision]
	Problem \ref{p1} was recently solved in the affirmative by the author \cite{G23}. 
\end{Remark}

Corollary \ref{cor:max Jones} shows that there are at least countably many distinct isomorphism classes of maximal subgroups of infinite index in Thompson's group $F$. However, the following problem remains open. 

\begin{Problem}\label{pro:uncountable}
	Are there uncountably many distinct isomorphism classes of maximal subgroups of Thompson's group $F$?
\end{Problem}

Clearly, in order to answer Problem \ref{pro:uncountable}, one has to consider infinitely generated maximal subgroups of Thompson's group $F$. Recall that $\Stab(\alpha)$ for $\alpha\in (0,1)$ is not finitely generated if and only if $\alpha$ is irrational (see \cite{GS2}). Hence, Thompson's group $F$ has maximal subgroups which are not finitely generated. However, the stabilizers $\Stab(\alpha)$ for $\alpha\in (0,1)\setminus\mathbb{Q}$ are all isomorphic \cite{GS2}. Hence, up to isomorphism, there is only one known maximal subgroup of $F$ which is not finitely generated. 

\begin{Problem}
	Is there a maximal subgroup of $F$ which is not finitely generated and does not fix any number in $(0,1)$? 
\end{Problem}

Note that for every $n\in\mathbb{N}$, the rank of Jones' subgroup $\arr{F}_n\cong F_{n+1}$ is $n+1$ (see \cite{GuSa97}). 
Hence, Thompson's group $F$ has finitely generated maximal subgroups of arbitrarily large rank.  Note also that all known finitely generated maximal subgroups of Thompson's group $F$ are finitely presented. Indeed, all brother groups $F_n$ of Thompson's group $F$ are finitely presented \cite{GuSa97}. For every $\alpha\in\mathbb{Q}$, the stabilizer $\Stab(\alpha)$ is finitely presented \cite{GS2}. Using results for diagram groups from \cite[Section 9]{GuSa97}, one can also prove that the maximal subgroups from \cite{G16}, from \cite{AN} and from Example \ref{exm} above are all finitely presented (the algorithm from \cite[Section 9]{GuSa97} can also be used to find explicit finite presentations for these subgroups).

\begin{Problem}
	Are all finitely generated maximal subgroups of Thompson's group $F$ finitely presented? 
\end{Problem}

\begin{minipage}{3 in}
	Gili Golan\\
	Department of Mathematics,\\
	Ben Gurion University of the Negev,\\ 
	golangi@bgu.ac.il
\end{minipage}

\end{document}